%% file: nepers.tex
\documentclass{article}
\usepackage{amsmath,amssymb}
\usepackage{amsthm}
\usepackage{stmaryrd}  
\usepackage{graphicx}  

\DeclareFontFamily{U}  {MnSymbolC}{}
\DeclareSymbolFont{MnSyC}         {U}  {MnSymbolC}{m}{n}
\DeclareFontShape{U}{MnSymbolC}{m}{n}{
    <-6>  MnSymbolC5
   <6-7>  MnSymbolC6
   <7-8>  MnSymbolC7
   <8-9>  MnSymbolC8
   <9-10> MnSymbolC9
  <10-12> MnSymbolC10
  <12->   MnSymbolC12}{}
\DeclareSymbolFont{MnSyC} {U} {MnSymbolC}{m}{n}
\let\hdots\undefined
\DeclareMathSymbol{\hdots}{\mathord}{MnSyC}{5}

\usepackage[protrusion=true,expansion=true]{microtype}

\usepackage{tocloft}  
\usepackage[usenames,dvipsnames,svgnames,table]{xcolor}
\usepackage[unicode=true,pdfusetitle,bookmarks=true,bookmarksnumbered=false,bookmarksopen=false,breaklinks=false,pdfborder={0 0 0},pdfborderstyle={},backref=false,colorlinks=true]{hyperref} 
\hypersetup{colorlinks, linkcolor=OliveGreen, citecolor=OliveGreen, urlcolor=OliveGreen}
\usepackage{cleveref}

\newcommand{\Gm}{{\mathbb{G}_m}}

\newcommand{\CP}{\mathbb{CP}}
\newcommand{\Z}{\mathbb{Z}}
\newcommand{\C}{\mathbb{C}}
\newcommand{\Q}{\mathbb{Q}}
\newcommand{\E}{\mathbb{E}}
\newcommand{\h}{\mathfrak{h}}
\newcommand{\T}{\mathbb{T}}

\newcommand{\KU}{\mathrm{KU}}
\newcommand{\MU}{\mathrm{MU}}
\newcommand{\BU}{\mathrm{BU}}
\newcommand{\bu}{\mathrm{bu}}
\newcommand{\Tate}{\mathrm{Tate}}

\renewcommand{\phi}{\varphi}
\newcommand{\Td}{\mathrm{Td}}
\newcommand{\Wit}{\mathrm{Wit}}
\newcommand{\SL}{\mathrm{SL}}

\newcommand{\co}{\colon\thinspace}

\renewcommand{\th}{\textsuperscript{th}}

\newcommand{\<}{\langle}
\renewcommand{\>}{\rangle}
\renewcommand{\#}{\sharp}
\newcommand{\ls}[1]{(\!({#1})\!)}
\newcommand{\ps}[1]{\llbracket{#1}\rrbracket}

\theoremstyle{plain}
\newtheorem{theorem}{Theorem}[section]
\newtheorem{maintheorem}{Main Theorem}

\newtheorem{lemma}[theorem]{Lemma}

\theoremstyle{definition}
\newtheorem{definition}[theorem]{Definition}
\theoremstyle{remark}
\newtheorem{remark}[theorem]{Remark}
\newtheorem{example}[theorem]{Example}

\title{$p$--adic congruences in iterated derivatives of $\wp$}
\date{}
\author{Kiran Luecke, Eric Peterson}

\begin{document}

\maketitle

\vspace{-1.5\baselineskip}
\begin{abstract}
    We use homotopy theoretic methods to prove congruence relations of number theoretic interest.
    Specifically, we use the theory of $\E_\infty$ complex orientations to establish $p$--adic K\"ummer congruences among iterated derivatives of the Weierstrass elliptic function.
    
    The machinery of \cite{AHR} was developed by Ando, Hopkins, and Rezk with the intended application of taking congruence relations as input and producing $\E_\infty$--orientations as output.
    We run their machine in reverse, using as input the recent results of \cite{CarmeliLuecke} on the existence of $\E_\infty$--orientations of Tate fixed-point objects.
\end{abstract}

\begin{center}
\resizebox{0.80\hsize}{!}{\input{todd.2.data.truncated.tex}}
\end{center}

\tableofcontents
\section{Introduction}

\begin{definition}[{\cite[Sections 3.1--2]{Katz}}]\label{defn_moment_sequence}
    For $A$ a torsion-free $\Z_p$-algebra, an \emph{$A$--valued $p$--adic moment sequence} is a sequence $M_n \in A$ such that for any rational polynomial $f(r) = \sum_n a_n r^n \in \Q[r]$ with $f(\Z_p^\times) \subseteq \Z_p$, the sum $\sum_n a_n M_n \in A \otimes \Q$ is an $A$--integer.
\end{definition}

Moment sequences arise out of $p$--adic integrals via the correspondence
\begin{align*}
    \int_{\Z_p^\times} dM \, \sum_{n=0}^N a_n r^n & := \sum_{n=0}^N a_n M_n, &
    M_n & := \int_{\Z_p^\times} dM \, r^n.
\end{align*}
The integrality condition can be interpreted as a family of \emph{K\"ummer congruences} on the sequence elements.
For example, the polynomial $p^{-1}(r^p - r)$ is $p$--adically integer-valued, hence $p^{-1}(M_p - M_1)$ must be an $A$--integer, or equivalently $M_p \equiv M_1 \pmod{p}$.
Other choices of test polynomial yield other congruences.

We prove the following:
\begin{maintheorem}\label{MainTheoremWitten}
    Pick any $c \in \Z_p^\times \setminus \{\pm 1\}$.
    There is a $\Z\ps{q}\ls{t}^\wedge_p$--valued moment sequence whose terms for $n \ge 1$ are given by
    \begin{align*}
    M_1^{\Wit^\#} & = (1 - c) \cdot \frac{\log(1 - t)}{(2 \pi i)^2} \cdot (G_2(q) - p \cdot G_2(q^p)) \\
    & \quad + \frac{1 - c}{2 \pi i} \cdot \left(\wp_q^{(-1)}\left(\frac{\log(1 - t)}{- 2 \pi i}\right) - \wp_{q^p}^{(-1)}\left(\frac{p \log(1 - t)}{- 2 \pi i}\right)\right), \\
    M_2^{\Wit^\#} & = (1 - c^2) \cdot \frac{G_2(q) - p \cdot G_2(q^p)}{-2 (2 \pi i)^2} \\
    & \quad + (1 - c^2) \cdot \left(\wp_q \left(\frac{\log(1 - t)}{-2 \pi i}\right) - \wp_{q^p} \left(\frac{p \log(1 - t)}{-2 \pi i}\right) \right), \\
    M_{n \ge 3}^{\Wit^\#} & = \frac{1 - c^n}{(2 \pi i)^n} \cdot \left( \wp_q^{(n-2)} \left(\frac{\log(1 - t)}{-2 \pi i}\right) - \wp_{q^p}^{(n-2)} \left(\frac{p \log(1 - t)}{-2 \pi i}\right) \right),
    \qedhere
    \end{align*}
    where $\wp_q$ is the $q$--expansion of the Weierstrass elliptic function, $\wp_q^{(n)}$ is its $n${\th} derivative, and $G_2$ is the second Eisenstein series.
\end{maintheorem}

\noindent
As warm-up, we also prove the following theorem:
\begin{maintheorem}\label{MainTheoremTodd}
    Pick any $c \in \Z_p^\times \setminus \{\pm 1\}$.
    There is a $\Z\ls{t}^\wedge_p$--valued moment sequence whose terms for $n \ge 1$ are given by
    \[(1 - c^n)\left(\sum_{m=1}^n \frac{\Delta^m[r^n](0)}{m} \left( (t^{-1})^m - p^{n-1} t^{-pm} (1 - p \zeta)^{-m} \right) \right),\]
    where $\Delta^m[r^n](0)$ is the evaluation at zero of $m${\th} forward finite difference of the test polynomial $r^n$, and where $\zeta \in \Z[t^{-1}]$ is defined by $(t^{-1} - 1)^p = t^{-p} - 1 + p \zeta$.
\end{maintheorem}

\begin{remark}
The $p$ in front of $\zeta$ yields a convergent sum in $\Z\ls{t}^\wedge_p$, despite the unboundedly increasing powers of $t^{-1}$.
\end{remark}

\begin{remark}
\Cref{MainTheoremTodd} is essentially the component of \Cref{MainTheoremWitten} induced by the pole of $\wp^{(n-2)}$.
More exactly:
\begin{itemize}
    \item The two Main Theorems' terms of nonpositive $t$--degree match exactly.
    \item The two Main Theorems' terms of vanishing $t$-- and $q$--degree are exactly the moments of the Mazur measure~\cite[Section 4.3]{Lang}.
    \item The terms of \Cref{MainTheoremWitten} of vanishing $t$--degree are exactly those of the $MF_{*, *}$--valued measure constructed by Katz~\cite[Lemma 3.5.6]{Katz}, appearing also in Ando, Hopkins, and Rezk~\cite[Section 10]{AHR}.
\end{itemize}
\end{remark}

\begin{remark}
See \Cref{ToddTableAt3} and \Cref{WittenTableAt3} for visualizations of these congruences at $p = 3$, and \Cref{ToddTableAt2} and \Cref{WittenTableAt2} (which extends the title figure) for $p = 2$.
\end{remark}

\subsection*{Acknowledgments}
We would like to thank Charles Rezk for helpful conversations,
Jonathan Beardsley for helpful comments on a draft, and  
Matthew Ando, Shachar Carmeli, and Tomer Schlank for their words of encouragement.
The second author would like to thank BART, where much of this work was carried out.

\section{R\'esum\'es}

Our main theorems are largely an application of existing work, especially a recent paper of the first author~\cite{CarmeliLuecke}.
Accordingly, we recall below the inputs which we require to complete our calculation.

\subsection{R\'esum\'e on elliptic functions}

It will be convenient to have formulas on hand for a variety of standard functions appearing in the theory of complex elliptic curves.
To fix terms, let us consider (twisted) functions on the compactified moduli of elliptic curves, referred to as \emph{modular forms}.
Such functions are typically presented in three ways~\cite[Section I.1, Remark 3.3]{Silverman}:
\begin{enumerate}
    \item A modular form $f$ of weight $k$ is a function on choices of lattice $\Lambda \subset \C$ with the property $f(c \cdot \Lambda) = c^{-k} f(\Lambda)$ and which extends to a meromorphic function at $\operatorname{span} \{+i \cdot \infty, 1\}$.
    \item Selecting a basis $\{\tau, 1\}$ for the (possibly positively rescaled) lattice $\Lambda$, $f$ can be understood as a function of $\tau \in \h$ satisfying $f(M \cdot \tau) = (c \tau + d)^{-k} f(\tau)$ and extending to a meromorphic function at $\tau = +i \cdot \infty$.
    \item Using periodicity against $\SL_2 \Z$ and meromorphicity at $\tau = + i \cdot \infty$, we may write the Fourier series of $f$ as a bounded-below sum $\sum_{n=-m}^\infty a_n q^n$, where $q = e^{2 \pi i \tau}$, called the \emph{$q$--expansion}.%
    \footnote{However, it is not trivial to discern when an arbitrary $q$--series is a modular form.}
\end{enumerate}
We will also be interested in functions on the universal elliptic curve, called \emph{Jacobi forms}~\cite[Section I.1]{EichlerZagier}, modeled as functions on $\C \times \h$ equivariant for $\SL_2 \Z \ltimes \Z^2$.
Such functions again can be parameterized in terms of $\Lambda$, $\tau$, or $q$.

\begin{example}\label{EisensteinSeries}\label{DivisorSumFunction}
For $k \ge 2$, the $k${\th} \emph{Eisenstein series} $G_k$ is the (quasi\footnote{$G_2(\tau)$ does not quite transform as a modular form, but instead acquires an extra summand.})modular form defined by the lattice sum~\cite[Equation 2.2.5]{Katz} \[G_k(\Lambda) = \sum_{\lambda \in \Lambda^*} \frac{1}{\lambda^k}.\]
It is evidently of weight $k$.
Note that if $k$ is odd, the sum telescopes to zero.
In the even case, the $q$--expansion of $G_{2k}$ is \[\frac{(2k-1)!}{(2 \pi i)^{2k}} \cdot G_{2k}(q) = -\frac{B_{2k}}{2k} + 2 \cdot \sum_{j=1}^\infty \sigma_{2k-1}(j) q^j,\] where $B_{2k}$ is the $2k${\th} Bernoulli number, defined by $\sum_{n=0}^\infty B_n x^n = \frac{x}{e^x - 1}$, and where $\sigma_{2k-1}(j) = \sum_{d \mid j} d^{2k-1}$ is the divisor-sum function~\cite[Equation 2.4.8]{Katz}.
\end{example}

\begin{example}\label{SigmaFunction}
The \emph{Weierstrass $\sigma$--function} is defined by~\cite[Proposition 5.4.a]{Silverman} \[\sigma_\Lambda(z) = z \cdot \prod_{\lambda \in \Lambda^*} \left[ \left(1 - \frac{z}{\lambda} \right) \cdot \exp\left(\frac{z}{\lambda} + \frac{1}{2} \left( \frac{z}{\lambda} \right)^2 \right) \right].\]
Among other properties, it is designed so that its logarithm becomes a generating function for Eisenstein series~\cite[Equation 2.2.7]{Katz}: \[\log z - \log \sigma_\Lambda(z) = \sum_{n=3}^\infty \frac{(2 \pi i z)^n}{n!} \cdot \left( \frac{(n-1)!}{(2 \pi i)^n} \cdot G_n(\Lambda) \right).\]
We will need a product formula for its $q$--expansion~\cite[Theorems 6.4, 6.3.b]{Silverman}: \[2 \pi i \cdot \sigma_q(z) = \exp\left(\frac{1}{2} (2 \pi i \cdot z + G_2(q) \cdot z^2) \right) \cdot (1-u) \cdot \prod_{j=1}^\infty \frac{(1 - q^j u)(1 - q^j u^{-1})}{(1 - q^j)^2},\] where $u = e^{-2 \pi i z}$.
Lastly, the second logarithmic derivative of $\sigma$ is named the \emph{Weierstrass $\wp$--function}~\cite[Proposition 5.4.c]{Silverman}: \[\wp_\Lambda(z) = -\frac{d^2}{dz^2} \log \sigma_\Lambda(z),\]
which is a meromorphic Jacobi form of index $0$ and weight $2$.
\end{example}

\subsection{R\'esum\'e on elliptic cohomology}
In this subsection we recall some basic elements of complex orientation theory.

\begin{definition}\label{defn_cx_or}
    Let $R$ be a homotopy ring spectrum.
    The following two pieces of data are equivalent:
    \begin{enumerate}
        \item A \emph{complex orientation} of $R$ is a homotopy ring map $\MU\to R$ from the complex bordism spectrum~\cite[Lemma I.4.6]{Adams}.
        \item A \emph{coordinate} is a class $\xi \in R^2\CP^\infty$ which restricts to $\Sigma^2 1 \in R^2 S^2$ along $S^2 \simeq \CP^1 \to \CP^\infty$ ~\cite[I.2.1]{Adams}.
    \end{enumerate}
    A coordinate induces an isomorphism $R_*\ps{\xi} \to R^* \CP^\infty$.
    The K{\"u}nneth formula then gives an isomorphism $R^*(\CP^\infty \times \CP^\infty)\simeq R_* \ps{\xi_l,\xi_r}$ where $\xi_l$ and $\xi_r$ are the pullbacks of $\xi$ along the two projections.
    Let $\mu\co \CP^\infty \times \CP^\infty \to \CP^\infty$ be the multiplication map classifying the tensor product of line bundles.
    The \emph{formal group law} $F(\xi_l,\xi_r)$ is defined to be the power series associated to the element $\mu^* \xi$~\cite[Lemma I.2.7]{Adams}.
    We often denote it by $\xi_l +_F \xi_r$.
\end{definition}

\begin{example}\label{rem_additive}
    Note that $\MU$ is a connective ring spectrum with $\pi_0 \MU = \Z$, so the Postnikov truncation functor induces a ring map $\MU \to \Z$.
    Thus, if $R$ is a $\Z$-algebra, there is a canonical orientation $\MU \to \Z \to R$.
    We will refer to this as the \emph{additive} orientation of $R$, as its group law is $F(x_l, x_r) = x_l + x_r$.
\end{example}

\begin{example}\label{examp_Todd_Witt_coords}
    Write $\pi_* \KU = \Z[\beta^\pm]$ and $\pi_* \KU_\Tate = \Z[\beta^\pm]\ps{q}$ with $|\beta| = 2$ and $|q| = 0$.
    \begin{itemize}
        \item
        The \emph{Todd orientation} $\omega_\mathrm{Td}\co \MU \to \KU$ has coordinate~\cite[I.2.3]{Adams} \[\xi_\mathrm{Td} := \beta^{-1} (1 - [L]) \in \KU^2 \CP^\infty. \]
        \item
        The \emph{Witten orientation} $\omega_\mathrm{Wit}\co \MU \to \KU_\Tate$ has coordinate~\cite[(6)]{Zagier}
    \[\xi_\mathrm{Wit} = \beta^{-1 }(1 - [L]) \prod_{j=1}^\infty \frac{(1 - [L] q^j) (1 - [L^{-1}] q^j)}{(1 - q^j)^2}\in \KU_\Tate^2 \CP^\infty.\]
    \end{itemize}
    For later convenience we also define the degree zero element $s := \beta \xi_\mathrm{Td}$.
\end{example}

\begin{definition}[{\cite[Corollary I.7.15]{Adams}}]
     Let $\omega\co \MU\to R$ be an orientation, which begets an orientation $\omega_\Q\co \MU\to R\otimes \Q$.
     Since $R\otimes \Q$ is a $\Z$-algebra, it has an additive orientation (cf.\ \Cref{rem_additive}) with coordinate $x$.
     Write $\xi_\omega$ for the coordinate associated to $\omega_\Q$.
     In terms of $x$, it is a power series of the form $\mathrm{exp}_\omega(x) := \xi_\omega(x) = x + O(x^2)$.
     It is called the \emph{exponential} of the formal group law $F_\omega$ associated to $\omega$ and satisfies the formula \[\mathrm{exp}_\omega(x_l)+_{F_\omega}\mathrm{exp}_\omega(x_r)=\mathrm{exp}_\omega(x_l+x_r).\]
\end{definition}

\begin{example}[{\cite[Chapter 3]{Hirzebruch}}]\label{examp_todd_witten_exp}
    The exponential of the Todd orientation is
    \[\mathrm{exp}_\mathrm{Td}(x)=\beta^{-1}(1-e^{-\beta x}).\]
    Therefore we have
    \[\mathrm{exp}_\mathrm{Wit}(x)=\beta^{-1}(1-e^{-\beta x})\prod_{j=1}^\infty\frac{(1-e^{-\beta x}q^j)(1-e^{\beta x}q^j)}{(1-q^j)^2}.\]
\end{example}

\begin{definition}[{\cite[Section 3.4]{AHR}}]\label{defn_neps}
    The \emph{nepers} of an orientation $\omega\co \MU\to R$ are the elements $N_n^\omega \in\pi_{2n}R\otimes \Q$ defined by the formula
    \[\log \left(\frac{x}{\mathrm{exp}_\omega(x)}\right)=\sum_{n=1}^\infty N_n^\omega \frac{x^n}{n!}\]
    When $R$ has a distinguished periodicity element $\beta\in\pi_2R$, it will often be convenient to work with the degree zero element $\beta^{-n}N_n^\omega $ instead of $N_n^\omega $.
\end{definition}

\begin{remark}[{\cite[Proposition 3.12]{AHR}}]
    The division of the two Thom classes for the universal line bundle $L\to\CP^\infty$ is classically referred to as the \emph{Hirzebruch series} of $\omega$. One can form the same division of Thom classes for any vector bundle, and the neper $N_n^\omega$ is what one obtains for the bundle over $S^n$ corresponding to a generator of $\pi_{2n}\BU$.
\end{remark}

Next, we outline a construction for getting new orientations from old ones, due to Ando, French, and Ganter~\cite{AFG}.
First, we recall a fundamental piece of homotopy theory.

\begin{definition}[{e.g., \cite{GreenleesMay}, \cite{NikolausScholze}}]\label{defn_Tate}
    Let $R$ be a complex oriented ring spectrum with coordinate $\gamma$.
    The \emph{Tate fixed-point spectrum} $R^{t\T}$ is the ring spectrum obtained by inverting the element $\gamma$ in the ring spectrum $R^{\CP^\infty}$.
    In particular $\pi_* R^{t\T}$ is the Laurent series ring $R_*\ls{\gamma}$, and there is a \emph{unit} map $u\co R \to R^{t\T}$ of ring spectra, 
    inducing the obvious inclusion $R_*\to R_*\ls{\gamma}$ on homotopy groups.
\end{definition}

\begin{definition}[{\cite{AFG}}]\label{defn_sharp}
    Let $\omega \co \MU \to R$ be a complex orientation with its induced isomorphism $\pi_* R^{t\T} \simeq R_*\ls{\gamma}$.
    Define the \emph{unit orientation} of $R^{t\T}$ as the composite
    \[\MU \xrightarrow{\omega} R \xrightarrow{u} R^{t\T},\]
    and let $\xi$ and $+_F$ denote the corresponding coordinate and formal group law.
    Define the \emph{sharped orientation} $\omega^\# \co \MU \to R^{t\T}$ to be the ring map associated to the coordinate
    \[\xi \cdot \frac{\gamma}{\xi +_F \gamma}.\]
    Let $y$ be such that $\exp_\omega(y) = \gamma$.
    Then the sharped exponential satisfies
    \[\exp_{\omega^\#}(x) = \exp_\omega(x) \cdot \frac{\exp_\omega(y)}{\exp_\omega(x + y)}.\]
\end{definition}

\begin{remark}[{\cite[Sections 6--7]{AFG}}]
    The \emph{Jacobi orientation} arises as the sharp construction applied to the Witten orientation.
    On homotopy, it takes values in meromorphic Jacobi forms of index $0$.
\end{remark}



\subsection{R\'esum\'e on $\E_\infty$--rings}

Whereas in classical algebra commutativity is a \emph{property} of a ring and maps of rings are automatically compatible with any commutativity present, a ring spectrum in homotopy theory requires \emph{extra data} to witness its commutativity, as do maps of commutative ring spectra.
A ring spectrum with such extra structure is called an $\E_\infty$--ring, and such a map an $\E_\infty$--ring map.
The complex bordism spectrum $\MU$ is naturally an $\E_\infty$--ring (\cite[Chapter III]{May}), and thus if $R$ is an $\E_\infty$--ring one can ask when a complex orientation $\MU \to R$ admits an $\E_\infty$--structure.
This question has received much interest and has seen recent progress (\cite{AHR}, \cite{HopkinsLawson}, \cite{HahnYuan}, \cite{Balderrama}, \cite{nullstel}, \ldots), but it is usually quite difficult---the number of known $\E_\infty$--orientations is small, the crown jewel of which is the \emph{string orientation} of $\mathrm{tmf}$.

Our interest comes from a result of the first author:

\begin{theorem}[{\cite[Theorem 5.14]{CarmeliLuecke}}]\label{thm_cl}
    The sharp construction sends $\E_\infty$--orientations to $\E_\infty$--orientations.
    \qed
\end{theorem}

\noindent
Therefore the sharped string orientation (or any of the handful of $\E_\infty$--orientations it begets, cf.\ \Cref{thm_todd_witten_Einfty}) automatically admits this extra $\E_\infty$--structure, and we would like to understand it in detail.
Ultimately, we will conclude that it entails the congruences advertised in the introduction. 

Let us describe others' pursuit of the program to express this $\E_\infty$--structure in a manageable form for palatable target rings.
There is a distinguished class of spectra, called \emph{$K(1)$--local}, whose $\E_\infty$--rings carry an extremely nice theory of cohomology operations, controlled by an essentially unique operation $\psi^p$~\cite[Chapter IX]{BMMS}.
A key insight of Ando, Hopkins, and Rezk is that if $R$ is $K(1)$--local, then an $\E_\infty$-structure on an orientation $\omega\co \MU \to R$ is well captured by its interaction with $\psi^p$.
Our rings of interest, $\KU^{t\T}$ and $\KU_\Tate^{t\T}$, become $K(1)$--local after $p$--completion.
Their $\psi^p$ maps are as follows:

\begin{example}\label{examp_psi_data}
     Upon $p$--completion, the Todd and Witten orientations of \Cref{examp_Todd_Witt_coords} induce identifications
     \begin{align*}
     \pi_*(\KU^{t\T})^\wedge_p & = \Z[\beta^\pm]\ls{t}^\wedge_p, &
     \pi_*(\KU_\Tate^{t\T})^\wedge_p & = \Z[\beta^\pm]\ps{q}\ls{t}^\wedge_p,
     \end{align*}
     where we have adjusted the Laurent element $t := \beta \gamma$ into degree 0 (cf.\ \Cref{defn_Tate}).
     The operations $\psi^p$ are the ring homomorphisms determined by  
\begin{align*}
    \psi^p(t) & = 1 - (1 - t)^p, &
    \psi^p(q) & = q^p, & \psi^p(\beta)=p\beta.
\end{align*}
For a homotopy element $f_q(t) \cdot \beta^k \in \pi_{2k}\KU_\Tate^{t\T} \cong \Z\ps{q}\ls{t}^\wedge_p\{\beta^k\}$ we then have
\[\psi^p(f_q(t) \cdot \beta^k) = p^k f_{q^p}(1 - (1 - t)^p) \cdot \beta^k.\]
\end{example}

\noindent
We measure the interaction of $\psi^p$ with the orientation $\omega$ like so:

\begin{definition}[{\cite[Proposition 7.10]{AHR}}]\label{defn_ahr_elements}
    Pick any $c \in \Z_p^\times \setminus \{\pm 1\}$.
    The \emph{moments} of an orientation $\omega\co \MU \to R$ of a $K(1)$--local ring are the elements
    \[M_n^\omega = (1 - c^n) \left( \mathrm{id} - \frac{\psi^p}{p} \right)(N_n^\omega) \in \pi_{2n} R \otimes \Q.\]
    As at the end of \Cref{defn_neps}, when $R$ has a distinguished periodicity element $\beta \in \pi_2 R$, it will often be convenient to work with the degree 0 element $\beta^{-n} M_n^\omega$. However, one has to take a bit of care, since $\psi^p(\beta)$ need not be equal to $\beta$ (as \Cref{examp_psi_data} shows).
\end{definition}

\begin{remark}[\cite{Rezk}]
    The operator $\mathrm{id} - p^{-1} \psi^p$ is part of a more general theory of \emph{logarithmic cohomology operations}.
\end{remark}

\begin{remark}[{\cite[Proposition 14.6]{AHR}}]
    The factor $1-c^n$ comes from an \emph{Adams resolution} of the $K(1)$--localized sphere spectrum.
\end{remark}

\noindent
We can now state a key result of Ando--Hopkins--Rezk:

\begin{theorem}[{\cite[Theorem 6.1]{AHR}, \cite[Appendix A.4]{Peterson}}]\label{thm_ahr}
    Let $\omega\co \MU \to R$ be an $\E_\infty$--orientation of a $K(1)$--local ring $R$ which is flat over $\KU^\wedge_p$.
    The moments $M^\omega_{n \ge 1}$ belong to a $p$--adic moment sequence with values in $R_*$.
    \qed
\end{theorem}

\begin{remark}[{\cite[Lemma 7.14, Proposition 14.6.2, Section 3.5]{AHR}, \cite{Miller}}]
    Orientations are automatically suitably compatible with the unit map of the ring spectrum, which manifests in the moment sequence as the identity \[\lim_{j \to \infty} M_{(p-1) p^j} = M_0 = p^{-1} \log c^{p-1}.\]
\end{remark}

Finally, we need some existing $\E_\infty$--orientations of $\KU$ and $\KU_\Tate$ to which we will apply \Cref{thm_cl} to get $\E_\infty$--orientations of $(\KU^{t\T})^\wedge_p$ and $(\KU_\Tate^{t\T})^\wedge_p$, and then further apply \Cref{thm_ahr} to get $p$--adic moment sequences.
We record the following theorem.

\begin{theorem}[{\cite[Theorem 10.3, Proposition 10.10]{AHR}}]\label{thm_todd_witten_Einfty}
    After $p$--completion, the orientations of \Cref{examp_Todd_Witt_coords} 
    \begin{align*}
    \omega_\mathrm{Td}\co \MU & \to \KU^\wedge_p, &
    \omega_\mathrm{Wit}\co \MU & \to (\KU_\Tate)^\wedge_p
    \end{align*}
    admit $\E_\infty$--structures.
    \qed
\end{theorem}
\begin{remark}
    To be precise, the statement of Proposition 10.10 of \cite{AHR} produces an $\E_\infty$--structure on the $\MU\<6\>$--orientation associated to $\omega_\mathrm{Wit}$.
    However, the same proof, with the target ring changed to $\pi_* \KU_\Tate$, shows that $\omega_\mathrm{Wit}$ itself admits an $\E_\infty$--orientation.
    One can also use obstruction theory along the Whitehead tower of $\bu$ 
    to push the $\E_\infty$--structure from $\MU\<6\>$ to $\MU$.
\end{remark}


\section{Proofs of the Main Theorems}

We now apply this machinery to the sharped orientations of (Tate) $K$--theory from \Cref{thm_todd_witten_Einfty} to produce $p$--adic moment sequences.
It will be convenient to fix the following notation:
\begin{align*}
    2 \pi i z & = x = -\log(1 - s) = -\log(u), &
    2 \pi i \alpha & = y = -\log(1 - t) = -\log(v).
\end{align*}
We remark that while $s$ agrees with the element defined at the end of \Cref{examp_Todd_Witt_coords}, $x$ does not quite agree with the additive coordinate of previous sections---it is the degree 0 element gotten by using the periodicity element $\beta$. This slight abuse of notation is well worth it, as the formulas that follow would otherwise be infested with powers of $\beta$.

Throughout, let $c \in \Z_p^\times \setminus \{\pm 1\}$ be a choice of $p$--adic unit.

\subsection{Todd and Todd${}^\#$}

\begin{lemma}[{\cite[Theorem 7.2]{Koblitz}, \cite[Proposition 10.2]{AHR}}]\label{ToddMoments}
    There is a $p$--adic moment sequence which in positions $n \ge 1$ is given by \[(1 - c^n) (1 - p^{n-1})\frac{-B_n}{n},\] where $B_n$ denotes the $n${\th} Bernoulli number.
\end{lemma}
\begin{proof}
Using \Cref{examp_todd_witten_exp}, the nepers associated to the Todd orientation are
\[\sum_{n=0}^\infty N_n^\mathrm{Td} \frac{x^n}{n!} = \log\left( \frac{x}{s} \right) = \log \left( \frac{x}{1 - e^{-x}} \right).\]
The series expansion of the right-hand side follows from:
\[
\frac{d}{dx} \log \frac{x}{1 - e^{-x}} = \frac{1}{x} - \frac{e^{-x}}{1 - e^{-x}} = \frac{1}{x} - \frac{1}{x} \cdot \frac{x}{e^x - 1} = \sum_{n=1}^\infty -B_n \frac{x^{n-1}}{n!}.\]
We then calculate the moments (cf.\ \Cref{defn_ahr_elements}) and find that they are the claimed sequence, which is a $p$-adic moment sequence by \Cref{thm_ahr}.
\qedhere
\end{proof}

\begin{proof}[{Proof of \Cref{MainTheoremTodd}}]
Combining \Cref{examp_todd_witten_exp} and \Cref{defn_sharp}, the nepers associated to the Todd${}^\#$ orientation are \[\sum_{n=0}^\infty N_n^{\mathrm{Td}^\#} \frac{x^n}{n!} = \log\left(\frac{x}{s}\right) + \log\left(\frac{1 - e^{-x}(1 - t)}{t}\right).\]
The first summand is the subject of \Cref{ToddMoments}, so we focus on the second: \[\log \left( \frac{1 - t^{-1}(1 - e^x)}{e^x} \right) = -x - \sum_{n=1}^\infty\frac{x^n}{n!} \sum_{m=1}^n \frac{t^{-m}}{m} \sum_{j=0}^m (-1)^j \binom{m}{j} j^n.\]
The inner sum is exactly the iterated forward finite difference $\Delta^m[r^n](0)$.
Apply \Cref{thm_ahr} to produce the $p$-adic moment sequence:
\[M^{\mathrm{Td}^\#}_n = (1 - c^n)\left(\sum_{m=1}^n \frac{\Delta^m[r^n](0)}{m} \left( t^{-m} - p^{n-1} (1 - (1 - t)^p)^{-m} \right) \right).\]
The binomial raised to the $(-m)${\th} power does not naively expand, as $(1 - t)^p$ is not a topologically nilpotent element of $\pi_* (K^{t\T})_p^\wedge$.
However, by defining $\zeta$ such that $(t^{-1} - 1)^p = t^{-p} - 1 + p \zeta$, we find:
\[M^{\mathrm{Td}^\#}_n = (1 - c^n)\left(\sum_{m=1}^n \frac{\Delta^m[r^n](0)}{m} \left( (t^{-1})^m - p^{n-1} t^{-pm} (1 - p \zeta)^{-m} \right) \right),\]
which does define a convergent series in $\pi_* (K^{t\T})_p^\wedge$.
\end{proof}

\begin{remark}
The $t^{-m}$ term in $N_n^{\mathrm{Td}^\#}$ vanishes when $m > n$ for two reasons: the $m$--fold forward finite difference operator annihilates $r^n$, and the term $(1 - e^{-x})^m$ appearing in the expansion of $\log(1 - t^{-1}(1 - e^x))$ has $x$--degree at least $m$.
\end{remark}


\subsection{Witten and Witten${}^\#$}

\begin{lemma}[{\cite[Lemma 3.5.6]{Katz}, \cite[Propositions 10.9--10]{AHR}}]\label{WittenMoments}
    There is a $\Z\ps{q}^\wedge_p$--valued moment sequence which in even positions $2n \ge 1$ is given by
    \[\frac{(1 - c^{2n}) \cdot (2n-1)!}{(2 \pi i)^{2n}} \cdot \left(G_{2n}(q) - p^{2n-1} G_{2n}(q^p)\right)\]
    and which vanishes in odd positions $2n+1 \ge 1$.
\end{lemma}
\begin{proof}
The exponential for the Witten orientation of Tate $K$--theory in \Cref{examp_todd_witten_exp} shows the associated sequence of nepers to be
\[\sum_{n=1}^\infty N_n^{\Wit} \frac{x^n}{n!} = \log \left( \frac{x}{1 - u}\right) + \log \left(\prod_{j=1}^\infty \frac{(1 - q^j)^2}{(1 - u q^j)(1 - u^{-1} q^j)} \right).\]
The first summand is felled by \Cref{ToddMoments}, so we concentrate on the second:
\begin{align*}
    \log \left(\prod_{j=1}^\infty \frac{(1 - q^j)^2}{(1 - u q^j)(1 - u^{-1} q^j)} \right)
    & = \sum_{j=1}^\infty \sum_{d=1}^\infty \frac{1}{d} (-2 q^{jd} + q^{jd} e^{-dx} + q^{jd} e^{dx}) \\
    & = 2 \sum_{n=1}^\infty \frac{x^{2n}}{(2n)!} \sum_{j=1}^\infty q^j \sigma_{2n-1}(j).
\end{align*}
Comparing with \Cref{EisensteinSeries} shows \[N_n^\Wit = \frac{(2n-1)!}{(2 \pi i)^{2n}} G_{2n}(q).\]
All that remains is to calculate the moments $M_n^\Wit$ and apply \Cref{thm_ahr}.
\end{proof}

\begin{proof}[Proof of \Cref{MainTheoremWitten}]
Again, combining \Cref{examp_todd_witten_exp} and \Cref{defn_sharp}, the sharped Witten nepers are given by \[\sum_{n=0}^\infty \frac{x^n}{n!} N_n^{\Wit^\#} = \log\left(\frac{x}{\exp_\Wit(x)}\right) + \log\left(\frac{\exp_\Wit(x + y)}{\exp_\Wit(y)}\right).\]
The first summand is felled by \Cref{WittenMoments}, so we focus on expanding the second.
Note from \Cref{SigmaFunction} that the Witten exponential can be written in terms of $\sigma_q$ as \[\exp_\Wit(x) = 2 \pi i \cdot \sigma_q(z) \cdot \exp\left(-\frac{1}{2} (2 \pi i \cdot z + G_2(q) \cdot z^2) \right),\]
and hence that second summand can be written as
\begin{align*}
    \log\left(\frac{\exp_\Wit(x + y)}{\exp_\Wit(y)}\right)
    & = \left(-\frac{1}{2} + \frac{G_2(q)}{(2 \pi i)^2} \log(1 - t)\right) \cdot x + \frac{-G_2(q)}{2 (2 \pi i)^2} \cdot x^2 \\
    & \quad \quad + \sum_{n=1}^\infty \frac{x^n}{n!} \cdot \frac{1}{(2 \pi i)^n} \cdot \wp_q^{(n-2)} \left(\frac{\log(1 - t)}{-2 \pi i}\right).
\end{align*}
The moments are then calculated to be exactly the expressions appearing in the Theorem statement, so \Cref{thm_ahr} finishes the proof.
\end{proof}

\section{Loose ends}

\begin{remark}
    The recipe followed here applies to any other $\E_\infty$--orientations the reader may have on hand.
    \begin{itemize}
        \item \cite{AHR}: 
        The spectrum $\mathrm{TMF}$ is a topological incarnation of the moduli stack of elliptic curves over a general base.
        Its homotopy, analogous to the cohomology of the moduli stack, brings to light certain torsion phenomena not visible in the homotopy of $\pi_* \KU_\Tate$.
        The $\sigma$--orientation of $\mathrm{tmf}$ refines the Witten orientation of $\KU_\Tate$ and is itself $\E_\infty$, suggesting one possible starting point for a generalization of our results.
        \item \cite{HillLawson}: 
        Topologists have also constructed variations on $\mathrm{TMF}$ for other congruence subgroups (though lifts of the $\sigma$--orientation have only been considered in limited situations).
        The number theory literature makes substantial use of these other stacks, and so an enrichment of our results to that setting would provide much better contact.
        \item \cite{LinYamashita}: 
        Topologists have recently constructed a spectrum $\mathrm{TJF}$ which refines the theory of Jacobi forms analogously to how $\mathrm{TMF}$ refines the moduli of elliptic curves.
        This perhaps forms a more appropriate target for parametric Jacobi orientations than $\mathrm{TMF}^{t\T}$.
        (We thank Yamashita for bringing this to our attention.)
        \item \cite[Appendix II]{MazurTate}: 
        Mazur and Tate construct functions analogous to the classical Weierstrass function, but over other adic rings and specified by quite different formulas.
        Since $\mathrm{TMF}$ lies over a general base, the $\sigma$--orientation suitably restricted may shed light on these other $\sigma$--functions.
        \item \cite{Balderrama}, \cite{nullstel}: 
        Morava $E$--theories are known to carry $\E_\infty$--orientations, and the relevant logarithms bear a number-theoretic description~\cite[1.10]{Rezk}.
        To our knowledge, the associated congruences have not been explored---nevermind their sharped versions.
    \end{itemize}
\end{remark}

\begin{remark}
    It is not clear to the authors how to unify the many derivatives appearing in this document.
    For example, it is tempting to interpret the finite differences appearing in \Cref{MainTheoremTodd} as belonging to a kind of pairing: \[\left\< \log \left( \frac{x}{\exp_{\Td^\#}(x)} \right), x^n t^{-m} \right\> = \frac{1}{n! m} \Delta^m[r^n](0),\] where the role of $t^{-m}$ is to select the iterated finite difference operator $\Delta^m|_{r=0}$ and $x^n$ to select the test polynomial $r^n$.
    This interpretation does not obviously extend to the positive $t$--terms appearing in $N_n^{\Wit^\#}$, but \emph{some} total interpretation seems surely possible.
\end{remark}

\begin{remark}[{\cite[Section 3]{Sofer}}]
    Sofer explores a construction which extracts from a Jacobi form a $p$--adic measure valued in modular forms, using formulas of similar flavor to ours but ``with one fewer variable''.
    It could be interesting to have some identification of the \emph{measure} in \Cref{MainTheoremWitten} in those terms, or to generalize their construction to produce measures valued in Jacobi forms, or even to see the limitations of any such generalizations.

    More concretely, it would be interesting to see topology-free proofs of the congruences described in our Main Theorems.
    One specific opportunity for improvement is that Ando, Hopkins, and Rezk's manual analysis allows them to ``halve'' the Mazur measure~\cite[Section 10.3]{AHR}.
    \Cref{ToddTableAt2} and \Cref{WittenTableAt2} suggest that this is possible for the measure constructed in \Cref{MainTheoremWitten}; perhaps this is accessible through topology, but it is likely equally accessible directly.
\end{remark}

\begin{remark}
    The three factors appearing in \Cref{ToddMoments} have discernable provenances in both homotopy theory and in number theory.
    In homotopy theory, they belong to the finite Adams resolution of the $K(1)$--local sphere, the Rezk logarithm, and the characteristic series of the Todd orientation.
    In number theory, they arise from the regularization of a distribution to a measure, the restriction to $\Z_p^\times$ to perform $p$--adic interpolation, and as special values of the $\zeta$--function.
    It would be deeply interesting to have a tighter correspondence between these three pieces.
    The Main Theorems also provide a new proving ground for any such analogy, which should additionally assign a measure-theoretic meaning to $\psi^p(t) = 1-(1-t)^p = [p]_{\widehat{\Gm}}(t)$.
\end{remark}

\begin{remark}
    In its incarnation as an operation on $\E_\infty$--rings (cf.\ \cite{CarmeliLuecke}), the sharp construction is part of a more general theory of $\E_\infty$--orientations of Tate objects.
    For example, this theory also subsumes the \emph{Frobenius homomorphism} of \cite{NikolausScholze}.
    One may wonder what the moment sequence associated to the Frobenius twist of the Witten orientation is.
\end{remark}

\bibliographystyle{halpha}
\bibliography{references}

\begin{figure}[p]
\resizebox{1.0\hsize}{!}{\input{todd.3.data.tex}}
\caption{%
A range of ternary expansions of the $3$--adic Todd${}^\#$ moments at $c = 1 + p$.
\Cref{MainTheoremTodd} with test polynomial $p^{-i} (r^{(p-1) p^i} - 1)$ shows that the $i${\th} digit within a column is periodic with periodicity $(p-1) p^i$. The colors are there to help the eye catch the vertical patterns.
}\label{ToddTableAt3}
\end{figure}

\begin{figure}[p]
\input{witten.3.data.tex}
\caption{%
Some $3$--adic Witten${}^\#$ moments at $c = 1 + p$, in a range of $t$-- and $q$--degrees.
Congruences here relate the colored digits across different tables rather than within columns.
The relevant test polynomials are \textcolor{orange}{$p^{-2}(r^{(p-1)p} - 1)$} and \textcolor{magenta}{$p^{-3}(r^{(p-1)p^2} - 1)$}.
}\label{WittenTableAt3}
\end{figure}

\begin{figure}[p]
\resizebox{1.0\hsize}{!}{\input{todd.2.data.tex}}
\caption{%
A range of binary expansions of the $2$--adic Todd${}^\#$ moments at $c = 1 + p$.
\Cref{MainTheoremTodd} with test polynomial $2^{-(i+2)} (r^{2^i} - 1)$ shows that the $0${\th}, $1$\textsuperscript{st}, and $2$\textsuperscript{nd} digit columns are $2$--periodic, and then after those the $i${\th} digit column is $2^{i-1}$--periodic.
}\label{ToddTableAt2}
\end{figure}

\begin{figure}[p]
\input{witten.2.data.tex}
\caption{%
Some $2$--adic Witten${}^\#$ moments at $c = 1 + p$, in a range of $t$-- and $q$--degrees.
Congruences here relate the colored digits across different tables rather than within columns.
Congruences here relate the colored digits across different tables rather than within columns.
The relevant test polynomials are \textcolor{orange}{$p^{-3}(r^2 - 1)$}, \textcolor{magenta}{$p^{-4}(r^4 - 1)$}, and \textcolor{cyan}{$p^{-5}(r^8 - 1)$}.
}\label{WittenTableAt2}
\end{figure}







\end{document}

%% file: todd.2.data.truncated.tex
$\begin{array}{r|cccccc}
& t^0 & t^{-1} & t^{-2} & t^{-3} & t^{-4} & t^{-5} \\
\hline
M_{1}^{\Wit^\#} & \hdots {\color[rgb]{0.451848, 0., 0.968246}0}{\color[rgb]{0., 0.935414, 0.904234}0}{\color[rgb]{0.173205, 0.866025, 0.}0}{\color[rgb]{0.707107, 0.683537, 0.}0}{\color[rgb]{0.707107, 0.329983, 0.}0}{\color[rgb]{0.707107, 0.212132, 0.}0}_{2} & \hdots {\color[rgb]{0.451848, 0., 0.968246}1}{\color[rgb]{0., 0.935414, 0.904234}1}{\color[rgb]{0.173205, 0.866025, 0.}1}{\color[rgb]{0.707107, 0.683537, 0.}1}{\color[rgb]{0.707107, 0.329983, 0.}1}{\color[rgb]{0.707107, 0.212132, 0.}0}_{2} & \hdots {\color[rgb]{0.451848, 0., 0.968246}1}{\color[rgb]{0., 0.935414, 0.904234}1}{\color[rgb]{0.173205, 0.866025, 0.}1}{\color[rgb]{0.707107, 0.683537, 0.}1}{\color[rgb]{0.707107, 0.329983, 0.}1}{\color[rgb]{0.707107, 0.212132, 0.}0}_{2} & \hdots {\color[rgb]{0.451848, 0., 0.968246}1}{\color[rgb]{0., 0.935414, 0.904234}1}{\color[rgb]{0.173205, 0.866025, 0.}1}{\color[rgb]{0.707107, 0.683537, 0.}1}{\color[rgb]{0.707107, 0.329983, 0.}0}{\color[rgb]{0.707107, 0.212132, 0.}0}_{2} & \hdots {\color[rgb]{0.451848, 0., 0.968246}1}{\color[rgb]{0., 0.935414, 0.904234}1}{\color[rgb]{0.173205, 0.866025, 0.}1}{\color[rgb]{0.707107, 0.683537, 0.}0}{\color[rgb]{0.707107, 0.329983, 0.}0}{\color[rgb]{0.707107, 0.212132, 0.}0}_{2} & \hdots {\color[rgb]{0.451848, 0., 0.968246}1}{\color[rgb]{0., 0.935414, 0.904234}1}{\color[rgb]{0.173205, 0.866025, 0.}0}{\color[rgb]{0.707107, 0.683537, 0.}0}{\color[rgb]{0.707107, 0.329983, 0.}0}{\color[rgb]{0.707107, 0.212132, 0.}0}_{2} \\
M_{2}^{\Wit^\#} & \hdots {\color[rgb]{0.436527, 0., 0.935414}1}{\color[rgb]{0., 0.866025, 0.837158}0}{\color[rgb]{0.141421, 0.707107, 0.}1}{\color[rgb]{0., 0., 0.}0}{\color[rgb]{0., 0., 0.}1}{\color[rgb]{0., 0., 0.}0}_{2} & \hdots {\color[rgb]{0.436527, 0., 0.935414}1}{\color[rgb]{0., 0.866025, 0.837158}1}{\color[rgb]{0.141421, 0.707107, 0.}1}{\color[rgb]{0., 0., 0.}0}{\color[rgb]{0., 0., 0.}0}{\color[rgb]{0., 0., 0.}0}_{2} & \hdots {\color[rgb]{0.436527, 0., 0.935414}1}{\color[rgb]{0., 0.866025, 0.837158}1}{\color[rgb]{0.141421, 0.707107, 0.}1}{\color[rgb]{0., 0., 0.}0}{\color[rgb]{0., 0., 0.}0}{\color[rgb]{0., 0., 0.}0}_{2} & \hdots {\color[rgb]{0.436527, 0., 0.935414}1}{\color[rgb]{0., 0.866025, 0.837158}0}{\color[rgb]{0.141421, 0.707107, 0.}0}{\color[rgb]{0., 0., 0.}0}{\color[rgb]{0., 0., 0.}0}{\color[rgb]{0., 0., 0.}0}_{2} & \hdots {\color[rgb]{0.436527, 0., 0.935414}1}{\color[rgb]{0., 0.866025, 0.837158}1}{\color[rgb]{0.141421, 0.707107, 0.}0}{\color[rgb]{0., 0., 0.}0}{\color[rgb]{0., 0., 0.}0}{\color[rgb]{0., 0., 0.}0}_{2} & \hdots {\color[rgb]{0.436527, 0., 0.935414}0}{\color[rgb]{0., 0.866025, 0.837158}0}{\color[rgb]{0.141421, 0.707107, 0.}0}{\color[rgb]{0., 0., 0.}0}{\color[rgb]{0., 0., 0.}0}{\color[rgb]{0., 0., 0.}0}_{2} \\
M_{3}^{\Wit^\#} & \hdots {\color[rgb]{0.420648, 0., 0.901388}0}{\color[rgb]{0., 0.790569, 0.764217}0}{\color[rgb]{0.1, 0.5, 0.}0}{\color[rgb]{0.707107, 0.683537, 0.}0}{\color[rgb]{0.707107, 0.329983, 0.}0}{\color[rgb]{0.707107, 0.212132, 0.}0}_{2} & \hdots {\color[rgb]{0.420648, 0., 0.901388}1}{\color[rgb]{0., 0.790569, 0.764217}0}{\color[rgb]{0.1, 0.5, 0.}0}{\color[rgb]{0.707107, 0.683537, 0.}1}{\color[rgb]{0.707107, 0.329983, 0.}1}{\color[rgb]{0.707107, 0.212132, 0.}0}_{2} & \hdots {\color[rgb]{0.420648, 0., 0.901388}1}{\color[rgb]{0., 0.790569, 0.764217}0}{\color[rgb]{0.1, 0.5, 0.}0}{\color[rgb]{0.707107, 0.683537, 0.}1}{\color[rgb]{0.707107, 0.329983, 0.}1}{\color[rgb]{0.707107, 0.212132, 0.}0}_{2} & \hdots {\color[rgb]{0.420648, 0., 0.901388}1}{\color[rgb]{0., 0.790569, 0.764217}1}{\color[rgb]{0.1, 0.5, 0.}1}{\color[rgb]{0.707107, 0.683537, 0.}1}{\color[rgb]{0.707107, 0.329983, 0.}0}{\color[rgb]{0.707107, 0.212132, 0.}0}_{2} & \hdots {\color[rgb]{0.420648, 0., 0.901388}1}{\color[rgb]{0., 0.790569, 0.764217}0}{\color[rgb]{0.1, 0.5, 0.}1}{\color[rgb]{0.707107, 0.683537, 0.}0}{\color[rgb]{0.707107, 0.329983, 0.}0}{\color[rgb]{0.707107, 0.212132, 0.}0}_{2} & \hdots {\color[rgb]{0.420648, 0., 0.901388}1}{\color[rgb]{0., 0.790569, 0.764217}0}{\color[rgb]{0.1, 0.5, 0.}0}{\color[rgb]{0.707107, 0.683537, 0.}0}{\color[rgb]{0.707107, 0.329983, 0.}0}{\color[rgb]{0.707107, 0.212132, 0.}0}_{2} \\
M_{4}^{\Wit^\#} & \hdots {\color[rgb]{0.404145, 0., 0.866025}0}{\color[rgb]{0., 0.707107, 0.683537}1}{\color[rgb]{0., 0., 0.}1}{\color[rgb]{0., 0., 0.}0}{\color[rgb]{0., 0., 0.}1}{\color[rgb]{0., 0., 0.}0}_{2} & \hdots {\color[rgb]{0.404145, 0., 0.866025}1}{\color[rgb]{0., 0.707107, 0.683537}1}{\color[rgb]{0., 0., 0.}0}{\color[rgb]{0., 0., 0.}0}{\color[rgb]{0., 0., 0.}0}{\color[rgb]{0., 0., 0.}0}_{2} & \hdots {\color[rgb]{0.404145, 0., 0.866025}1}{\color[rgb]{0., 0.707107, 0.683537}1}{\color[rgb]{0., 0., 0.}0}{\color[rgb]{0., 0., 0.}0}{\color[rgb]{0., 0., 0.}0}{\color[rgb]{0., 0., 0.}0}_{2} & \hdots {\color[rgb]{0.404145, 0., 0.866025}0}{\color[rgb]{0., 0.707107, 0.683537}0}{\color[rgb]{0., 0., 0.}0}{\color[rgb]{0., 0., 0.}0}{\color[rgb]{0., 0., 0.}0}{\color[rgb]{0., 0., 0.}0}_{2} & \hdots {\color[rgb]{0.404145, 0., 0.866025}1}{\color[rgb]{0., 0.707107, 0.683537}0}{\color[rgb]{0., 0., 0.}0}{\color[rgb]{0., 0., 0.}0}{\color[rgb]{0., 0., 0.}0}{\color[rgb]{0., 0., 0.}0}_{2} & \hdots {\color[rgb]{0.404145, 0., 0.866025}0}{\color[rgb]{0., 0.707107, 0.683537}0}{\color[rgb]{0., 0., 0.}0}{\color[rgb]{0., 0., 0.}0}{\color[rgb]{0., 0., 0.}0}{\color[rgb]{0., 0., 0.}0}_{2} \\
M_{5}^{\Wit^\#} & \hdots {\color[rgb]{0.38694, 0., 0.829156}0}{\color[rgb]{0., 0.612372, 0.59196}0}{\color[rgb]{0.173205, 0.866025, 0.}0}{\color[rgb]{0.707107, 0.683537, 0.}0}{\color[rgb]{0.707107, 0.329983, 0.}0}{\color[rgb]{0.707107, 0.212132, 0.}0}_{2} & \hdots {\color[rgb]{0.38694, 0., 0.829156}0}{\color[rgb]{0., 0.612372, 0.59196}0}{\color[rgb]{0.173205, 0.866025, 0.}1}{\color[rgb]{0.707107, 0.683537, 0.}1}{\color[rgb]{0.707107, 0.329983, 0.}1}{\color[rgb]{0.707107, 0.212132, 0.}0}_{2} & \hdots {\color[rgb]{0.38694, 0., 0.829156}0}{\color[rgb]{0., 0.612372, 0.59196}0}{\color[rgb]{0.173205, 0.866025, 0.}1}{\color[rgb]{0.707107, 0.683537, 0.}1}{\color[rgb]{0.707107, 0.329983, 0.}1}{\color[rgb]{0.707107, 0.212132, 0.}0}_{2} & \hdots {\color[rgb]{0.38694, 0., 0.829156}1}{\color[rgb]{0., 0.612372, 0.59196}1}{\color[rgb]{0.173205, 0.866025, 0.}1}{\color[rgb]{0.707107, 0.683537, 0.}1}{\color[rgb]{0.707107, 0.329983, 0.}0}{\color[rgb]{0.707107, 0.212132, 0.}0}_{2} & \hdots {\color[rgb]{0.38694, 0., 0.829156}0}{\color[rgb]{0., 0.612372, 0.59196}1}{\color[rgb]{0.173205, 0.866025, 0.}1}{\color[rgb]{0.707107, 0.683537, 0.}0}{\color[rgb]{0.707107, 0.329983, 0.}0}{\color[rgb]{0.707107, 0.212132, 0.}0}_{2} & \hdots {\color[rgb]{0.38694, 0., 0.829156}0}{\color[rgb]{0., 0.612372, 0.59196}1}{\color[rgb]{0.173205, 0.866025, 0.}0}{\color[rgb]{0.707107, 0.683537, 0.}0}{\color[rgb]{0.707107, 0.329983, 0.}0}{\color[rgb]{0.707107, 0.212132, 0.}0}_{2} \\
M_{6}^{\Wit^\#} & \hdots {\color[rgb]{0.368932, 0., 0.790569}0}{\color[rgb]{0., 0.5, 0.483333}0}{\color[rgb]{0.141421, 0.707107, 0.}1}{\color[rgb]{0., 0., 0.}0}{\color[rgb]{0., 0., 0.}1}{\color[rgb]{0., 0., 0.}0}_{2} & \hdots {\color[rgb]{0.368932, 0., 0.790569}1}{\color[rgb]{0., 0.5, 0.483333}0}{\color[rgb]{0.141421, 0.707107, 0.}1}{\color[rgb]{0., 0., 0.}0}{\color[rgb]{0., 0., 0.}0}{\color[rgb]{0., 0., 0.}0}_{2} & \hdots {\color[rgb]{0.368932, 0., 0.790569}1}{\color[rgb]{0., 0.5, 0.483333}0}{\color[rgb]{0.141421, 0.707107, 0.}1}{\color[rgb]{0., 0., 0.}0}{\color[rgb]{0., 0., 0.}0}{\color[rgb]{0., 0., 0.}0}_{2} & \hdots {\color[rgb]{0.368932, 0., 0.790569}1}{\color[rgb]{0., 0.5, 0.483333}0}{\color[rgb]{0.141421, 0.707107, 0.}0}{\color[rgb]{0., 0., 0.}0}{\color[rgb]{0., 0., 0.}0}{\color[rgb]{0., 0., 0.}0}_{2} & \hdots {\color[rgb]{0.368932, 0., 0.790569}0}{\color[rgb]{0., 0.5, 0.483333}1}{\color[rgb]{0.141421, 0.707107, 0.}0}{\color[rgb]{0., 0., 0.}0}{\color[rgb]{0., 0., 0.}0}{\color[rgb]{0., 0., 0.}0}_{2} & \hdots {\color[rgb]{0.368932, 0., 0.790569}0}{\color[rgb]{0., 0.5, 0.483333}0}{\color[rgb]{0.141421, 0.707107, 0.}0}{\color[rgb]{0., 0., 0.}0}{\color[rgb]{0., 0., 0.}0}{\color[rgb]{0., 0., 0.}0}_{2} \\
M_{7}^{\Wit^\#} & \hdots {\color[rgb]{0.35, 0., 0.75}0}{\color[rgb]{0., 0.353553, 0.341768}0}{\color[rgb]{0.1, 0.5, 0.}0}{\color[rgb]{0.707107, 0.683537, 0.}0}{\color[rgb]{0.707107, 0.329983, 0.}0}{\color[rgb]{0.707107, 0.212132, 0.}0}_{2} & \hdots {\color[rgb]{0.35, 0., 0.75}1}{\color[rgb]{0., 0.353553, 0.341768}1}{\color[rgb]{0.1, 0.5, 0.}0}{\color[rgb]{0.707107, 0.683537, 0.}1}{\color[rgb]{0.707107, 0.329983, 0.}1}{\color[rgb]{0.707107, 0.212132, 0.}0}_{2} & \hdots {\color[rgb]{0.35, 0., 0.75}1}{\color[rgb]{0., 0.353553, 0.341768}1}{\color[rgb]{0.1, 0.5, 0.}0}{\color[rgb]{0.707107, 0.683537, 0.}1}{\color[rgb]{0.707107, 0.329983, 0.}1}{\color[rgb]{0.707107, 0.212132, 0.}0}_{2} & \hdots {\color[rgb]{0.35, 0., 0.75}1}{\color[rgb]{0., 0.353553, 0.341768}1}{\color[rgb]{0.1, 0.5, 0.}1}{\color[rgb]{0.707107, 0.683537, 0.}1}{\color[rgb]{0.707107, 0.329983, 0.}0}{\color[rgb]{0.707107, 0.212132, 0.}0}_{2} & \hdots {\color[rgb]{0.35, 0., 0.75}0}{\color[rgb]{0., 0.353553, 0.341768}0}{\color[rgb]{0.1, 0.5, 0.}1}{\color[rgb]{0.707107, 0.683537, 0.}0}{\color[rgb]{0.707107, 0.329983, 0.}0}{\color[rgb]{0.707107, 0.212132, 0.}0}_{2} & \hdots {\color[rgb]{0.35, 0., 0.75}0}{\color[rgb]{0., 0.353553, 0.341768}0}{\color[rgb]{0.1, 0.5, 0.}0}{\color[rgb]{0.707107, 0.683537, 0.}0}{\color[rgb]{0.707107, 0.329983, 0.}0}{\color[rgb]{0.707107, 0.212132, 0.}0}_{2} \\
M_{8}^{\Wit^\#} & \hdots {\color[rgb]{0.329983, 0., 0.707107}1}{\color[rgb]{0., 0., 0.}1}{\color[rgb]{0., 0., 0.}1}{\color[rgb]{0., 0., 0.}0}{\color[rgb]{0., 0., 0.}1}{\color[rgb]{0., 0., 0.}0}_{2} & \hdots {\color[rgb]{0.329983, 0., 0.707107}1}{\color[rgb]{0., 0., 0.}0}{\color[rgb]{0., 0., 0.}0}{\color[rgb]{0., 0., 0.}0}{\color[rgb]{0., 0., 0.}0}{\color[rgb]{0., 0., 0.}0}_{2} & \hdots {\color[rgb]{0.329983, 0., 0.707107}1}{\color[rgb]{0., 0., 0.}0}{\color[rgb]{0., 0., 0.}0}{\color[rgb]{0., 0., 0.}0}{\color[rgb]{0., 0., 0.}0}{\color[rgb]{0., 0., 0.}0}_{2} & \hdots {\color[rgb]{0.329983, 0., 0.707107}0}{\color[rgb]{0., 0., 0.}0}{\color[rgb]{0., 0., 0.}0}{\color[rgb]{0., 0., 0.}0}{\color[rgb]{0., 0., 0.}0}{\color[rgb]{0., 0., 0.}0}_{2} & \hdots {\color[rgb]{0.329983, 0., 0.707107}0}{\color[rgb]{0., 0., 0.}0}{\color[rgb]{0., 0., 0.}0}{\color[rgb]{0., 0., 0.}0}{\color[rgb]{0., 0., 0.}0}{\color[rgb]{0., 0., 0.}0}_{2} & \hdots {\color[rgb]{0.329983, 0., 0.707107}0}{\color[rgb]{0., 0., 0.}0}{\color[rgb]{0., 0., 0.}0}{\color[rgb]{0., 0., 0.}0}{\color[rgb]{0., 0., 0.}0}{\color[rgb]{0., 0., 0.}0}_{2} \\
M_{9}^{\Wit^\#} & \hdots {\color[rgb]{0.308671, 0., 0.661438}0}{\color[rgb]{0., 0.935414, 0.904234}0}{\color[rgb]{0.173205, 0.866025, 0.}0}{\color[rgb]{0.707107, 0.683537, 0.}0}{\color[rgb]{0.707107, 0.329983, 0.}0}{\color[rgb]{0.707107, 0.212132, 0.}0}_{2} & \hdots {\color[rgb]{0.308671, 0., 0.661438}0}{\color[rgb]{0., 0.935414, 0.904234}1}{\color[rgb]{0.173205, 0.866025, 0.}1}{\color[rgb]{0.707107, 0.683537, 0.}1}{\color[rgb]{0.707107, 0.329983, 0.}1}{\color[rgb]{0.707107, 0.212132, 0.}0}_{2} & \hdots {\color[rgb]{0.308671, 0., 0.661438}0}{\color[rgb]{0., 0.935414, 0.904234}1}{\color[rgb]{0.173205, 0.866025, 0.}1}{\color[rgb]{0.707107, 0.683537, 0.}1}{\color[rgb]{0.707107, 0.329983, 0.}1}{\color[rgb]{0.707107, 0.212132, 0.}0}_{2} & \hdots {\color[rgb]{0.308671, 0., 0.661438}1}{\color[rgb]{0., 0.935414, 0.904234}1}{\color[rgb]{0.173205, 0.866025, 0.}1}{\color[rgb]{0.707107, 0.683537, 0.}1}{\color[rgb]{0.707107, 0.329983, 0.}0}{\color[rgb]{0.707107, 0.212132, 0.}0}_{2} & \hdots {\color[rgb]{0.308671, 0., 0.661438}1}{\color[rgb]{0., 0.935414, 0.904234}1}{\color[rgb]{0.173205, 0.866025, 0.}1}{\color[rgb]{0.707107, 0.683537, 0.}0}{\color[rgb]{0.707107, 0.329983, 0.}0}{\color[rgb]{0.707107, 0.212132, 0.}0}_{2} & \hdots {\color[rgb]{0.308671, 0., 0.661438}1}{\color[rgb]{0., 0.935414, 0.904234}1}{\color[rgb]{0.173205, 0.866025, 0.}0}{\color[rgb]{0.707107, 0.683537, 0.}0}{\color[rgb]{0.707107, 0.329983, 0.}0}{\color[rgb]{0.707107, 0.212132, 0.}0}_{2} \\
M_{10}^{\Wit^\#} & \hdots {\color[rgb]{0.285774, 0., 0.612372}1}{\color[rgb]{0., 0.866025, 0.837158}0}{\color[rgb]{0.141421, 0.707107, 0.}1}{\color[rgb]{0., 0., 0.}0}{\color[rgb]{0., 0., 0.}1}{\color[rgb]{0., 0., 0.}0}_{2} & \hdots {\color[rgb]{0.285774, 0., 0.612372}0}{\color[rgb]{0., 0.866025, 0.837158}1}{\color[rgb]{0.141421, 0.707107, 0.}1}{\color[rgb]{0., 0., 0.}0}{\color[rgb]{0., 0., 0.}0}{\color[rgb]{0., 0., 0.}0}_{2} & \hdots {\color[rgb]{0.285774, 0., 0.612372}0}{\color[rgb]{0., 0.866025, 0.837158}1}{\color[rgb]{0.141421, 0.707107, 0.}1}{\color[rgb]{0., 0., 0.}0}{\color[rgb]{0., 0., 0.}0}{\color[rgb]{0., 0., 0.}0}_{2} & \hdots {\color[rgb]{0.285774, 0., 0.612372}1}{\color[rgb]{0., 0.866025, 0.837158}0}{\color[rgb]{0.141421, 0.707107, 0.}0}{\color[rgb]{0., 0., 0.}0}{\color[rgb]{0., 0., 0.}0}{\color[rgb]{0., 0., 0.}0}_{2} & \hdots {\color[rgb]{0.285774, 0., 0.612372}1}{\color[rgb]{0., 0.866025, 0.837158}1}{\color[rgb]{0.141421, 0.707107, 0.}0}{\color[rgb]{0., 0., 0.}0}{\color[rgb]{0., 0., 0.}0}{\color[rgb]{0., 0., 0.}0}_{2} & \hdots {\color[rgb]{0.285774, 0., 0.612372}0}{\color[rgb]{0., 0.866025, 0.837158}0}{\color[rgb]{0.141421, 0.707107, 0.}0}{\color[rgb]{0., 0., 0.}0}{\color[rgb]{0., 0., 0.}0}{\color[rgb]{0., 0., 0.}0}_{2} \\
M_{11}^{\Wit^\#} & \hdots {\color[rgb]{0.260875, 0., 0.559017}0}{\color[rgb]{0., 0.790569, 0.764217}0}{\color[rgb]{0.1, 0.5, 0.}0}{\color[rgb]{0.707107, 0.683537, 0.}0}{\color[rgb]{0.707107, 0.329983, 0.}0}{\color[rgb]{0.707107, 0.212132, 0.}0}_{2} & \hdots {\color[rgb]{0.260875, 0., 0.559017}0}{\color[rgb]{0., 0.790569, 0.764217}0}{\color[rgb]{0.1, 0.5, 0.}0}{\color[rgb]{0.707107, 0.683537, 0.}1}{\color[rgb]{0.707107, 0.329983, 0.}1}{\color[rgb]{0.707107, 0.212132, 0.}0}_{2} & \hdots {\color[rgb]{0.260875, 0., 0.559017}0}{\color[rgb]{0., 0.790569, 0.764217}0}{\color[rgb]{0.1, 0.5, 0.}0}{\color[rgb]{0.707107, 0.683537, 0.}1}{\color[rgb]{0.707107, 0.329983, 0.}1}{\color[rgb]{0.707107, 0.212132, 0.}0}_{2} & \hdots {\color[rgb]{0.260875, 0., 0.559017}1}{\color[rgb]{0., 0.790569, 0.764217}1}{\color[rgb]{0.1, 0.5, 0.}1}{\color[rgb]{0.707107, 0.683537, 0.}1}{\color[rgb]{0.707107, 0.329983, 0.}0}{\color[rgb]{0.707107, 0.212132, 0.}0}_{2} & \hdots {\color[rgb]{0.260875, 0., 0.559017}1}{\color[rgb]{0., 0.790569, 0.764217}0}{\color[rgb]{0.1, 0.5, 0.}1}{\color[rgb]{0.707107, 0.683537, 0.}0}{\color[rgb]{0.707107, 0.329983, 0.}0}{\color[rgb]{0.707107, 0.212132, 0.}0}_{2} & \hdots {\color[rgb]{0.260875, 0., 0.559017}1}{\color[rgb]{0., 0.790569, 0.764217}0}{\color[rgb]{0.1, 0.5, 0.}0}{\color[rgb]{0.707107, 0.683537, 0.}0}{\color[rgb]{0.707107, 0.329983, 0.}0}{\color[rgb]{0.707107, 0.212132, 0.}0}_{2} \\
M_{12}^{\Wit^\#} & \hdots {\color[rgb]{0.233333, 0., 0.5}0}{\color[rgb]{0., 0.707107, 0.683537}1}{\color[rgb]{0., 0., 0.}1}{\color[rgb]{0., 0., 0.}0}{\color[rgb]{0., 0., 0.}1}{\color[rgb]{0., 0., 0.}0}_{2} & \hdots {\color[rgb]{0.233333, 0., 0.5}0}{\color[rgb]{0., 0.707107, 0.683537}1}{\color[rgb]{0., 0., 0.}0}{\color[rgb]{0., 0., 0.}0}{\color[rgb]{0., 0., 0.}0}{\color[rgb]{0., 0., 0.}0}_{2} & \hdots {\color[rgb]{0.233333, 0., 0.5}0}{\color[rgb]{0., 0.707107, 0.683537}1}{\color[rgb]{0., 0., 0.}0}{\color[rgb]{0., 0., 0.}0}{\color[rgb]{0., 0., 0.}0}{\color[rgb]{0., 0., 0.}0}_{2} & \hdots {\color[rgb]{0.233333, 0., 0.5}0}{\color[rgb]{0., 0.707107, 0.683537}0}{\color[rgb]{0., 0., 0.}0}{\color[rgb]{0., 0., 0.}0}{\color[rgb]{0., 0., 0.}0}{\color[rgb]{0., 0., 0.}0}_{2} & \hdots {\color[rgb]{0.233333, 0., 0.5}1}{\color[rgb]{0., 0.707107, 0.683537}0}{\color[rgb]{0., 0., 0.}0}{\color[rgb]{0., 0., 0.}0}{\color[rgb]{0., 0., 0.}0}{\color[rgb]{0., 0., 0.}0}_{2} & \hdots {\color[rgb]{0.233333, 0., 0.5}0}{\color[rgb]{0., 0.707107, 0.683537}0}{\color[rgb]{0., 0., 0.}0}{\color[rgb]{0., 0., 0.}0}{\color[rgb]{0., 0., 0.}0}{\color[rgb]{0., 0., 0.}0}_{2} \\
M_{13}^{\Wit^\#} & \hdots {\color[rgb]{0.202073, 0., 0.433013}0}{\color[rgb]{0., 0.612372, 0.59196}0}{\color[rgb]{0.173205, 0.866025, 0.}0}{\color[rgb]{0.707107, 0.683537, 0.}0}{\color[rgb]{0.707107, 0.329983, 0.}0}{\color[rgb]{0.707107, 0.212132, 0.}0}_{2} & \hdots {\color[rgb]{0.202073, 0., 0.433013}1}{\color[rgb]{0., 0.612372, 0.59196}0}{\color[rgb]{0.173205, 0.866025, 0.}1}{\color[rgb]{0.707107, 0.683537, 0.}1}{\color[rgb]{0.707107, 0.329983, 0.}1}{\color[rgb]{0.707107, 0.212132, 0.}0}_{2} & \hdots {\color[rgb]{0.202073, 0., 0.433013}1}{\color[rgb]{0., 0.612372, 0.59196}0}{\color[rgb]{0.173205, 0.866025, 0.}1}{\color[rgb]{0.707107, 0.683537, 0.}1}{\color[rgb]{0.707107, 0.329983, 0.}1}{\color[rgb]{0.707107, 0.212132, 0.}0}_{2} & \hdots {\color[rgb]{0.202073, 0., 0.433013}1}{\color[rgb]{0., 0.612372, 0.59196}1}{\color[rgb]{0.173205, 0.866025, 0.}1}{\color[rgb]{0.707107, 0.683537, 0.}1}{\color[rgb]{0.707107, 0.329983, 0.}0}{\color[rgb]{0.707107, 0.212132, 0.}0}_{2} & \hdots {\color[rgb]{0.202073, 0., 0.433013}0}{\color[rgb]{0., 0.612372, 0.59196}1}{\color[rgb]{0.173205, 0.866025, 0.}1}{\color[rgb]{0.707107, 0.683537, 0.}0}{\color[rgb]{0.707107, 0.329983, 0.}0}{\color[rgb]{0.707107, 0.212132, 0.}0}_{2} & \hdots {\color[rgb]{0.202073, 0., 0.433013}0}{\color[rgb]{0., 0.612372, 0.59196}1}{\color[rgb]{0.173205, 0.866025, 0.}0}{\color[rgb]{0.707107, 0.683537, 0.}0}{\color[rgb]{0.707107, 0.329983, 0.}0}{\color[rgb]{0.707107, 0.212132, 0.}0}_{2} \\
M_{14}^{\Wit^\#} & \hdots {\color[rgb]{0.164992, 0., 0.353553}0}{\color[rgb]{0., 0.5, 0.483333}0}{\color[rgb]{0.141421, 0.707107, 0.}1}{\color[rgb]{0., 0., 0.}0}{\color[rgb]{0., 0., 0.}1}{\color[rgb]{0., 0., 0.}0}_{2} & \hdots {\color[rgb]{0.164992, 0., 0.353553}0}{\color[rgb]{0., 0.5, 0.483333}0}{\color[rgb]{0.141421, 0.707107, 0.}1}{\color[rgb]{0., 0., 0.}0}{\color[rgb]{0., 0., 0.}0}{\color[rgb]{0., 0., 0.}0}_{2} & \hdots {\color[rgb]{0.164992, 0., 0.353553}0}{\color[rgb]{0., 0.5, 0.483333}0}{\color[rgb]{0.141421, 0.707107, 0.}1}{\color[rgb]{0., 0., 0.}0}{\color[rgb]{0., 0., 0.}0}{\color[rgb]{0., 0., 0.}0}_{2} & \hdots {\color[rgb]{0.164992, 0., 0.353553}1}{\color[rgb]{0., 0.5, 0.483333}0}{\color[rgb]{0.141421, 0.707107, 0.}0}{\color[rgb]{0., 0., 0.}0}{\color[rgb]{0., 0., 0.}0}{\color[rgb]{0., 0., 0.}0}_{2} & \hdots {\color[rgb]{0.164992, 0., 0.353553}0}{\color[rgb]{0., 0.5, 0.483333}1}{\color[rgb]{0.141421, 0.707107, 0.}0}{\color[rgb]{0., 0., 0.}0}{\color[rgb]{0., 0., 0.}0}{\color[rgb]{0., 0., 0.}0}_{2} & \hdots {\color[rgb]{0.164992, 0., 0.353553}0}{\color[rgb]{0., 0.5, 0.483333}0}{\color[rgb]{0.141421, 0.707107, 0.}0}{\color[rgb]{0., 0., 0.}0}{\color[rgb]{0., 0., 0.}0}{\color[rgb]{0., 0., 0.}0}_{2} \\
M_{15}^{\Wit^\#} & \hdots {\color[rgb]{0.116667, 0., 0.25}0}{\color[rgb]{0., 0.353553, 0.341768}0}{\color[rgb]{0.1, 0.5, 0.}0}{\color[rgb]{0.707107, 0.683537, 0.}0}{\color[rgb]{0.707107, 0.329983, 0.}0}{\color[rgb]{0.707107, 0.212132, 0.}0}_{2} & \hdots {\color[rgb]{0.116667, 0., 0.25}0}{\color[rgb]{0., 0.353553, 0.341768}1}{\color[rgb]{0.1, 0.5, 0.}0}{\color[rgb]{0.707107, 0.683537, 0.}1}{\color[rgb]{0.707107, 0.329983, 0.}1}{\color[rgb]{0.707107, 0.212132, 0.}0}_{2} & \hdots {\color[rgb]{0.116667, 0., 0.25}0}{\color[rgb]{0., 0.353553, 0.341768}1}{\color[rgb]{0.1, 0.5, 0.}0}{\color[rgb]{0.707107, 0.683537, 0.}1}{\color[rgb]{0.707107, 0.329983, 0.}1}{\color[rgb]{0.707107, 0.212132, 0.}0}_{2} & \hdots {\color[rgb]{0.116667, 0., 0.25}1}{\color[rgb]{0., 0.353553, 0.341768}1}{\color[rgb]{0.1, 0.5, 0.}1}{\color[rgb]{0.707107, 0.683537, 0.}1}{\color[rgb]{0.707107, 0.329983, 0.}0}{\color[rgb]{0.707107, 0.212132, 0.}0}_{2} & \hdots {\color[rgb]{0.116667, 0., 0.25}0}{\color[rgb]{0., 0.353553, 0.341768}0}{\color[rgb]{0.1, 0.5, 0.}1}{\color[rgb]{0.707107, 0.683537, 0.}0}{\color[rgb]{0.707107, 0.329983, 0.}0}{\color[rgb]{0.707107, 0.212132, 0.}0}_{2} & \hdots {\color[rgb]{0.116667, 0., 0.25}0}{\color[rgb]{0., 0.353553, 0.341768}0}{\color[rgb]{0.1, 0.5, 0.}0}{\color[rgb]{0.707107, 0.683537, 0.}0}{\color[rgb]{0.707107, 0.329983, 0.}0}{\color[rgb]{0.707107, 0.212132, 0.}0}_{2} \\
M_{16}^{\Wit^\#} & \hdots {\color[rgb]{0., 0., 0.}1}{\color[rgb]{0., 0., 0.}1}{\color[rgb]{0., 0., 0.}1}{\color[rgb]{0., 0., 0.}0}{\color[rgb]{0., 0., 0.}1}{\color[rgb]{0., 0., 0.}0}_{2} & \hdots {\color[rgb]{0., 0., 0.}0}{\color[rgb]{0., 0., 0.}0}{\color[rgb]{0., 0., 0.}0}{\color[rgb]{0., 0., 0.}0}{\color[rgb]{0., 0., 0.}0}{\color[rgb]{0., 0., 0.}0}_{2} & \hdots {\color[rgb]{0., 0., 0.}0}{\color[rgb]{0., 0., 0.}0}{\color[rgb]{0., 0., 0.}0}{\color[rgb]{0., 0., 0.}0}{\color[rgb]{0., 0., 0.}0}{\color[rgb]{0., 0., 0.}0}_{2} & \hdots {\color[rgb]{0., 0., 0.}0}{\color[rgb]{0., 0., 0.}0}{\color[rgb]{0., 0., 0.}0}{\color[rgb]{0., 0., 0.}0}{\color[rgb]{0., 0., 0.}0}{\color[rgb]{0., 0., 0.}0}_{2} & \hdots {\color[rgb]{0., 0., 0.}0}{\color[rgb]{0., 0., 0.}0}{\color[rgb]{0., 0., 0.}0}{\color[rgb]{0., 0., 0.}0}{\color[rgb]{0., 0., 0.}0}{\color[rgb]{0., 0., 0.}0}_{2} & \hdots {\color[rgb]{0., 0., 0.}0}{\color[rgb]{0., 0., 0.}0}{\color[rgb]{0., 0., 0.}0}{\color[rgb]{0., 0., 0.}0}{\color[rgb]{0., 0., 0.}0}{\color[rgb]{0., 0., 0.}0}_{2} \\
M_{17}^{\Wit^\#} & \hdots {\color[rgb]{0.451848, 0., 0.968246}0}{\color[rgb]{0., 0.935414, 0.904234}0}{\color[rgb]{0.173205, 0.866025, 0.}0}{\color[rgb]{0.707107, 0.683537, 0.}0}{\color[rgb]{0.707107, 0.329983, 0.}0}{\color[rgb]{0.707107, 0.212132, 0.}0}_{2} & \hdots {\color[rgb]{0.451848, 0., 0.968246}1}{\color[rgb]{0., 0.935414, 0.904234}1}{\color[rgb]{0.173205, 0.866025, 0.}1}{\color[rgb]{0.707107, 0.683537, 0.}1}{\color[rgb]{0.707107, 0.329983, 0.}1}{\color[rgb]{0.707107, 0.212132, 0.}0}_{2} & \hdots {\color[rgb]{0.451848, 0., 0.968246}1}{\color[rgb]{0., 0.935414, 0.904234}1}{\color[rgb]{0.173205, 0.866025, 0.}1}{\color[rgb]{0.707107, 0.683537, 0.}1}{\color[rgb]{0.707107, 0.329983, 0.}1}{\color[rgb]{0.707107, 0.212132, 0.}0}_{2} & \hdots {\color[rgb]{0.451848, 0., 0.968246}1}{\color[rgb]{0., 0.935414, 0.904234}1}{\color[rgb]{0.173205, 0.866025, 0.}1}{\color[rgb]{0.707107, 0.683537, 0.}1}{\color[rgb]{0.707107, 0.329983, 0.}0}{\color[rgb]{0.707107, 0.212132, 0.}0}_{2} & \hdots {\color[rgb]{0.451848, 0., 0.968246}1}{\color[rgb]{0., 0.935414, 0.904234}1}{\color[rgb]{0.173205, 0.866025, 0.}1}{\color[rgb]{0.707107, 0.683537, 0.}0}{\color[rgb]{0.707107, 0.329983, 0.}0}{\color[rgb]{0.707107, 0.212132, 0.}0}_{2} & \hdots {\color[rgb]{0.451848, 0., 0.968246}1}{\color[rgb]{0., 0.935414, 0.904234}1}{\color[rgb]{0.173205, 0.866025, 0.}0}{\color[rgb]{0.707107, 0.683537, 0.}0}{\color[rgb]{0.707107, 0.329983, 0.}0}{\color[rgb]{0.707107, 0.212132, 0.}0}_{2} \\
M_{18}^{\Wit^\#} & \hdots {\color[rgb]{0.436527, 0., 0.935414}1}{\color[rgb]{0., 0.866025, 0.837158}0}{\color[rgb]{0.141421, 0.707107, 0.}1}{\color[rgb]{0., 0., 0.}0}{\color[rgb]{0., 0., 0.}1}{\color[rgb]{0., 0., 0.}0}_{2} & \hdots {\color[rgb]{0.436527, 0., 0.935414}1}{\color[rgb]{0., 0.866025, 0.837158}1}{\color[rgb]{0.141421, 0.707107, 0.}1}{\color[rgb]{0., 0., 0.}0}{\color[rgb]{0., 0., 0.}0}{\color[rgb]{0., 0., 0.}0}_{2} & \hdots {\color[rgb]{0.436527, 0., 0.935414}1}{\color[rgb]{0., 0.866025, 0.837158}1}{\color[rgb]{0.141421, 0.707107, 0.}1}{\color[rgb]{0., 0., 0.}0}{\color[rgb]{0., 0., 0.}0}{\color[rgb]{0., 0., 0.}0}_{2} & \hdots {\color[rgb]{0.436527, 0., 0.935414}1}{\color[rgb]{0., 0.866025, 0.837158}0}{\color[rgb]{0.141421, 0.707107, 0.}0}{\color[rgb]{0., 0., 0.}0}{\color[rgb]{0., 0., 0.}0}{\color[rgb]{0., 0., 0.}0}_{2} & \hdots {\color[rgb]{0.436527, 0., 0.935414}1}{\color[rgb]{0., 0.866025, 0.837158}1}{\color[rgb]{0.141421, 0.707107, 0.}0}{\color[rgb]{0., 0., 0.}0}{\color[rgb]{0., 0., 0.}0}{\color[rgb]{0., 0., 0.}0}_{2} & \hdots {\color[rgb]{0.436527, 0., 0.935414}0}{\color[rgb]{0., 0.866025, 0.837158}0}{\color[rgb]{0.141421, 0.707107, 0.}0}{\color[rgb]{0., 0., 0.}0}{\color[rgb]{0., 0., 0.}0}{\color[rgb]{0., 0., 0.}0}_{2} \\
M_{19}^{\Wit^\#} & \hdots {\color[rgb]{0.420648, 0., 0.901388}0}{\color[rgb]{0., 0.790569, 0.764217}0}{\color[rgb]{0.1, 0.5, 0.}0}{\color[rgb]{0.707107, 0.683537, 0.}0}{\color[rgb]{0.707107, 0.329983, 0.}0}{\color[rgb]{0.707107, 0.212132, 0.}0}_{2} & \hdots {\color[rgb]{0.420648, 0., 0.901388}1}{\color[rgb]{0., 0.790569, 0.764217}0}{\color[rgb]{0.1, 0.5, 0.}0}{\color[rgb]{0.707107, 0.683537, 0.}1}{\color[rgb]{0.707107, 0.329983, 0.}1}{\color[rgb]{0.707107, 0.212132, 0.}0}_{2} & \hdots {\color[rgb]{0.420648, 0., 0.901388}1}{\color[rgb]{0., 0.790569, 0.764217}0}{\color[rgb]{0.1, 0.5, 0.}0}{\color[rgb]{0.707107, 0.683537, 0.}1}{\color[rgb]{0.707107, 0.329983, 0.}1}{\color[rgb]{0.707107, 0.212132, 0.}0}_{2} & \hdots {\color[rgb]{0.420648, 0., 0.901388}1}{\color[rgb]{0., 0.790569, 0.764217}1}{\color[rgb]{0.1, 0.5, 0.}1}{\color[rgb]{0.707107, 0.683537, 0.}1}{\color[rgb]{0.707107, 0.329983, 0.}0}{\color[rgb]{0.707107, 0.212132, 0.}0}_{2} & \hdots {\color[rgb]{0.420648, 0., 0.901388}1}{\color[rgb]{0., 0.790569, 0.764217}0}{\color[rgb]{0.1, 0.5, 0.}1}{\color[rgb]{0.707107, 0.683537, 0.}0}{\color[rgb]{0.707107, 0.329983, 0.}0}{\color[rgb]{0.707107, 0.212132, 0.}0}_{2} & \hdots {\color[rgb]{0.420648, 0., 0.901388}1}{\color[rgb]{0., 0.790569, 0.764217}0}{\color[rgb]{0.1, 0.5, 0.}0}{\color[rgb]{0.707107, 0.683537, 0.}0}{\color[rgb]{0.707107, 0.329983, 0.}0}{\color[rgb]{0.707107, 0.212132, 0.}0}_{2} \\
M_{20}^{\Wit^\#} & \hdots {\color[rgb]{0.404145, 0., 0.866025}0}{\color[rgb]{0., 0.707107, 0.683537}1}{\color[rgb]{0., 0., 0.}1}{\color[rgb]{0., 0., 0.}0}{\color[rgb]{0., 0., 0.}1}{\color[rgb]{0., 0., 0.}0}_{2} & \hdots {\color[rgb]{0.404145, 0., 0.866025}1}{\color[rgb]{0., 0.707107, 0.683537}1}{\color[rgb]{0., 0., 0.}0}{\color[rgb]{0., 0., 0.}0}{\color[rgb]{0., 0., 0.}0}{\color[rgb]{0., 0., 0.}0}_{2} & \hdots {\color[rgb]{0.404145, 0., 0.866025}1}{\color[rgb]{0., 0.707107, 0.683537}1}{\color[rgb]{0., 0., 0.}0}{\color[rgb]{0., 0., 0.}0}{\color[rgb]{0., 0., 0.}0}{\color[rgb]{0., 0., 0.}0}_{2} & \hdots {\color[rgb]{0.404145, 0., 0.866025}0}{\color[rgb]{0., 0.707107, 0.683537}0}{\color[rgb]{0., 0., 0.}0}{\color[rgb]{0., 0., 0.}0}{\color[rgb]{0., 0., 0.}0}{\color[rgb]{0., 0., 0.}0}_{2} & \hdots {\color[rgb]{0.404145, 0., 0.866025}1}{\color[rgb]{0., 0.707107, 0.683537}0}{\color[rgb]{0., 0., 0.}0}{\color[rgb]{0., 0., 0.}0}{\color[rgb]{0., 0., 0.}0}{\color[rgb]{0., 0., 0.}0}_{2} & \hdots {\color[rgb]{0.404145, 0., 0.866025}0}{\color[rgb]{0., 0.707107, 0.683537}0}{\color[rgb]{0., 0., 0.}0}{\color[rgb]{0., 0., 0.}0}{\color[rgb]{0., 0., 0.}0}{\color[rgb]{0., 0., 0.}0}_{2} \\
M_{21}^{\Wit^\#} & \hdots {\color[rgb]{0.38694, 0., 0.829156}0}{\color[rgb]{0., 0.612372, 0.59196}0}{\color[rgb]{0.173205, 0.866025, 0.}0}{\color[rgb]{0.707107, 0.683537, 0.}0}{\color[rgb]{0.707107, 0.329983, 0.}0}{\color[rgb]{0.707107, 0.212132, 0.}0}_{2} & \hdots {\color[rgb]{0.38694, 0., 0.829156}0}{\color[rgb]{0., 0.612372, 0.59196}0}{\color[rgb]{0.173205, 0.866025, 0.}1}{\color[rgb]{0.707107, 0.683537, 0.}1}{\color[rgb]{0.707107, 0.329983, 0.}1}{\color[rgb]{0.707107, 0.212132, 0.}0}_{2} & \hdots {\color[rgb]{0.38694, 0., 0.829156}0}{\color[rgb]{0., 0.612372, 0.59196}0}{\color[rgb]{0.173205, 0.866025, 0.}1}{\color[rgb]{0.707107, 0.683537, 0.}1}{\color[rgb]{0.707107, 0.329983, 0.}1}{\color[rgb]{0.707107, 0.212132, 0.}0}_{2} & \hdots {\color[rgb]{0.38694, 0., 0.829156}1}{\color[rgb]{0., 0.612372, 0.59196}1}{\color[rgb]{0.173205, 0.866025, 0.}1}{\color[rgb]{0.707107, 0.683537, 0.}1}{\color[rgb]{0.707107, 0.329983, 0.}0}{\color[rgb]{0.707107, 0.212132, 0.}0}_{2} & \hdots {\color[rgb]{0.38694, 0., 0.829156}0}{\color[rgb]{0., 0.612372, 0.59196}1}{\color[rgb]{0.173205, 0.866025, 0.}1}{\color[rgb]{0.707107, 0.683537, 0.}0}{\color[rgb]{0.707107, 0.329983, 0.}0}{\color[rgb]{0.707107, 0.212132, 0.}0}_{2} & \hdots {\color[rgb]{0.38694, 0., 0.829156}0}{\color[rgb]{0., 0.612372, 0.59196}1}{\color[rgb]{0.173205, 0.866025, 0.}0}{\color[rgb]{0.707107, 0.683537, 0.}0}{\color[rgb]{0.707107, 0.329983, 0.}0}{\color[rgb]{0.707107, 0.212132, 0.}0}_{2} \\
M_{22}^{\Wit^\#} & \hdots {\color[rgb]{0.368932, 0., 0.790569}0}{\color[rgb]{0., 0.5, 0.483333}0}{\color[rgb]{0.141421, 0.707107, 0.}1}{\color[rgb]{0., 0., 0.}0}{\color[rgb]{0., 0., 0.}1}{\color[rgb]{0., 0., 0.}0}_{2} & \hdots {\color[rgb]{0.368932, 0., 0.790569}1}{\color[rgb]{0., 0.5, 0.483333}0}{\color[rgb]{0.141421, 0.707107, 0.}1}{\color[rgb]{0., 0., 0.}0}{\color[rgb]{0., 0., 0.}0}{\color[rgb]{0., 0., 0.}0}_{2} & \hdots {\color[rgb]{0.368932, 0., 0.790569}1}{\color[rgb]{0., 0.5, 0.483333}0}{\color[rgb]{0.141421, 0.707107, 0.}1}{\color[rgb]{0., 0., 0.}0}{\color[rgb]{0., 0., 0.}0}{\color[rgb]{0., 0., 0.}0}_{2} & \hdots {\color[rgb]{0.368932, 0., 0.790569}1}{\color[rgb]{0., 0.5, 0.483333}0}{\color[rgb]{0.141421, 0.707107, 0.}0}{\color[rgb]{0., 0., 0.}0}{\color[rgb]{0., 0., 0.}0}{\color[rgb]{0., 0., 0.}0}_{2} & \hdots {\color[rgb]{0.368932, 0., 0.790569}0}{\color[rgb]{0., 0.5, 0.483333}1}{\color[rgb]{0.141421, 0.707107, 0.}0}{\color[rgb]{0., 0., 0.}0}{\color[rgb]{0., 0., 0.}0}{\color[rgb]{0., 0., 0.}0}_{2} & \hdots {\color[rgb]{0.368932, 0., 0.790569}0}{\color[rgb]{0., 0.5, 0.483333}0}{\color[rgb]{0.141421, 0.707107, 0.}0}{\color[rgb]{0., 0., 0.}0}{\color[rgb]{0., 0., 0.}0}{\color[rgb]{0., 0., 0.}0}_{2} \\
M_{23}^{\Wit^\#} & \hdots {\color[rgb]{0.35, 0., 0.75}0}{\color[rgb]{0., 0.353553, 0.341768}0}{\color[rgb]{0.1, 0.5, 0.}0}{\color[rgb]{0.707107, 0.683537, 0.}0}{\color[rgb]{0.707107, 0.329983, 0.}0}{\color[rgb]{0.707107, 0.212132, 0.}0}_{2} & \hdots {\color[rgb]{0.35, 0., 0.75}1}{\color[rgb]{0., 0.353553, 0.341768}1}{\color[rgb]{0.1, 0.5, 0.}0}{\color[rgb]{0.707107, 0.683537, 0.}1}{\color[rgb]{0.707107, 0.329983, 0.}1}{\color[rgb]{0.707107, 0.212132, 0.}0}_{2} & \hdots {\color[rgb]{0.35, 0., 0.75}1}{\color[rgb]{0., 0.353553, 0.341768}1}{\color[rgb]{0.1, 0.5, 0.}0}{\color[rgb]{0.707107, 0.683537, 0.}1}{\color[rgb]{0.707107, 0.329983, 0.}1}{\color[rgb]{0.707107, 0.212132, 0.}0}_{2} & \hdots {\color[rgb]{0.35, 0., 0.75}1}{\color[rgb]{0., 0.353553, 0.341768}1}{\color[rgb]{0.1, 0.5, 0.}1}{\color[rgb]{0.707107, 0.683537, 0.}1}{\color[rgb]{0.707107, 0.329983, 0.}0}{\color[rgb]{0.707107, 0.212132, 0.}0}_{2} & \hdots {\color[rgb]{0.35, 0., 0.75}0}{\color[rgb]{0., 0.353553, 0.341768}0}{\color[rgb]{0.1, 0.5, 0.}1}{\color[rgb]{0.707107, 0.683537, 0.}0}{\color[rgb]{0.707107, 0.329983, 0.}0}{\color[rgb]{0.707107, 0.212132, 0.}0}_{2} & \hdots {\color[rgb]{0.35, 0., 0.75}0}{\color[rgb]{0., 0.353553, 0.341768}0}{\color[rgb]{0.1, 0.5, 0.}0}{\color[rgb]{0.707107, 0.683537, 0.}0}{\color[rgb]{0.707107, 0.329983, 0.}0}{\color[rgb]{0.707107, 0.212132, 0.}0}_{2} \\
M_{24}^{\Wit^\#} & \hdots {\color[rgb]{0.329983, 0., 0.707107}1}{\color[rgb]{0., 0., 0.}1}{\color[rgb]{0., 0., 0.}1}{\color[rgb]{0., 0., 0.}0}{\color[rgb]{0., 0., 0.}1}{\color[rgb]{0., 0., 0.}0}_{2} & \hdots {\color[rgb]{0.329983, 0., 0.707107}1}{\color[rgb]{0., 0., 0.}0}{\color[rgb]{0., 0., 0.}0}{\color[rgb]{0., 0., 0.}0}{\color[rgb]{0., 0., 0.}0}{\color[rgb]{0., 0., 0.}0}_{2} & \hdots {\color[rgb]{0.329983, 0., 0.707107}1}{\color[rgb]{0., 0., 0.}0}{\color[rgb]{0., 0., 0.}0}{\color[rgb]{0., 0., 0.}0}{\color[rgb]{0., 0., 0.}0}{\color[rgb]{0., 0., 0.}0}_{2} & \hdots {\color[rgb]{0.329983, 0., 0.707107}0}{\color[rgb]{0., 0., 0.}0}{\color[rgb]{0., 0., 0.}0}{\color[rgb]{0., 0., 0.}0}{\color[rgb]{0., 0., 0.}0}{\color[rgb]{0., 0., 0.}0}_{2} & \hdots {\color[rgb]{0.329983, 0., 0.707107}0}{\color[rgb]{0., 0., 0.}0}{\color[rgb]{0., 0., 0.}0}{\color[rgb]{0., 0., 0.}0}{\color[rgb]{0., 0., 0.}0}{\color[rgb]{0., 0., 0.}0}_{2} & \hdots {\color[rgb]{0.329983, 0., 0.707107}0}{\color[rgb]{0., 0., 0.}0}{\color[rgb]{0., 0., 0.}0}{\color[rgb]{0., 0., 0.}0}{\color[rgb]{0., 0., 0.}0}{\color[rgb]{0., 0., 0.}0}_{2} \\
M_{25}^{\Wit^\#} & \hdots {\color[rgb]{0.308671, 0., 0.661438}0}{\color[rgb]{0., 0.935414, 0.904234}0}{\color[rgb]{0.173205, 0.866025, 0.}0}{\color[rgb]{0.707107, 0.683537, 0.}0}{\color[rgb]{0.707107, 0.329983, 0.}0}{\color[rgb]{0.707107, 0.212132, 0.}0}_{2} & \hdots {\color[rgb]{0.308671, 0., 0.661438}0}{\color[rgb]{0., 0.935414, 0.904234}1}{\color[rgb]{0.173205, 0.866025, 0.}1}{\color[rgb]{0.707107, 0.683537, 0.}1}{\color[rgb]{0.707107, 0.329983, 0.}1}{\color[rgb]{0.707107, 0.212132, 0.}0}_{2} & \hdots {\color[rgb]{0.308671, 0., 0.661438}0}{\color[rgb]{0., 0.935414, 0.904234}1}{\color[rgb]{0.173205, 0.866025, 0.}1}{\color[rgb]{0.707107, 0.683537, 0.}1}{\color[rgb]{0.707107, 0.329983, 0.}1}{\color[rgb]{0.707107, 0.212132, 0.}0}_{2} & \hdots {\color[rgb]{0.308671, 0., 0.661438}1}{\color[rgb]{0., 0.935414, 0.904234}1}{\color[rgb]{0.173205, 0.866025, 0.}1}{\color[rgb]{0.707107, 0.683537, 0.}1}{\color[rgb]{0.707107, 0.329983, 0.}0}{\color[rgb]{0.707107, 0.212132, 0.}0}_{2} & \hdots {\color[rgb]{0.308671, 0., 0.661438}1}{\color[rgb]{0., 0.935414, 0.904234}1}{\color[rgb]{0.173205, 0.866025, 0.}1}{\color[rgb]{0.707107, 0.683537, 0.}0}{\color[rgb]{0.707107, 0.329983, 0.}0}{\color[rgb]{0.707107, 0.212132, 0.}0}_{2} & \hdots {\color[rgb]{0.308671, 0., 0.661438}1}{\color[rgb]{0., 0.935414, 0.904234}1}{\color[rgb]{0.173205, 0.866025, 0.}0}{\color[rgb]{0.707107, 0.683537, 0.}0}{\color[rgb]{0.707107, 0.329983, 0.}0}{\color[rgb]{0.707107, 0.212132, 0.}0}_{2} \\
M_{26}^{\Wit^\#} & \hdots {\color[rgb]{0.285774, 0., 0.612372}1}{\color[rgb]{0., 0.866025, 0.837158}0}{\color[rgb]{0.141421, 0.707107, 0.}1}{\color[rgb]{0., 0., 0.}0}{\color[rgb]{0., 0., 0.}1}{\color[rgb]{0., 0., 0.}0}_{2} & \hdots {\color[rgb]{0.285774, 0., 0.612372}0}{\color[rgb]{0., 0.866025, 0.837158}1}{\color[rgb]{0.141421, 0.707107, 0.}1}{\color[rgb]{0., 0., 0.}0}{\color[rgb]{0., 0., 0.}0}{\color[rgb]{0., 0., 0.}0}_{2} & \hdots {\color[rgb]{0.285774, 0., 0.612372}0}{\color[rgb]{0., 0.866025, 0.837158}1}{\color[rgb]{0.141421, 0.707107, 0.}1}{\color[rgb]{0., 0., 0.}0}{\color[rgb]{0., 0., 0.}0}{\color[rgb]{0., 0., 0.}0}_{2} & \hdots {\color[rgb]{0.285774, 0., 0.612372}1}{\color[rgb]{0., 0.866025, 0.837158}0}{\color[rgb]{0.141421, 0.707107, 0.}0}{\color[rgb]{0., 0., 0.}0}{\color[rgb]{0., 0., 0.}0}{\color[rgb]{0., 0., 0.}0}_{2} & \hdots {\color[rgb]{0.285774, 0., 0.612372}1}{\color[rgb]{0., 0.866025, 0.837158}1}{\color[rgb]{0.141421, 0.707107, 0.}0}{\color[rgb]{0., 0., 0.}0}{\color[rgb]{0., 0., 0.}0}{\color[rgb]{0., 0., 0.}0}_{2} & \hdots {\color[rgb]{0.285774, 0., 0.612372}0}{\color[rgb]{0., 0.866025, 0.837158}0}{\color[rgb]{0.141421, 0.707107, 0.}0}{\color[rgb]{0., 0., 0.}0}{\color[rgb]{0., 0., 0.}0}{\color[rgb]{0., 0., 0.}0}_{2} \\
M_{27}^{\Wit^\#} & \hdots {\color[rgb]{0.260875, 0., 0.559017}0}{\color[rgb]{0., 0.790569, 0.764217}0}{\color[rgb]{0.1, 0.5, 0.}0}{\color[rgb]{0.707107, 0.683537, 0.}0}{\color[rgb]{0.707107, 0.329983, 0.}0}{\color[rgb]{0.707107, 0.212132, 0.}0}_{2} & \hdots {\color[rgb]{0.260875, 0., 0.559017}0}{\color[rgb]{0., 0.790569, 0.764217}0}{\color[rgb]{0.1, 0.5, 0.}0}{\color[rgb]{0.707107, 0.683537, 0.}1}{\color[rgb]{0.707107, 0.329983, 0.}1}{\color[rgb]{0.707107, 0.212132, 0.}0}_{2} & \hdots {\color[rgb]{0.260875, 0., 0.559017}0}{\color[rgb]{0., 0.790569, 0.764217}0}{\color[rgb]{0.1, 0.5, 0.}0}{\color[rgb]{0.707107, 0.683537, 0.}1}{\color[rgb]{0.707107, 0.329983, 0.}1}{\color[rgb]{0.707107, 0.212132, 0.}0}_{2} & \hdots {\color[rgb]{0.260875, 0., 0.559017}1}{\color[rgb]{0., 0.790569, 0.764217}1}{\color[rgb]{0.1, 0.5, 0.}1}{\color[rgb]{0.707107, 0.683537, 0.}1}{\color[rgb]{0.707107, 0.329983, 0.}0}{\color[rgb]{0.707107, 0.212132, 0.}0}_{2} & \hdots {\color[rgb]{0.260875, 0., 0.559017}1}{\color[rgb]{0., 0.790569, 0.764217}0}{\color[rgb]{0.1, 0.5, 0.}1}{\color[rgb]{0.707107, 0.683537, 0.}0}{\color[rgb]{0.707107, 0.329983, 0.}0}{\color[rgb]{0.707107, 0.212132, 0.}0}_{2} & \hdots {\color[rgb]{0.260875, 0., 0.559017}1}{\color[rgb]{0., 0.790569, 0.764217}0}{\color[rgb]{0.1, 0.5, 0.}0}{\color[rgb]{0.707107, 0.683537, 0.}0}{\color[rgb]{0.707107, 0.329983, 0.}0}{\color[rgb]{0.707107, 0.212132, 0.}0}_{2} \\
M_{28}^{\Wit^\#} & \hdots {\color[rgb]{0.233333, 0., 0.5}0}{\color[rgb]{0., 0.707107, 0.683537}1}{\color[rgb]{0., 0., 0.}1}{\color[rgb]{0., 0., 0.}0}{\color[rgb]{0., 0., 0.}1}{\color[rgb]{0., 0., 0.}0}_{2} & \hdots {\color[rgb]{0.233333, 0., 0.5}0}{\color[rgb]{0., 0.707107, 0.683537}1}{\color[rgb]{0., 0., 0.}0}{\color[rgb]{0., 0., 0.}0}{\color[rgb]{0., 0., 0.}0}{\color[rgb]{0., 0., 0.}0}_{2} & \hdots {\color[rgb]{0.233333, 0., 0.5}0}{\color[rgb]{0., 0.707107, 0.683537}1}{\color[rgb]{0., 0., 0.}0}{\color[rgb]{0., 0., 0.}0}{\color[rgb]{0., 0., 0.}0}{\color[rgb]{0., 0., 0.}0}_{2} & \hdots {\color[rgb]{0.233333, 0., 0.5}0}{\color[rgb]{0., 0.707107, 0.683537}0}{\color[rgb]{0., 0., 0.}0}{\color[rgb]{0., 0., 0.}0}{\color[rgb]{0., 0., 0.}0}{\color[rgb]{0., 0., 0.}0}_{2} & \hdots {\color[rgb]{0.233333, 0., 0.5}1}{\color[rgb]{0., 0.707107, 0.683537}0}{\color[rgb]{0., 0., 0.}0}{\color[rgb]{0., 0., 0.}0}{\color[rgb]{0., 0., 0.}0}{\color[rgb]{0., 0., 0.}0}_{2} & \hdots {\color[rgb]{0.233333, 0., 0.5}0}{\color[rgb]{0., 0.707107, 0.683537}0}{\color[rgb]{0., 0., 0.}0}{\color[rgb]{0., 0., 0.}0}{\color[rgb]{0., 0., 0.}0}{\color[rgb]{0., 0., 0.}0}_{2} \\
M_{29}^{\Wit^\#} & \hdots {\color[rgb]{0.202073, 0., 0.433013}0}{\color[rgb]{0., 0.612372, 0.59196}0}{\color[rgb]{0.173205, 0.866025, 0.}0}{\color[rgb]{0.707107, 0.683537, 0.}0}{\color[rgb]{0.707107, 0.329983, 0.}0}{\color[rgb]{0.707107, 0.212132, 0.}0}_{2} & \hdots {\color[rgb]{0.202073, 0., 0.433013}1}{\color[rgb]{0., 0.612372, 0.59196}0}{\color[rgb]{0.173205, 0.866025, 0.}1}{\color[rgb]{0.707107, 0.683537, 0.}1}{\color[rgb]{0.707107, 0.329983, 0.}1}{\color[rgb]{0.707107, 0.212132, 0.}0}_{2} & \hdots {\color[rgb]{0.202073, 0., 0.433013}1}{\color[rgb]{0., 0.612372, 0.59196}0}{\color[rgb]{0.173205, 0.866025, 0.}1}{\color[rgb]{0.707107, 0.683537, 0.}1}{\color[rgb]{0.707107, 0.329983, 0.}1}{\color[rgb]{0.707107, 0.212132, 0.}0}_{2} & \hdots {\color[rgb]{0.202073, 0., 0.433013}1}{\color[rgb]{0., 0.612372, 0.59196}1}{\color[rgb]{0.173205, 0.866025, 0.}1}{\color[rgb]{0.707107, 0.683537, 0.}1}{\color[rgb]{0.707107, 0.329983, 0.}0}{\color[rgb]{0.707107, 0.212132, 0.}0}_{2} & \hdots {\color[rgb]{0.202073, 0., 0.433013}0}{\color[rgb]{0., 0.612372, 0.59196}1}{\color[rgb]{0.173205, 0.866025, 0.}1}{\color[rgb]{0.707107, 0.683537, 0.}0}{\color[rgb]{0.707107, 0.329983, 0.}0}{\color[rgb]{0.707107, 0.212132, 0.}0}_{2} & \hdots {\color[rgb]{0.202073, 0., 0.433013}0}{\color[rgb]{0., 0.612372, 0.59196}1}{\color[rgb]{0.173205, 0.866025, 0.}0}{\color[rgb]{0.707107, 0.683537, 0.}0}{\color[rgb]{0.707107, 0.329983, 0.}0}{\color[rgb]{0.707107, 0.212132, 0.}0}_{2} \\
M_{30}^{\Wit^\#} & \hdots {\color[rgb]{0.164992, 0., 0.353553}0}{\color[rgb]{0., 0.5, 0.483333}0}{\color[rgb]{0.141421, 0.707107, 0.}1}{\color[rgb]{0., 0., 0.}0}{\color[rgb]{0., 0., 0.}1}{\color[rgb]{0., 0., 0.}0}_{2} & \hdots {\color[rgb]{0.164992, 0., 0.353553}0}{\color[rgb]{0., 0.5, 0.483333}0}{\color[rgb]{0.141421, 0.707107, 0.}1}{\color[rgb]{0., 0., 0.}0}{\color[rgb]{0., 0., 0.}0}{\color[rgb]{0., 0., 0.}0}_{2} & \hdots {\color[rgb]{0.164992, 0., 0.353553}0}{\color[rgb]{0., 0.5, 0.483333}0}{\color[rgb]{0.141421, 0.707107, 0.}1}{\color[rgb]{0., 0., 0.}0}{\color[rgb]{0., 0., 0.}0}{\color[rgb]{0., 0., 0.}0}_{2} & \hdots {\color[rgb]{0.164992, 0., 0.353553}1}{\color[rgb]{0., 0.5, 0.483333}0}{\color[rgb]{0.141421, 0.707107, 0.}0}{\color[rgb]{0., 0., 0.}0}{\color[rgb]{0., 0., 0.}0}{\color[rgb]{0., 0., 0.}0}_{2} & \hdots {\color[rgb]{0.164992, 0., 0.353553}0}{\color[rgb]{0., 0.5, 0.483333}1}{\color[rgb]{0.141421, 0.707107, 0.}0}{\color[rgb]{0., 0., 0.}0}{\color[rgb]{0., 0., 0.}0}{\color[rgb]{0., 0., 0.}0}_{2} & \hdots {\color[rgb]{0.164992, 0., 0.353553}0}{\color[rgb]{0., 0.5, 0.483333}0}{\color[rgb]{0.141421, 0.707107, 0.}0}{\color[rgb]{0., 0., 0.}0}{\color[rgb]{0., 0., 0.}0}{\color[rgb]{0., 0., 0.}0}_{2} \\
M_{31}^{\Wit^\#} & \hdots {\color[rgb]{0.116667, 0., 0.25}0}{\color[rgb]{0., 0.353553, 0.341768}0}{\color[rgb]{0.1, 0.5, 0.}0}{\color[rgb]{0.707107, 0.683537, 0.}0}{\color[rgb]{0.707107, 0.329983, 0.}0}{\color[rgb]{0.707107, 0.212132, 0.}0}_{2} & \hdots {\color[rgb]{0.116667, 0., 0.25}0}{\color[rgb]{0., 0.353553, 0.341768}1}{\color[rgb]{0.1, 0.5, 0.}0}{\color[rgb]{0.707107, 0.683537, 0.}1}{\color[rgb]{0.707107, 0.329983, 0.}1}{\color[rgb]{0.707107, 0.212132, 0.}0}_{2} & \hdots {\color[rgb]{0.116667, 0., 0.25}0}{\color[rgb]{0., 0.353553, 0.341768}1}{\color[rgb]{0.1, 0.5, 0.}0}{\color[rgb]{0.707107, 0.683537, 0.}1}{\color[rgb]{0.707107, 0.329983, 0.}1}{\color[rgb]{0.707107, 0.212132, 0.}0}_{2} & \hdots {\color[rgb]{0.116667, 0., 0.25}1}{\color[rgb]{0., 0.353553, 0.341768}1}{\color[rgb]{0.1, 0.5, 0.}1}{\color[rgb]{0.707107, 0.683537, 0.}1}{\color[rgb]{0.707107, 0.329983, 0.}0}{\color[rgb]{0.707107, 0.212132, 0.}0}_{2} & \hdots {\color[rgb]{0.116667, 0., 0.25}0}{\color[rgb]{0., 0.353553, 0.341768}0}{\color[rgb]{0.1, 0.5, 0.}1}{\color[rgb]{0.707107, 0.683537, 0.}0}{\color[rgb]{0.707107, 0.329983, 0.}0}{\color[rgb]{0.707107, 0.212132, 0.}0}_{2} & \hdots {\color[rgb]{0.116667, 0., 0.25}0}{\color[rgb]{0., 0.353553, 0.341768}0}{\color[rgb]{0.1, 0.5, 0.}0}{\color[rgb]{0.707107, 0.683537, 0.}0}{\color[rgb]{0.707107, 0.329983, 0.}0}{\color[rgb]{0.707107, 0.212132, 0.}0}_{2} \\
\end{array}$

%% file: todd.3.data.tex
$
\begin{array}{r|ccccccc}
& t^0 & t^{-1} & t^{-2} & t^{-3} & t^{-4} & t^{-5} & t^{-6} \\
\hline
M_{1}^{\Td^\#} & \cdots {\color[rgb]{0., 0.321977, 0.990697}0}{\color[rgb]{0.194365, 0.971825, 0.}0}{\color[rgb]{0.912871, 0.616188, 0.}0}{\color[rgb]{0.707107, 0.212132, 0.}0}_{3} & \cdots {\color[rgb]{0., 0.321977, 0.990697}2}{\color[rgb]{0.194365, 0.971825, 0.}2}{\color[rgb]{0.912871, 0.616188, 0.}2}{\color[rgb]{0.707107, 0.212132, 0.}0}_{3} & \cdots {\color[rgb]{0., 0.321977, 0.990697}0}{\color[rgb]{0.194365, 0.971825, 0.}0}{\color[rgb]{0.912871, 0.616188, 0.}0}{\color[rgb]{0.707107, 0.212132, 0.}0}_{3} & \cdots {\color[rgb]{0., 0.321977, 0.990697}0}{\color[rgb]{0.194365, 0.971825, 0.}0}{\color[rgb]{0.912871, 0.616188, 0.}1}{\color[rgb]{0.707107, 0.212132, 0.}0}_{3} & \cdots {\color[rgb]{0., 0.321977, 0.990697}0}{\color[rgb]{0.194365, 0.971825, 0.}1}{\color[rgb]{0.912871, 0.616188, 0.}0}{\color[rgb]{0.707107, 0.212132, 0.}0}_{3} & \cdots {\color[rgb]{0., 0.321977, 0.990697}0}{\color[rgb]{0.194365, 0.971825, 0.}2}{\color[rgb]{0.912871, 0.616188, 0.}0}{\color[rgb]{0.707107, 0.212132, 0.}0}_{3} & \cdots {\color[rgb]{0., 0.321977, 0.990697}1}{\color[rgb]{0.194365, 0.971825, 0.}0}{\color[rgb]{0.912871, 0.616188, 0.}0}{\color[rgb]{0.707107, 0.212132, 0.}0}_{3} \\
M_{2}^{\Td^\#} & \cdots {\color[rgb]{0., 0.318925, 0.981307}1}{\color[rgb]{0.188562, 0.942809, 0.}1}{\color[rgb]{0.816497, 0.551135, 0.}0}{\color[rgb]{0., 0., 0.}2}_{3} & \cdots {\color[rgb]{0., 0.318925, 0.981307}2}{\color[rgb]{0.188562, 0.942809, 0.}1}{\color[rgb]{0.816497, 0.551135, 0.}1}{\color[rgb]{0., 0., 0.}0}_{3} & \cdots {\color[rgb]{0., 0.318925, 0.981307}0}{\color[rgb]{0.188562, 0.942809, 0.}1}{\color[rgb]{0.816497, 0.551135, 0.}2}{\color[rgb]{0., 0., 0.}0}_{3} & \cdots {\color[rgb]{0., 0.318925, 0.981307}1}{\color[rgb]{0.188562, 0.942809, 0.}2}{\color[rgb]{0.816497, 0.551135, 0.}0}{\color[rgb]{0., 0., 0.}0}_{3} & \cdots {\color[rgb]{0., 0.318925, 0.981307}2}{\color[rgb]{0.188562, 0.942809, 0.}0}{\color[rgb]{0.816497, 0.551135, 0.}0}{\color[rgb]{0., 0., 0.}0}_{3} & \cdots {\color[rgb]{0., 0.318925, 0.981307}1}{\color[rgb]{0.188562, 0.942809, 0.}0}{\color[rgb]{0.816497, 0.551135, 0.}0}{\color[rgb]{0., 0., 0.}0}_{3} & \cdots {\color[rgb]{0., 0.318925, 0.981307}1}{\color[rgb]{0.188562, 0.942809, 0.}1}{\color[rgb]{0.816497, 0.551135, 0.}0}{\color[rgb]{0., 0., 0.}0}_{3} \\
M_{3}^{\Td^\#} & \cdots {\color[rgb]{0., 0.315843, 0.971825}0}{\color[rgb]{0.182574, 0.912871, 0.}0}{\color[rgb]{0.707107, 0.477297, 0.}0}{\color[rgb]{0.707107, 0.212132, 0.}0}_{3} & \cdots {\color[rgb]{0., 0.315843, 0.971825}0}{\color[rgb]{0.182574, 0.912871, 0.}2}{\color[rgb]{0.707107, 0.477297, 0.}0}{\color[rgb]{0.707107, 0.212132, 0.}0}_{3} & \cdots {\color[rgb]{0., 0.315843, 0.971825}1}{\color[rgb]{0.182574, 0.912871, 0.}0}{\color[rgb]{0.707107, 0.477297, 0.}0}{\color[rgb]{0.707107, 0.212132, 0.}0}_{3} & \cdots {\color[rgb]{0., 0.315843, 0.971825}1}{\color[rgb]{0.182574, 0.912871, 0.}1}{\color[rgb]{0.707107, 0.477297, 0.}0}{\color[rgb]{0.707107, 0.212132, 0.}0}_{3} & \cdots {\color[rgb]{0., 0.315843, 0.971825}0}{\color[rgb]{0.182574, 0.912871, 0.}0}{\color[rgb]{0.707107, 0.477297, 0.}0}{\color[rgb]{0.707107, 0.212132, 0.}0}_{3} & \cdots {\color[rgb]{0., 0.315843, 0.971825}0}{\color[rgb]{0.182574, 0.912871, 0.}0}{\color[rgb]{0.707107, 0.477297, 0.}0}{\color[rgb]{0.707107, 0.212132, 0.}0}_{3} & \cdots {\color[rgb]{0., 0.315843, 0.971825}0}{\color[rgb]{0.182574, 0.912871, 0.}0}{\color[rgb]{0.707107, 0.477297, 0.}0}{\color[rgb]{0.707107, 0.212132, 0.}0}_{3} \\
M_{4}^{\Td^\#} & \cdots {\color[rgb]{0., 0.312731, 0.96225}1}{\color[rgb]{0.176383, 0.881917, 0.}0}{\color[rgb]{0.57735, 0.389711, 0.}2}{\color[rgb]{0., 0., 0.}2}_{3} & \cdots {\color[rgb]{0., 0.312731, 0.96225}2}{\color[rgb]{0.176383, 0.881917, 0.}1}{\color[rgb]{0.57735, 0.389711, 0.}2}{\color[rgb]{0., 0., 0.}0}_{3} & \cdots {\color[rgb]{0., 0.312731, 0.96225}0}{\color[rgb]{0.176383, 0.881917, 0.}0}{\color[rgb]{0.57735, 0.389711, 0.}1}{\color[rgb]{0., 0., 0.}0}_{3} & \cdots {\color[rgb]{0., 0.312731, 0.96225}0}{\color[rgb]{0.176383, 0.881917, 0.}2}{\color[rgb]{0.57735, 0.389711, 0.}0}{\color[rgb]{0., 0., 0.}0}_{3} & \cdots {\color[rgb]{0., 0.312731, 0.96225}2}{\color[rgb]{0.176383, 0.881917, 0.}2}{\color[rgb]{0.57735, 0.389711, 0.}0}{\color[rgb]{0., 0., 0.}0}_{3} & \cdots {\color[rgb]{0., 0.312731, 0.96225}0}{\color[rgb]{0.176383, 0.881917, 0.}0}{\color[rgb]{0.57735, 0.389711, 0.}0}{\color[rgb]{0., 0., 0.}0}_{3} & \cdots {\color[rgb]{0., 0.312731, 0.96225}0}{\color[rgb]{0.176383, 0.881917, 0.}0}{\color[rgb]{0.57735, 0.389711, 0.}0}{\color[rgb]{0., 0., 0.}0}_{3} \\
M_{5}^{\Td^\#} & \cdots {\color[rgb]{0., 0.309588, 0.952579}0}{\color[rgb]{0.169967, 0.849837, 0.}0}{\color[rgb]{0.408248, 0.275568, 0.}0}{\color[rgb]{0.707107, 0.212132, 0.}0}_{3} & \cdots {\color[rgb]{0., 0.309588, 0.952579}1}{\color[rgb]{0.169967, 0.849837, 0.}0}{\color[rgb]{0.408248, 0.275568, 0.}1}{\color[rgb]{0.707107, 0.212132, 0.}0}_{3} & \cdots {\color[rgb]{0., 0.309588, 0.952579}1}{\color[rgb]{0.169967, 0.849837, 0.}1}{\color[rgb]{0.408248, 0.275568, 0.}0}{\color[rgb]{0.707107, 0.212132, 0.}0}_{3} & \cdots {\color[rgb]{0., 0.309588, 0.952579}1}{\color[rgb]{0.169967, 0.849837, 0.}1}{\color[rgb]{0.408248, 0.275568, 0.}2}{\color[rgb]{0.707107, 0.212132, 0.}0}_{3} & \cdots {\color[rgb]{0., 0.309588, 0.952579}2}{\color[rgb]{0.169967, 0.849837, 0.}1}{\color[rgb]{0.408248, 0.275568, 0.}0}{\color[rgb]{0.707107, 0.212132, 0.}0}_{3} & \cdots {\color[rgb]{0., 0.309588, 0.952579}2}{\color[rgb]{0.169967, 0.849837, 0.}2}{\color[rgb]{0.408248, 0.275568, 0.}0}{\color[rgb]{0.707107, 0.212132, 0.}0}_{3} & \cdots {\color[rgb]{0., 0.309588, 0.952579}0}{\color[rgb]{0.169967, 0.849837, 0.}0}{\color[rgb]{0.408248, 0.275568, 0.}0}{\color[rgb]{0.707107, 0.212132, 0.}0}_{3} \\
M_{6}^{\Td^\#} & \cdots {\color[rgb]{0., 0.306413, 0.942809}2}{\color[rgb]{0.163299, 0.816497, 0.}2}{\color[rgb]{0., 0., 0.}1}{\color[rgb]{0., 0., 0.}2}_{3} & \cdots {\color[rgb]{0., 0.306413, 0.942809}1}{\color[rgb]{0.163299, 0.816497, 0.}1}{\color[rgb]{0., 0., 0.}0}{\color[rgb]{0., 0., 0.}0}_{3} & \cdots {\color[rgb]{0., 0.306413, 0.942809}0}{\color[rgb]{0.163299, 0.816497, 0.}2}{\color[rgb]{0., 0., 0.}0}{\color[rgb]{0., 0., 0.}0}_{3} & \cdots {\color[rgb]{0., 0.306413, 0.942809}0}{\color[rgb]{0.163299, 0.816497, 0.}0}{\color[rgb]{0., 0., 0.}0}{\color[rgb]{0., 0., 0.}0}_{3} & \cdots {\color[rgb]{0., 0.306413, 0.942809}2}{\color[rgb]{0.163299, 0.816497, 0.}0}{\color[rgb]{0., 0., 0.}0}{\color[rgb]{0., 0., 0.}0}_{3} & \cdots {\color[rgb]{0., 0.306413, 0.942809}0}{\color[rgb]{0.163299, 0.816497, 0.}0}{\color[rgb]{0., 0., 0.}0}{\color[rgb]{0., 0., 0.}0}_{3} & \cdots {\color[rgb]{0., 0.306413, 0.942809}2}{\color[rgb]{0.163299, 0.816497, 0.}0}{\color[rgb]{0., 0., 0.}0}{\color[rgb]{0., 0., 0.}0}_{3} \\
M_{7}^{\Td^\#} & \cdots {\color[rgb]{0., 0.303204, 0.932936}0}{\color[rgb]{0.156347, 0.781736, 0.}0}{\color[rgb]{0.912871, 0.616188, 0.}0}{\color[rgb]{0.707107, 0.212132, 0.}0}_{3} & \cdots {\color[rgb]{0., 0.303204, 0.932936}2}{\color[rgb]{0.156347, 0.781736, 0.}0}{\color[rgb]{0.912871, 0.616188, 0.}2}{\color[rgb]{0.707107, 0.212132, 0.}0}_{3} & \cdots {\color[rgb]{0., 0.303204, 0.932936}1}{\color[rgb]{0.156347, 0.781736, 0.}0}{\color[rgb]{0.912871, 0.616188, 0.}0}{\color[rgb]{0.707107, 0.212132, 0.}0}_{3} & \cdots {\color[rgb]{0., 0.303204, 0.932936}2}{\color[rgb]{0.156347, 0.781736, 0.}2}{\color[rgb]{0.912871, 0.616188, 0.}1}{\color[rgb]{0.707107, 0.212132, 0.}0}_{3} & \cdots {\color[rgb]{0., 0.303204, 0.932936}1}{\color[rgb]{0.156347, 0.781736, 0.}1}{\color[rgb]{0.912871, 0.616188, 0.}0}{\color[rgb]{0.707107, 0.212132, 0.}0}_{3} & \cdots {\color[rgb]{0., 0.303204, 0.932936}2}{\color[rgb]{0.156347, 0.781736, 0.}2}{\color[rgb]{0.912871, 0.616188, 0.}0}{\color[rgb]{0.707107, 0.212132, 0.}0}_{3} & \cdots {\color[rgb]{0., 0.303204, 0.932936}1}{\color[rgb]{0.156347, 0.781736, 0.}0}{\color[rgb]{0.912871, 0.616188, 0.}0}{\color[rgb]{0.707107, 0.212132, 0.}0}_{3} \\
M_{8}^{\Td^\#} & \cdots {\color[rgb]{0., 0.299961, 0.922958}2}{\color[rgb]{0.149071, 0.745356, 0.}0}{\color[rgb]{0.816497, 0.551135, 0.}0}{\color[rgb]{0., 0., 0.}2}_{3} & \cdots {\color[rgb]{0., 0.299961, 0.922958}2}{\color[rgb]{0.149071, 0.745356, 0.}2}{\color[rgb]{0.816497, 0.551135, 0.}1}{\color[rgb]{0., 0., 0.}0}_{3} & \cdots {\color[rgb]{0., 0.299961, 0.922958}1}{\color[rgb]{0.149071, 0.745356, 0.}0}{\color[rgb]{0.816497, 0.551135, 0.}2}{\color[rgb]{0., 0., 0.}0}_{3} & \cdots {\color[rgb]{0., 0.299961, 0.922958}2}{\color[rgb]{0.149071, 0.745356, 0.}2}{\color[rgb]{0.816497, 0.551135, 0.}0}{\color[rgb]{0., 0., 0.}0}_{3} & \cdots {\color[rgb]{0., 0.299961, 0.922958}0}{\color[rgb]{0.149071, 0.745356, 0.}0}{\color[rgb]{0.816497, 0.551135, 0.}0}{\color[rgb]{0., 0., 0.}0}_{3} & \cdots {\color[rgb]{0., 0.299961, 0.922958}1}{\color[rgb]{0.149071, 0.745356, 0.}0}{\color[rgb]{0.816497, 0.551135, 0.}0}{\color[rgb]{0., 0., 0.}0}_{3} & \cdots {\color[rgb]{0., 0.299961, 0.922958}1}{\color[rgb]{0.149071, 0.745356, 0.}1}{\color[rgb]{0.816497, 0.551135, 0.}0}{\color[rgb]{0., 0., 0.}0}_{3} \\
M_{9}^{\Td^\#} & \cdots {\color[rgb]{0., 0.296683, 0.912871}0}{\color[rgb]{0.141421, 0.707107, 0.}0}{\color[rgb]{0.707107, 0.477297, 0.}0}{\color[rgb]{0.707107, 0.212132, 0.}0}_{3} & \cdots {\color[rgb]{0., 0.296683, 0.912871}2}{\color[rgb]{0.141421, 0.707107, 0.}0}{\color[rgb]{0.707107, 0.477297, 0.}0}{\color[rgb]{0.707107, 0.212132, 0.}0}_{3} & \cdots {\color[rgb]{0., 0.296683, 0.912871}0}{\color[rgb]{0.141421, 0.707107, 0.}0}{\color[rgb]{0.707107, 0.477297, 0.}0}{\color[rgb]{0.707107, 0.212132, 0.}0}_{3} & \cdots {\color[rgb]{0., 0.296683, 0.912871}1}{\color[rgb]{0.141421, 0.707107, 0.}0}{\color[rgb]{0.707107, 0.477297, 0.}0}{\color[rgb]{0.707107, 0.212132, 0.}0}_{3} & \cdots {\color[rgb]{0., 0.296683, 0.912871}0}{\color[rgb]{0.141421, 0.707107, 0.}0}{\color[rgb]{0.707107, 0.477297, 0.}0}{\color[rgb]{0.707107, 0.212132, 0.}0}_{3} & \cdots {\color[rgb]{0., 0.296683, 0.912871}0}{\color[rgb]{0.141421, 0.707107, 0.}0}{\color[rgb]{0.707107, 0.477297, 0.}0}{\color[rgb]{0.707107, 0.212132, 0.}0}_{3} & \cdots {\color[rgb]{0., 0.296683, 0.912871}0}{\color[rgb]{0.141421, 0.707107, 0.}0}{\color[rgb]{0.707107, 0.477297, 0.}0}{\color[rgb]{0.707107, 0.212132, 0.}0}_{3} \\
M_{10}^{\Td^\#} & \cdots {\color[rgb]{0., 0.293368, 0.902671}0}{\color[rgb]{0.133333, 0.666667, 0.}2}{\color[rgb]{0.57735, 0.389711, 0.}2}{\color[rgb]{0., 0., 0.}2}_{3} & \cdots {\color[rgb]{0., 0.293368, 0.902671}1}{\color[rgb]{0.133333, 0.666667, 0.}2}{\color[rgb]{0.57735, 0.389711, 0.}2}{\color[rgb]{0., 0., 0.}0}_{3} & \cdots {\color[rgb]{0., 0.293368, 0.902671}0}{\color[rgb]{0.133333, 0.666667, 0.}2}{\color[rgb]{0.57735, 0.389711, 0.}1}{\color[rgb]{0., 0., 0.}0}_{3} & \cdots {\color[rgb]{0., 0.293368, 0.902671}2}{\color[rgb]{0.133333, 0.666667, 0.}2}{\color[rgb]{0.57735, 0.389711, 0.}0}{\color[rgb]{0., 0., 0.}0}_{3} & \cdots {\color[rgb]{0., 0.293368, 0.902671}2}{\color[rgb]{0.133333, 0.666667, 0.}2}{\color[rgb]{0.57735, 0.389711, 0.}0}{\color[rgb]{0., 0., 0.}0}_{3} & \cdots {\color[rgb]{0., 0.293368, 0.902671}0}{\color[rgb]{0.133333, 0.666667, 0.}0}{\color[rgb]{0.57735, 0.389711, 0.}0}{\color[rgb]{0., 0., 0.}0}_{3} & \cdots {\color[rgb]{0., 0.293368, 0.902671}1}{\color[rgb]{0.133333, 0.666667, 0.}0}{\color[rgb]{0.57735, 0.389711, 0.}0}{\color[rgb]{0., 0., 0.}0}_{3} \\
M_{11}^{\Td^\#} & \cdots {\color[rgb]{0., 0.290015, 0.892354}0}{\color[rgb]{0.124722, 0.62361, 0.}0}{\color[rgb]{0.408248, 0.275568, 0.}0}{\color[rgb]{0.707107, 0.212132, 0.}0}_{3} & \cdots {\color[rgb]{0., 0.290015, 0.892354}1}{\color[rgb]{0.124722, 0.62361, 0.}1}{\color[rgb]{0.408248, 0.275568, 0.}1}{\color[rgb]{0.707107, 0.212132, 0.}0}_{3} & \cdots {\color[rgb]{0., 0.290015, 0.892354}1}{\color[rgb]{0.124722, 0.62361, 0.}1}{\color[rgb]{0.408248, 0.275568, 0.}0}{\color[rgb]{0.707107, 0.212132, 0.}0}_{3} & \cdots {\color[rgb]{0., 0.290015, 0.892354}1}{\color[rgb]{0.124722, 0.62361, 0.}0}{\color[rgb]{0.408248, 0.275568, 0.}2}{\color[rgb]{0.707107, 0.212132, 0.}0}_{3} & \cdots {\color[rgb]{0., 0.290015, 0.892354}1}{\color[rgb]{0.124722, 0.62361, 0.}1}{\color[rgb]{0.408248, 0.275568, 0.}0}{\color[rgb]{0.707107, 0.212132, 0.}0}_{3} & \cdots {\color[rgb]{0., 0.290015, 0.892354}0}{\color[rgb]{0.124722, 0.62361, 0.}2}{\color[rgb]{0.408248, 0.275568, 0.}0}{\color[rgb]{0.707107, 0.212132, 0.}0}_{3} & \cdots {\color[rgb]{0., 0.290015, 0.892354}0}{\color[rgb]{0.124722, 0.62361, 0.}0}{\color[rgb]{0.408248, 0.275568, 0.}0}{\color[rgb]{0.707107, 0.212132, 0.}0}_{3} \\
M_{12}^{\Td^\#} & \cdots {\color[rgb]{0., 0.286623, 0.881917}1}{\color[rgb]{0.11547, 0.57735, 0.}1}{\color[rgb]{0., 0., 0.}1}{\color[rgb]{0., 0., 0.}2}_{3} & \cdots {\color[rgb]{0., 0.286623, 0.881917}2}{\color[rgb]{0.11547, 0.57735, 0.}2}{\color[rgb]{0., 0., 0.}0}{\color[rgb]{0., 0., 0.}0}_{3} & \cdots {\color[rgb]{0., 0.286623, 0.881917}1}{\color[rgb]{0.11547, 0.57735, 0.}1}{\color[rgb]{0., 0., 0.}0}{\color[rgb]{0., 0., 0.}0}_{3} & \cdots {\color[rgb]{0., 0.286623, 0.881917}0}{\color[rgb]{0.11547, 0.57735, 0.}0}{\color[rgb]{0., 0., 0.}0}{\color[rgb]{0., 0., 0.}0}_{3} & \cdots {\color[rgb]{0., 0.286623, 0.881917}1}{\color[rgb]{0.11547, 0.57735, 0.}0}{\color[rgb]{0., 0., 0.}0}{\color[rgb]{0., 0., 0.}0}_{3} & \cdots {\color[rgb]{0., 0.286623, 0.881917}0}{\color[rgb]{0.11547, 0.57735, 0.}0}{\color[rgb]{0., 0., 0.}0}{\color[rgb]{0., 0., 0.}0}_{3} & \cdots {\color[rgb]{0., 0.286623, 0.881917}1}{\color[rgb]{0.11547, 0.57735, 0.}0}{\color[rgb]{0., 0., 0.}0}{\color[rgb]{0., 0., 0.}0}_{3} \\
M_{13}^{\Td^\#} & \cdots {\color[rgb]{0., 0.28319, 0.871355}0}{\color[rgb]{0.105409, 0.527046, 0.}0}{\color[rgb]{0.912871, 0.616188, 0.}0}{\color[rgb]{0.707107, 0.212132, 0.}0}_{3} & \cdots {\color[rgb]{0., 0.28319, 0.871355}1}{\color[rgb]{0.105409, 0.527046, 0.}1}{\color[rgb]{0.912871, 0.616188, 0.}2}{\color[rgb]{0.707107, 0.212132, 0.}0}_{3} & \cdots {\color[rgb]{0., 0.28319, 0.871355}2}{\color[rgb]{0.105409, 0.527046, 0.}0}{\color[rgb]{0.912871, 0.616188, 0.}0}{\color[rgb]{0.707107, 0.212132, 0.}0}_{3} & \cdots {\color[rgb]{0., 0.28319, 0.871355}2}{\color[rgb]{0.105409, 0.527046, 0.}1}{\color[rgb]{0.912871, 0.616188, 0.}1}{\color[rgb]{0.707107, 0.212132, 0.}0}_{3} & \cdots {\color[rgb]{0., 0.28319, 0.871355}2}{\color[rgb]{0.105409, 0.527046, 0.}1}{\color[rgb]{0.912871, 0.616188, 0.}0}{\color[rgb]{0.707107, 0.212132, 0.}0}_{3} & \cdots {\color[rgb]{0., 0.28319, 0.871355}1}{\color[rgb]{0.105409, 0.527046, 0.}2}{\color[rgb]{0.912871, 0.616188, 0.}0}{\color[rgb]{0.707107, 0.212132, 0.}0}_{3} & \cdots {\color[rgb]{0., 0.28319, 0.871355}1}{\color[rgb]{0.105409, 0.527046, 0.}0}{\color[rgb]{0.912871, 0.616188, 0.}0}{\color[rgb]{0.707107, 0.212132, 0.}0}_{3} \\
M_{14}^{\Td^\#} & \cdots {\color[rgb]{0., 0.279715, 0.860663}2}{\color[rgb]{0.0942809, 0.471405, 0.}2}{\color[rgb]{0.816497, 0.551135, 0.}0}{\color[rgb]{0., 0., 0.}2}_{3} & \cdots {\color[rgb]{0., 0.279715, 0.860663}0}{\color[rgb]{0.0942809, 0.471405, 0.}0}{\color[rgb]{0.816497, 0.551135, 0.}1}{\color[rgb]{0., 0., 0.}0}_{3} & \cdots {\color[rgb]{0., 0.279715, 0.860663}1}{\color[rgb]{0.0942809, 0.471405, 0.}2}{\color[rgb]{0.816497, 0.551135, 0.}2}{\color[rgb]{0., 0., 0.}0}_{3} & \cdots {\color[rgb]{0., 0.279715, 0.860663}0}{\color[rgb]{0.0942809, 0.471405, 0.}2}{\color[rgb]{0.816497, 0.551135, 0.}0}{\color[rgb]{0., 0., 0.}0}_{3} & \cdots {\color[rgb]{0., 0.279715, 0.860663}1}{\color[rgb]{0.0942809, 0.471405, 0.}0}{\color[rgb]{0.816497, 0.551135, 0.}0}{\color[rgb]{0., 0., 0.}0}_{3} & \cdots {\color[rgb]{0., 0.279715, 0.860663}1}{\color[rgb]{0.0942809, 0.471405, 0.}0}{\color[rgb]{0.816497, 0.551135, 0.}0}{\color[rgb]{0., 0., 0.}0}_{3} & \cdots {\color[rgb]{0., 0.279715, 0.860663}1}{\color[rgb]{0.0942809, 0.471405, 0.}1}{\color[rgb]{0.816497, 0.551135, 0.}0}{\color[rgb]{0., 0., 0.}0}_{3} \\
M_{15}^{\Td^\#} & \cdots {\color[rgb]{0., 0.276197, 0.849837}0}{\color[rgb]{0.0816497, 0.408248, 0.}0}{\color[rgb]{0.707107, 0.477297, 0.}0}{\color[rgb]{0.707107, 0.212132, 0.}0}_{3} & \cdots {\color[rgb]{0., 0.276197, 0.849837}0}{\color[rgb]{0.0816497, 0.408248, 0.}1}{\color[rgb]{0.707107, 0.477297, 0.}0}{\color[rgb]{0.707107, 0.212132, 0.}0}_{3} & \cdots {\color[rgb]{0., 0.276197, 0.849837}2}{\color[rgb]{0.0816497, 0.408248, 0.}0}{\color[rgb]{0.707107, 0.477297, 0.}0}{\color[rgb]{0.707107, 0.212132, 0.}0}_{3} & \cdots {\color[rgb]{0., 0.276197, 0.849837}0}{\color[rgb]{0.0816497, 0.408248, 0.}2}{\color[rgb]{0.707107, 0.477297, 0.}0}{\color[rgb]{0.707107, 0.212132, 0.}0}_{3} & \cdots {\color[rgb]{0., 0.276197, 0.849837}0}{\color[rgb]{0.0816497, 0.408248, 0.}0}{\color[rgb]{0.707107, 0.477297, 0.}0}{\color[rgb]{0.707107, 0.212132, 0.}0}_{3} & \cdots {\color[rgb]{0., 0.276197, 0.849837}0}{\color[rgb]{0.0816497, 0.408248, 0.}0}{\color[rgb]{0.707107, 0.477297, 0.}0}{\color[rgb]{0.707107, 0.212132, 0.}0}_{3} & \cdots {\color[rgb]{0., 0.276197, 0.849837}0}{\color[rgb]{0.0816497, 0.408248, 0.}0}{\color[rgb]{0.707107, 0.477297, 0.}0}{\color[rgb]{0.707107, 0.212132, 0.}0}_{3} \\
M_{16}^{\Td^\#} & \cdots {\color[rgb]{0., 0.272633, 0.83887}0}{\color[rgb]{0.0666667, 0.333333, 0.}1}{\color[rgb]{0.57735, 0.389711, 0.}2}{\color[rgb]{0., 0., 0.}2}_{3} & \cdots {\color[rgb]{0., 0.272633, 0.83887}1}{\color[rgb]{0.0666667, 0.333333, 0.}0}{\color[rgb]{0.57735, 0.389711, 0.}2}{\color[rgb]{0., 0., 0.}0}_{3} & \cdots {\color[rgb]{0., 0.272633, 0.83887}1}{\color[rgb]{0.0666667, 0.333333, 0.}1}{\color[rgb]{0.57735, 0.389711, 0.}1}{\color[rgb]{0., 0., 0.}0}_{3} & \cdots {\color[rgb]{0., 0.272633, 0.83887}1}{\color[rgb]{0.0666667, 0.333333, 0.}2}{\color[rgb]{0.57735, 0.389711, 0.}0}{\color[rgb]{0., 0., 0.}0}_{3} & \cdots {\color[rgb]{0., 0.272633, 0.83887}2}{\color[rgb]{0.0666667, 0.333333, 0.}2}{\color[rgb]{0.57735, 0.389711, 0.}0}{\color[rgb]{0., 0., 0.}0}_{3} & \cdots {\color[rgb]{0., 0.272633, 0.83887}0}{\color[rgb]{0.0666667, 0.333333, 0.}0}{\color[rgb]{0.57735, 0.389711, 0.}0}{\color[rgb]{0., 0., 0.}0}_{3} & \cdots {\color[rgb]{0., 0.272633, 0.83887}2}{\color[rgb]{0.0666667, 0.333333, 0.}0}{\color[rgb]{0.57735, 0.389711, 0.}0}{\color[rgb]{0., 0., 0.}0}_{3} \\
M_{17}^{\Td^\#} & \cdots {\color[rgb]{0., 0.269022, 0.827759}0}{\color[rgb]{0.0471405, 0.235702, 0.}0}{\color[rgb]{0.408248, 0.275568, 0.}0}{\color[rgb]{0.707107, 0.212132, 0.}0}_{3} & \cdots {\color[rgb]{0., 0.269022, 0.827759}1}{\color[rgb]{0.0471405, 0.235702, 0.}2}{\color[rgb]{0.408248, 0.275568, 0.}1}{\color[rgb]{0.707107, 0.212132, 0.}0}_{3} & \cdots {\color[rgb]{0., 0.269022, 0.827759}1}{\color[rgb]{0.0471405, 0.235702, 0.}1}{\color[rgb]{0.408248, 0.275568, 0.}0}{\color[rgb]{0.707107, 0.212132, 0.}0}_{3} & \cdots {\color[rgb]{0., 0.269022, 0.827759}0}{\color[rgb]{0.0471405, 0.235702, 0.}2}{\color[rgb]{0.408248, 0.275568, 0.}2}{\color[rgb]{0.707107, 0.212132, 0.}0}_{3} & \cdots {\color[rgb]{0., 0.269022, 0.827759}0}{\color[rgb]{0.0471405, 0.235702, 0.}1}{\color[rgb]{0.408248, 0.275568, 0.}0}{\color[rgb]{0.707107, 0.212132, 0.}0}_{3} & \cdots {\color[rgb]{0., 0.269022, 0.827759}1}{\color[rgb]{0.0471405, 0.235702, 0.}2}{\color[rgb]{0.408248, 0.275568, 0.}0}{\color[rgb]{0.707107, 0.212132, 0.}0}_{3} & \cdots {\color[rgb]{0., 0.269022, 0.827759}0}{\color[rgb]{0.0471405, 0.235702, 0.}0}{\color[rgb]{0.408248, 0.275568, 0.}0}{\color[rgb]{0.707107, 0.212132, 0.}0}_{3} \\
M_{18}^{\Td^\#} & \cdots {\color[rgb]{0., 0.265361, 0.816497}0}{\color[rgb]{0., 0., 0.}0}{\color[rgb]{0., 0., 0.}1}{\color[rgb]{0., 0., 0.}2}_{3} & \cdots {\color[rgb]{0., 0.265361, 0.816497}1}{\color[rgb]{0., 0., 0.}0}{\color[rgb]{0., 0., 0.}0}{\color[rgb]{0., 0., 0.}0}_{3} & \cdots {\color[rgb]{0., 0.265361, 0.816497}2}{\color[rgb]{0., 0., 0.}0}{\color[rgb]{0., 0., 0.}0}{\color[rgb]{0., 0., 0.}0}_{3} & \cdots {\color[rgb]{0., 0.265361, 0.816497}0}{\color[rgb]{0., 0., 0.}0}{\color[rgb]{0., 0., 0.}0}{\color[rgb]{0., 0., 0.}0}_{3} & \cdots {\color[rgb]{0., 0.265361, 0.816497}0}{\color[rgb]{0., 0., 0.}0}{\color[rgb]{0., 0., 0.}0}{\color[rgb]{0., 0., 0.}0}_{3} & \cdots {\color[rgb]{0., 0.265361, 0.816497}0}{\color[rgb]{0., 0., 0.}0}{\color[rgb]{0., 0., 0.}0}{\color[rgb]{0., 0., 0.}0}_{3} & \cdots {\color[rgb]{0., 0.265361, 0.816497}0}{\color[rgb]{0., 0., 0.}0}{\color[rgb]{0., 0., 0.}0}{\color[rgb]{0., 0., 0.}0}_{3} \\
M_{19}^{\Td^\#} & \cdots {\color[rgb]{0., 0.26165, 0.805076}0}{\color[rgb]{0.194365, 0.971825, 0.}0}{\color[rgb]{0.912871, 0.616188, 0.}0}{\color[rgb]{0.707107, 0.212132, 0.}0}_{3} & \cdots {\color[rgb]{0., 0.26165, 0.805076}0}{\color[rgb]{0.194365, 0.971825, 0.}2}{\color[rgb]{0.912871, 0.616188, 0.}2}{\color[rgb]{0.707107, 0.212132, 0.}0}_{3} & \cdots {\color[rgb]{0., 0.26165, 0.805076}0}{\color[rgb]{0.194365, 0.971825, 0.}0}{\color[rgb]{0.912871, 0.616188, 0.}0}{\color[rgb]{0.707107, 0.212132, 0.}0}_{3} & \cdots {\color[rgb]{0., 0.26165, 0.805076}2}{\color[rgb]{0.194365, 0.971825, 0.}0}{\color[rgb]{0.912871, 0.616188, 0.}1}{\color[rgb]{0.707107, 0.212132, 0.}0}_{3} & \cdots {\color[rgb]{0., 0.26165, 0.805076}0}{\color[rgb]{0.194365, 0.971825, 0.}1}{\color[rgb]{0.912871, 0.616188, 0.}0}{\color[rgb]{0.707107, 0.212132, 0.}0}_{3} & \cdots {\color[rgb]{0., 0.26165, 0.805076}0}{\color[rgb]{0.194365, 0.971825, 0.}2}{\color[rgb]{0.912871, 0.616188, 0.}0}{\color[rgb]{0.707107, 0.212132, 0.}0}_{3} & \cdots {\color[rgb]{0., 0.26165, 0.805076}1}{\color[rgb]{0.194365, 0.971825, 0.}0}{\color[rgb]{0.912871, 0.616188, 0.}0}{\color[rgb]{0.707107, 0.212132, 0.}0}_{3} \\
M_{20}^{\Td^\#} & \cdots {\color[rgb]{0., 0.257885, 0.793492}0}{\color[rgb]{0.188562, 0.942809, 0.}1}{\color[rgb]{0.816497, 0.551135, 0.}0}{\color[rgb]{0., 0., 0.}2}_{3} & \cdots {\color[rgb]{0., 0.257885, 0.793492}0}{\color[rgb]{0.188562, 0.942809, 0.}1}{\color[rgb]{0.816497, 0.551135, 0.}1}{\color[rgb]{0., 0., 0.}0}_{3} & \cdots {\color[rgb]{0., 0.257885, 0.793492}2}{\color[rgb]{0.188562, 0.942809, 0.}1}{\color[rgb]{0.816497, 0.551135, 0.}2}{\color[rgb]{0., 0., 0.}0}_{3} & \cdots {\color[rgb]{0., 0.257885, 0.793492}1}{\color[rgb]{0.188562, 0.942809, 0.}2}{\color[rgb]{0.816497, 0.551135, 0.}0}{\color[rgb]{0., 0., 0.}0}_{3} & \cdots {\color[rgb]{0., 0.257885, 0.793492}2}{\color[rgb]{0.188562, 0.942809, 0.}0}{\color[rgb]{0.816497, 0.551135, 0.}0}{\color[rgb]{0., 0., 0.}0}_{3} & \cdots {\color[rgb]{0., 0.257885, 0.793492}1}{\color[rgb]{0.188562, 0.942809, 0.}0}{\color[rgb]{0.816497, 0.551135, 0.}0}{\color[rgb]{0., 0., 0.}0}_{3} & \cdots {\color[rgb]{0., 0.257885, 0.793492}1}{\color[rgb]{0.188562, 0.942809, 0.}1}{\color[rgb]{0.816497, 0.551135, 0.}0}{\color[rgb]{0., 0., 0.}0}_{3} \\
M_{21}^{\Td^\#} & \cdots {\color[rgb]{0., 0.254064, 0.781736}0}{\color[rgb]{0.182574, 0.912871, 0.}0}{\color[rgb]{0.707107, 0.477297, 0.}0}{\color[rgb]{0.707107, 0.212132, 0.}0}_{3} & \cdots {\color[rgb]{0., 0.254064, 0.781736}1}{\color[rgb]{0.182574, 0.912871, 0.}2}{\color[rgb]{0.707107, 0.477297, 0.}0}{\color[rgb]{0.707107, 0.212132, 0.}0}_{3} & \cdots {\color[rgb]{0., 0.254064, 0.781736}1}{\color[rgb]{0.182574, 0.912871, 0.}0}{\color[rgb]{0.707107, 0.477297, 0.}0}{\color[rgb]{0.707107, 0.212132, 0.}0}_{3} & \cdots {\color[rgb]{0., 0.254064, 0.781736}0}{\color[rgb]{0.182574, 0.912871, 0.}1}{\color[rgb]{0.707107, 0.477297, 0.}0}{\color[rgb]{0.707107, 0.212132, 0.}0}_{3} & \cdots {\color[rgb]{0., 0.254064, 0.781736}0}{\color[rgb]{0.182574, 0.912871, 0.}0}{\color[rgb]{0.707107, 0.477297, 0.}0}{\color[rgb]{0.707107, 0.212132, 0.}0}_{3} & \cdots {\color[rgb]{0., 0.254064, 0.781736}0}{\color[rgb]{0.182574, 0.912871, 0.}0}{\color[rgb]{0.707107, 0.477297, 0.}0}{\color[rgb]{0.707107, 0.212132, 0.}0}_{3} & \cdots {\color[rgb]{0., 0.254064, 0.781736}0}{\color[rgb]{0.182574, 0.912871, 0.}0}{\color[rgb]{0.707107, 0.477297, 0.}0}{\color[rgb]{0.707107, 0.212132, 0.}0}_{3} \\
M_{22}^{\Td^\#} & \cdots {\color[rgb]{0., 0.250185, 0.7698}0}{\color[rgb]{0.176383, 0.881917, 0.}0}{\color[rgb]{0.57735, 0.389711, 0.}2}{\color[rgb]{0., 0., 0.}2}_{3} & \cdots {\color[rgb]{0., 0.250185, 0.7698}0}{\color[rgb]{0.176383, 0.881917, 0.}1}{\color[rgb]{0.57735, 0.389711, 0.}2}{\color[rgb]{0., 0., 0.}0}_{3} & \cdots {\color[rgb]{0., 0.250185, 0.7698}2}{\color[rgb]{0.176383, 0.881917, 0.}0}{\color[rgb]{0.57735, 0.389711, 0.}1}{\color[rgb]{0., 0., 0.}0}_{3} & \cdots {\color[rgb]{0., 0.250185, 0.7698}0}{\color[rgb]{0.176383, 0.881917, 0.}2}{\color[rgb]{0.57735, 0.389711, 0.}0}{\color[rgb]{0., 0., 0.}0}_{3} & \cdots {\color[rgb]{0., 0.250185, 0.7698}2}{\color[rgb]{0.176383, 0.881917, 0.}2}{\color[rgb]{0.57735, 0.389711, 0.}0}{\color[rgb]{0., 0., 0.}0}_{3} & \cdots {\color[rgb]{0., 0.250185, 0.7698}0}{\color[rgb]{0.176383, 0.881917, 0.}0}{\color[rgb]{0.57735, 0.389711, 0.}0}{\color[rgb]{0., 0., 0.}0}_{3} & \cdots {\color[rgb]{0., 0.250185, 0.7698}0}{\color[rgb]{0.176383, 0.881917, 0.}0}{\color[rgb]{0.57735, 0.389711, 0.}0}{\color[rgb]{0., 0., 0.}0}_{3} \\
M_{23}^{\Td^\#} & \cdots {\color[rgb]{0., 0.246245, 0.757677}0}{\color[rgb]{0.169967, 0.849837, 0.}0}{\color[rgb]{0.408248, 0.275568, 0.}0}{\color[rgb]{0.707107, 0.212132, 0.}0}_{3} & \cdots {\color[rgb]{0., 0.246245, 0.757677}2}{\color[rgb]{0.169967, 0.849837, 0.}0}{\color[rgb]{0.408248, 0.275568, 0.}1}{\color[rgb]{0.707107, 0.212132, 0.}0}_{3} & \cdots {\color[rgb]{0., 0.246245, 0.757677}1}{\color[rgb]{0.169967, 0.849837, 0.}1}{\color[rgb]{0.408248, 0.275568, 0.}0}{\color[rgb]{0.707107, 0.212132, 0.}0}_{3} & \cdots {\color[rgb]{0., 0.246245, 0.757677}0}{\color[rgb]{0.169967, 0.849837, 0.}1}{\color[rgb]{0.408248, 0.275568, 0.}2}{\color[rgb]{0.707107, 0.212132, 0.}0}_{3} & \cdots {\color[rgb]{0., 0.246245, 0.757677}2}{\color[rgb]{0.169967, 0.849837, 0.}1}{\color[rgb]{0.408248, 0.275568, 0.}0}{\color[rgb]{0.707107, 0.212132, 0.}0}_{3} & \cdots {\color[rgb]{0., 0.246245, 0.757677}2}{\color[rgb]{0.169967, 0.849837, 0.}2}{\color[rgb]{0.408248, 0.275568, 0.}0}{\color[rgb]{0.707107, 0.212132, 0.}0}_{3} & \cdots {\color[rgb]{0., 0.246245, 0.757677}0}{\color[rgb]{0.169967, 0.849837, 0.}0}{\color[rgb]{0.408248, 0.275568, 0.}0}{\color[rgb]{0.707107, 0.212132, 0.}0}_{3} \\
M_{24}^{\Td^\#} & \cdots {\color[rgb]{0., 0.242241, 0.745356}1}{\color[rgb]{0.163299, 0.816497, 0.}2}{\color[rgb]{0., 0., 0.}1}{\color[rgb]{0., 0., 0.}2}_{3} & \cdots {\color[rgb]{0., 0.242241, 0.745356}2}{\color[rgb]{0.163299, 0.816497, 0.}1}{\color[rgb]{0., 0., 0.}0}{\color[rgb]{0., 0., 0.}0}_{3} & \cdots {\color[rgb]{0., 0.242241, 0.745356}2}{\color[rgb]{0.163299, 0.816497, 0.}2}{\color[rgb]{0., 0., 0.}0}{\color[rgb]{0., 0., 0.}0}_{3} & \cdots {\color[rgb]{0., 0.242241, 0.745356}0}{\color[rgb]{0.163299, 0.816497, 0.}0}{\color[rgb]{0., 0., 0.}0}{\color[rgb]{0., 0., 0.}0}_{3} & \cdots {\color[rgb]{0., 0.242241, 0.745356}2}{\color[rgb]{0.163299, 0.816497, 0.}0}{\color[rgb]{0., 0., 0.}0}{\color[rgb]{0., 0., 0.}0}_{3} & \cdots {\color[rgb]{0., 0.242241, 0.745356}0}{\color[rgb]{0.163299, 0.816497, 0.}0}{\color[rgb]{0., 0., 0.}0}{\color[rgb]{0., 0., 0.}0}_{3} & \cdots {\color[rgb]{0., 0.242241, 0.745356}2}{\color[rgb]{0.163299, 0.816497, 0.}0}{\color[rgb]{0., 0., 0.}0}{\color[rgb]{0., 0., 0.}0}_{3} \\
M_{25}^{\Td^\#} & \cdots {\color[rgb]{0., 0.238169, 0.732828}0}{\color[rgb]{0.156347, 0.781736, 0.}0}{\color[rgb]{0.912871, 0.616188, 0.}0}{\color[rgb]{0.707107, 0.212132, 0.}0}_{3} & \cdots {\color[rgb]{0., 0.238169, 0.732828}0}{\color[rgb]{0.156347, 0.781736, 0.}0}{\color[rgb]{0.912871, 0.616188, 0.}2}{\color[rgb]{0.707107, 0.212132, 0.}0}_{3} & \cdots {\color[rgb]{0., 0.238169, 0.732828}1}{\color[rgb]{0.156347, 0.781736, 0.}0}{\color[rgb]{0.912871, 0.616188, 0.}0}{\color[rgb]{0.707107, 0.212132, 0.}0}_{3} & \cdots {\color[rgb]{0., 0.238169, 0.732828}1}{\color[rgb]{0.156347, 0.781736, 0.}2}{\color[rgb]{0.912871, 0.616188, 0.}1}{\color[rgb]{0.707107, 0.212132, 0.}0}_{3} & \cdots {\color[rgb]{0., 0.238169, 0.732828}1}{\color[rgb]{0.156347, 0.781736, 0.}1}{\color[rgb]{0.912871, 0.616188, 0.}0}{\color[rgb]{0.707107, 0.212132, 0.}0}_{3} & \cdots {\color[rgb]{0., 0.238169, 0.732828}2}{\color[rgb]{0.156347, 0.781736, 0.}2}{\color[rgb]{0.912871, 0.616188, 0.}0}{\color[rgb]{0.707107, 0.212132, 0.}0}_{3} & \cdots {\color[rgb]{0., 0.238169, 0.732828}1}{\color[rgb]{0.156347, 0.781736, 0.}0}{\color[rgb]{0.912871, 0.616188, 0.}0}{\color[rgb]{0.707107, 0.212132, 0.}0}_{3} \\
M_{26}^{\Td^\#} & \cdots {\color[rgb]{0., 0.234027, 0.720082}1}{\color[rgb]{0.149071, 0.745356, 0.}0}{\color[rgb]{0.816497, 0.551135, 0.}0}{\color[rgb]{0., 0., 0.}2}_{3} & \cdots {\color[rgb]{0., 0.234027, 0.720082}0}{\color[rgb]{0.149071, 0.745356, 0.}2}{\color[rgb]{0.816497, 0.551135, 0.}1}{\color[rgb]{0., 0., 0.}0}_{3} & \cdots {\color[rgb]{0., 0.234027, 0.720082}0}{\color[rgb]{0.149071, 0.745356, 0.}0}{\color[rgb]{0.816497, 0.551135, 0.}2}{\color[rgb]{0., 0., 0.}0}_{3} & \cdots {\color[rgb]{0., 0.234027, 0.720082}2}{\color[rgb]{0.149071, 0.745356, 0.}2}{\color[rgb]{0.816497, 0.551135, 0.}0}{\color[rgb]{0., 0., 0.}0}_{3} & \cdots {\color[rgb]{0., 0.234027, 0.720082}0}{\color[rgb]{0.149071, 0.745356, 0.}0}{\color[rgb]{0.816497, 0.551135, 0.}0}{\color[rgb]{0., 0., 0.}0}_{3} & \cdots {\color[rgb]{0., 0.234027, 0.720082}1}{\color[rgb]{0.149071, 0.745356, 0.}0}{\color[rgb]{0.816497, 0.551135, 0.}0}{\color[rgb]{0., 0., 0.}0}_{3} & \cdots {\color[rgb]{0., 0.234027, 0.720082}1}{\color[rgb]{0.149071, 0.745356, 0.}1}{\color[rgb]{0.816497, 0.551135, 0.}0}{\color[rgb]{0., 0., 0.}0}_{3} \\
M_{27}^{\Td^\#} & \cdots {\color[rgb]{0., 0.22981, 0.707107}0}{\color[rgb]{0.141421, 0.707107, 0.}0}{\color[rgb]{0.707107, 0.477297, 0.}0}{\color[rgb]{0.707107, 0.212132, 0.}0}_{3} & \cdots {\color[rgb]{0., 0.22981, 0.707107}0}{\color[rgb]{0.141421, 0.707107, 0.}0}{\color[rgb]{0.707107, 0.477297, 0.}0}{\color[rgb]{0.707107, 0.212132, 0.}0}_{3} & \cdots {\color[rgb]{0., 0.22981, 0.707107}0}{\color[rgb]{0.141421, 0.707107, 0.}0}{\color[rgb]{0.707107, 0.477297, 0.}0}{\color[rgb]{0.707107, 0.212132, 0.}0}_{3} & \cdots {\color[rgb]{0., 0.22981, 0.707107}0}{\color[rgb]{0.141421, 0.707107, 0.}0}{\color[rgb]{0.707107, 0.477297, 0.}0}{\color[rgb]{0.707107, 0.212132, 0.}0}_{3} & \cdots {\color[rgb]{0., 0.22981, 0.707107}0}{\color[rgb]{0.141421, 0.707107, 0.}0}{\color[rgb]{0.707107, 0.477297, 0.}0}{\color[rgb]{0.707107, 0.212132, 0.}0}_{3} & \cdots {\color[rgb]{0., 0.22981, 0.707107}0}{\color[rgb]{0.141421, 0.707107, 0.}0}{\color[rgb]{0.707107, 0.477297, 0.}0}{\color[rgb]{0.707107, 0.212132, 0.}0}_{3} & \cdots {\color[rgb]{0., 0.22981, 0.707107}0}{\color[rgb]{0.141421, 0.707107, 0.}0}{\color[rgb]{0.707107, 0.477297, 0.}0}{\color[rgb]{0.707107, 0.212132, 0.}0}_{3} \\
M_{28}^{\Td^\#} & \cdots {\color[rgb]{0., 0.225514, 0.693889}2}{\color[rgb]{0.133333, 0.666667, 0.}2}{\color[rgb]{0.57735, 0.389711, 0.}2}{\color[rgb]{0., 0., 0.}2}_{3} & \cdots {\color[rgb]{0., 0.225514, 0.693889}2}{\color[rgb]{0.133333, 0.666667, 0.}2}{\color[rgb]{0.57735, 0.389711, 0.}2}{\color[rgb]{0., 0., 0.}0}_{3} & \cdots {\color[rgb]{0., 0.225514, 0.693889}2}{\color[rgb]{0.133333, 0.666667, 0.}2}{\color[rgb]{0.57735, 0.389711, 0.}1}{\color[rgb]{0., 0., 0.}0}_{3} & \cdots {\color[rgb]{0., 0.225514, 0.693889}2}{\color[rgb]{0.133333, 0.666667, 0.}2}{\color[rgb]{0.57735, 0.389711, 0.}0}{\color[rgb]{0., 0., 0.}0}_{3} & \cdots {\color[rgb]{0., 0.225514, 0.693889}2}{\color[rgb]{0.133333, 0.666667, 0.}2}{\color[rgb]{0.57735, 0.389711, 0.}0}{\color[rgb]{0., 0., 0.}0}_{3} & \cdots {\color[rgb]{0., 0.225514, 0.693889}0}{\color[rgb]{0.133333, 0.666667, 0.}0}{\color[rgb]{0.57735, 0.389711, 0.}0}{\color[rgb]{0., 0., 0.}0}_{3} & \cdots {\color[rgb]{0., 0.225514, 0.693889}1}{\color[rgb]{0.133333, 0.666667, 0.}0}{\color[rgb]{0.57735, 0.389711, 0.}0}{\color[rgb]{0., 0., 0.}0}_{3} \\
M_{29}^{\Td^\#} & \cdots {\color[rgb]{0., 0.221134, 0.680414}0}{\color[rgb]{0.124722, 0.62361, 0.}0}{\color[rgb]{0.408248, 0.275568, 0.}0}{\color[rgb]{0.707107, 0.212132, 0.}0}_{3} & \cdots {\color[rgb]{0., 0.221134, 0.680414}2}{\color[rgb]{0.124722, 0.62361, 0.}1}{\color[rgb]{0.408248, 0.275568, 0.}1}{\color[rgb]{0.707107, 0.212132, 0.}0}_{3} & \cdots {\color[rgb]{0., 0.221134, 0.680414}1}{\color[rgb]{0.124722, 0.62361, 0.}1}{\color[rgb]{0.408248, 0.275568, 0.}0}{\color[rgb]{0.707107, 0.212132, 0.}0}_{3} & \cdots {\color[rgb]{0., 0.221134, 0.680414}0}{\color[rgb]{0.124722, 0.62361, 0.}0}{\color[rgb]{0.408248, 0.275568, 0.}2}{\color[rgb]{0.707107, 0.212132, 0.}0}_{3} & \cdots {\color[rgb]{0., 0.221134, 0.680414}1}{\color[rgb]{0.124722, 0.62361, 0.}1}{\color[rgb]{0.408248, 0.275568, 0.}0}{\color[rgb]{0.707107, 0.212132, 0.}0}_{3} & \cdots {\color[rgb]{0., 0.221134, 0.680414}0}{\color[rgb]{0.124722, 0.62361, 0.}2}{\color[rgb]{0.408248, 0.275568, 0.}0}{\color[rgb]{0.707107, 0.212132, 0.}0}_{3} & \cdots {\color[rgb]{0., 0.221134, 0.680414}0}{\color[rgb]{0.124722, 0.62361, 0.}0}{\color[rgb]{0.408248, 0.275568, 0.}0}{\color[rgb]{0.707107, 0.212132, 0.}0}_{3} \\
M_{30}^{\Td^\#} & \cdots {\color[rgb]{0., 0.216667, 0.666667}0}{\color[rgb]{0.11547, 0.57735, 0.}1}{\color[rgb]{0., 0., 0.}1}{\color[rgb]{0., 0., 0.}2}_{3} & \cdots {\color[rgb]{0., 0.216667, 0.666667}0}{\color[rgb]{0.11547, 0.57735, 0.}2}{\color[rgb]{0., 0., 0.}0}{\color[rgb]{0., 0., 0.}0}_{3} & \cdots {\color[rgb]{0., 0.216667, 0.666667}0}{\color[rgb]{0.11547, 0.57735, 0.}1}{\color[rgb]{0., 0., 0.}0}{\color[rgb]{0., 0., 0.}0}_{3} & \cdots {\color[rgb]{0., 0.216667, 0.666667}0}{\color[rgb]{0.11547, 0.57735, 0.}0}{\color[rgb]{0., 0., 0.}0}{\color[rgb]{0., 0., 0.}0}_{3} & \cdots {\color[rgb]{0., 0.216667, 0.666667}1}{\color[rgb]{0.11547, 0.57735, 0.}0}{\color[rgb]{0., 0., 0.}0}{\color[rgb]{0., 0., 0.}0}_{3} & \cdots {\color[rgb]{0., 0.216667, 0.666667}0}{\color[rgb]{0.11547, 0.57735, 0.}0}{\color[rgb]{0., 0., 0.}0}{\color[rgb]{0., 0., 0.}0}_{3} & \cdots {\color[rgb]{0., 0.216667, 0.666667}1}{\color[rgb]{0.11547, 0.57735, 0.}0}{\color[rgb]{0., 0., 0.}0}{\color[rgb]{0., 0., 0.}0}_{3} \\
M_{31}^{\Td^\#} & \cdots {\color[rgb]{0., 0.212105, 0.65263}0}{\color[rgb]{0.105409, 0.527046, 0.}0}{\color[rgb]{0.912871, 0.616188, 0.}0}{\color[rgb]{0.707107, 0.212132, 0.}0}_{3} & \cdots {\color[rgb]{0., 0.212105, 0.65263}2}{\color[rgb]{0.105409, 0.527046, 0.}1}{\color[rgb]{0.912871, 0.616188, 0.}2}{\color[rgb]{0.707107, 0.212132, 0.}0}_{3} & \cdots {\color[rgb]{0., 0.212105, 0.65263}2}{\color[rgb]{0.105409, 0.527046, 0.}0}{\color[rgb]{0.912871, 0.616188, 0.}0}{\color[rgb]{0.707107, 0.212132, 0.}0}_{3} & \cdots {\color[rgb]{0., 0.212105, 0.65263}1}{\color[rgb]{0.105409, 0.527046, 0.}1}{\color[rgb]{0.912871, 0.616188, 0.}1}{\color[rgb]{0.707107, 0.212132, 0.}0}_{3} & \cdots {\color[rgb]{0., 0.212105, 0.65263}2}{\color[rgb]{0.105409, 0.527046, 0.}1}{\color[rgb]{0.912871, 0.616188, 0.}0}{\color[rgb]{0.707107, 0.212132, 0.}0}_{3} & \cdots {\color[rgb]{0., 0.212105, 0.65263}1}{\color[rgb]{0.105409, 0.527046, 0.}2}{\color[rgb]{0.912871, 0.616188, 0.}0}{\color[rgb]{0.707107, 0.212132, 0.}0}_{3} & \cdots {\color[rgb]{0., 0.212105, 0.65263}1}{\color[rgb]{0.105409, 0.527046, 0.}0}{\color[rgb]{0.912871, 0.616188, 0.}0}{\color[rgb]{0.707107, 0.212132, 0.}0}_{3} \\
M_{32}^{\Td^\#} & \cdots {\color[rgb]{0., 0.207443, 0.638285}1}{\color[rgb]{0.0942809, 0.471405, 0.}2}{\color[rgb]{0.816497, 0.551135, 0.}0}{\color[rgb]{0., 0., 0.}2}_{3} & \cdots {\color[rgb]{0., 0.207443, 0.638285}1}{\color[rgb]{0.0942809, 0.471405, 0.}0}{\color[rgb]{0.816497, 0.551135, 0.}1}{\color[rgb]{0., 0., 0.}0}_{3} & \cdots {\color[rgb]{0., 0.207443, 0.638285}0}{\color[rgb]{0.0942809, 0.471405, 0.}2}{\color[rgb]{0.816497, 0.551135, 0.}2}{\color[rgb]{0., 0., 0.}0}_{3} & \cdots {\color[rgb]{0., 0.207443, 0.638285}0}{\color[rgb]{0.0942809, 0.471405, 0.}2}{\color[rgb]{0.816497, 0.551135, 0.}0}{\color[rgb]{0., 0., 0.}0}_{3} & \cdots {\color[rgb]{0., 0.207443, 0.638285}1}{\color[rgb]{0.0942809, 0.471405, 0.}0}{\color[rgb]{0.816497, 0.551135, 0.}0}{\color[rgb]{0., 0., 0.}0}_{3} & \cdots {\color[rgb]{0., 0.207443, 0.638285}1}{\color[rgb]{0.0942809, 0.471405, 0.}0}{\color[rgb]{0.816497, 0.551135, 0.}0}{\color[rgb]{0., 0., 0.}0}_{3} & \cdots {\color[rgb]{0., 0.207443, 0.638285}1}{\color[rgb]{0.0942809, 0.471405, 0.}1}{\color[rgb]{0.816497, 0.551135, 0.}0}{\color[rgb]{0., 0., 0.}0}_{3} \\
M_{33}^{\Td^\#} & \cdots {\color[rgb]{0., 0.202673, 0.62361}0}{\color[rgb]{0.0816497, 0.408248, 0.}0}{\color[rgb]{0.707107, 0.477297, 0.}0}{\color[rgb]{0.707107, 0.212132, 0.}0}_{3} & \cdots {\color[rgb]{0., 0.202673, 0.62361}1}{\color[rgb]{0.0816497, 0.408248, 0.}1}{\color[rgb]{0.707107, 0.477297, 0.}0}{\color[rgb]{0.707107, 0.212132, 0.}0}_{3} & \cdots {\color[rgb]{0., 0.202673, 0.62361}2}{\color[rgb]{0.0816497, 0.408248, 0.}0}{\color[rgb]{0.707107, 0.477297, 0.}0}{\color[rgb]{0.707107, 0.212132, 0.}0}_{3} & \cdots {\color[rgb]{0., 0.202673, 0.62361}2}{\color[rgb]{0.0816497, 0.408248, 0.}2}{\color[rgb]{0.707107, 0.477297, 0.}0}{\color[rgb]{0.707107, 0.212132, 0.}0}_{3} & \cdots {\color[rgb]{0., 0.202673, 0.62361}0}{\color[rgb]{0.0816497, 0.408248, 0.}0}{\color[rgb]{0.707107, 0.477297, 0.}0}{\color[rgb]{0.707107, 0.212132, 0.}0}_{3} & \cdots {\color[rgb]{0., 0.202673, 0.62361}0}{\color[rgb]{0.0816497, 0.408248, 0.}0}{\color[rgb]{0.707107, 0.477297, 0.}0}{\color[rgb]{0.707107, 0.212132, 0.}0}_{3} & \cdots {\color[rgb]{0., 0.202673, 0.62361}0}{\color[rgb]{0.0816497, 0.408248, 0.}0}{\color[rgb]{0.707107, 0.477297, 0.}0}{\color[rgb]{0.707107, 0.212132, 0.}0}_{3} \\
M_{34}^{\Td^\#} & \cdots {\color[rgb]{0., 0.197789, 0.608581}2}{\color[rgb]{0.0666667, 0.333333, 0.}1}{\color[rgb]{0.57735, 0.389711, 0.}2}{\color[rgb]{0., 0., 0.}2}_{3} & \cdots {\color[rgb]{0., 0.197789, 0.608581}2}{\color[rgb]{0.0666667, 0.333333, 0.}0}{\color[rgb]{0.57735, 0.389711, 0.}2}{\color[rgb]{0., 0., 0.}0}_{3} & \cdots {\color[rgb]{0., 0.197789, 0.608581}0}{\color[rgb]{0.0666667, 0.333333, 0.}1}{\color[rgb]{0.57735, 0.389711, 0.}1}{\color[rgb]{0., 0., 0.}0}_{3} & \cdots {\color[rgb]{0., 0.197789, 0.608581}1}{\color[rgb]{0.0666667, 0.333333, 0.}2}{\color[rgb]{0.57735, 0.389711, 0.}0}{\color[rgb]{0., 0., 0.}0}_{3} & \cdots {\color[rgb]{0., 0.197789, 0.608581}2}{\color[rgb]{0.0666667, 0.333333, 0.}2}{\color[rgb]{0.57735, 0.389711, 0.}0}{\color[rgb]{0., 0., 0.}0}_{3} & \cdots {\color[rgb]{0., 0.197789, 0.608581}0}{\color[rgb]{0.0666667, 0.333333, 0.}0}{\color[rgb]{0.57735, 0.389711, 0.}0}{\color[rgb]{0., 0., 0.}0}_{3} & \cdots {\color[rgb]{0., 0.197789, 0.608581}2}{\color[rgb]{0.0666667, 0.333333, 0.}0}{\color[rgb]{0.57735, 0.389711, 0.}0}{\color[rgb]{0., 0., 0.}0}_{3} \\
M_{35}^{\Td^\#} & \cdots {\color[rgb]{0., 0.192781, 0.593171}0}{\color[rgb]{0.0471405, 0.235702, 0.}0}{\color[rgb]{0.408248, 0.275568, 0.}0}{\color[rgb]{0.707107, 0.212132, 0.}0}_{3} & \cdots {\color[rgb]{0., 0.192781, 0.593171}2}{\color[rgb]{0.0471405, 0.235702, 0.}2}{\color[rgb]{0.408248, 0.275568, 0.}1}{\color[rgb]{0.707107, 0.212132, 0.}0}_{3} & \cdots {\color[rgb]{0., 0.192781, 0.593171}1}{\color[rgb]{0.0471405, 0.235702, 0.}1}{\color[rgb]{0.408248, 0.275568, 0.}0}{\color[rgb]{0.707107, 0.212132, 0.}0}_{3} & \cdots {\color[rgb]{0., 0.192781, 0.593171}2}{\color[rgb]{0.0471405, 0.235702, 0.}2}{\color[rgb]{0.408248, 0.275568, 0.}2}{\color[rgb]{0.707107, 0.212132, 0.}0}_{3} & \cdots {\color[rgb]{0., 0.192781, 0.593171}0}{\color[rgb]{0.0471405, 0.235702, 0.}1}{\color[rgb]{0.408248, 0.275568, 0.}0}{\color[rgb]{0.707107, 0.212132, 0.}0}_{3} & \cdots {\color[rgb]{0., 0.192781, 0.593171}1}{\color[rgb]{0.0471405, 0.235702, 0.}2}{\color[rgb]{0.408248, 0.275568, 0.}0}{\color[rgb]{0.707107, 0.212132, 0.}0}_{3} & \cdots {\color[rgb]{0., 0.192781, 0.593171}0}{\color[rgb]{0.0471405, 0.235702, 0.}0}{\color[rgb]{0.408248, 0.275568, 0.}0}{\color[rgb]{0.707107, 0.212132, 0.}0}_{3} \\
M_{36}^{\Td^\#} & \cdots {\color[rgb]{0., 0.187639, 0.57735}2}{\color[rgb]{0., 0., 0.}0}{\color[rgb]{0., 0., 0.}1}{\color[rgb]{0., 0., 0.}2}_{3} & \cdots {\color[rgb]{0., 0.187639, 0.57735}2}{\color[rgb]{0., 0., 0.}0}{\color[rgb]{0., 0., 0.}0}{\color[rgb]{0., 0., 0.}0}_{3} & \cdots {\color[rgb]{0., 0.187639, 0.57735}1}{\color[rgb]{0., 0., 0.}0}{\color[rgb]{0., 0., 0.}0}{\color[rgb]{0., 0., 0.}0}_{3} & \cdots {\color[rgb]{0., 0.187639, 0.57735}0}{\color[rgb]{0., 0., 0.}0}{\color[rgb]{0., 0., 0.}0}{\color[rgb]{0., 0., 0.}0}_{3} & \cdots {\color[rgb]{0., 0.187639, 0.57735}0}{\color[rgb]{0., 0., 0.}0}{\color[rgb]{0., 0., 0.}0}{\color[rgb]{0., 0., 0.}0}_{3} & \cdots {\color[rgb]{0., 0.187639, 0.57735}0}{\color[rgb]{0., 0., 0.}0}{\color[rgb]{0., 0., 0.}0}{\color[rgb]{0., 0., 0.}0}_{3} & \cdots {\color[rgb]{0., 0.187639, 0.57735}0}{\color[rgb]{0., 0., 0.}0}{\color[rgb]{0., 0., 0.}0}{\color[rgb]{0., 0., 0.}0}_{3} \\
M_{37}^{\Td^\#} & \cdots {\color[rgb]{0., 0.182352, 0.561084}0}{\color[rgb]{0.194365, 0.971825, 0.}0}{\color[rgb]{0.912871, 0.616188, 0.}0}{\color[rgb]{0.707107, 0.212132, 0.}0}_{3} & \cdots {\color[rgb]{0., 0.182352, 0.561084}1}{\color[rgb]{0.194365, 0.971825, 0.}2}{\color[rgb]{0.912871, 0.616188, 0.}2}{\color[rgb]{0.707107, 0.212132, 0.}0}_{3} & \cdots {\color[rgb]{0., 0.182352, 0.561084}0}{\color[rgb]{0.194365, 0.971825, 0.}0}{\color[rgb]{0.912871, 0.616188, 0.}0}{\color[rgb]{0.707107, 0.212132, 0.}0}_{3} & \cdots {\color[rgb]{0., 0.182352, 0.561084}1}{\color[rgb]{0.194365, 0.971825, 0.}0}{\color[rgb]{0.912871, 0.616188, 0.}1}{\color[rgb]{0.707107, 0.212132, 0.}0}_{3} & \cdots {\color[rgb]{0., 0.182352, 0.561084}0}{\color[rgb]{0.194365, 0.971825, 0.}1}{\color[rgb]{0.912871, 0.616188, 0.}0}{\color[rgb]{0.707107, 0.212132, 0.}0}_{3} & \cdots {\color[rgb]{0., 0.182352, 0.561084}0}{\color[rgb]{0.194365, 0.971825, 0.}2}{\color[rgb]{0.912871, 0.616188, 0.}0}{\color[rgb]{0.707107, 0.212132, 0.}0}_{3} & \cdots {\color[rgb]{0., 0.182352, 0.561084}1}{\color[rgb]{0.194365, 0.971825, 0.}0}{\color[rgb]{0.912871, 0.616188, 0.}0}{\color[rgb]{0.707107, 0.212132, 0.}0}_{3} \\
M_{38}^{\Td^\#} & \cdots {\color[rgb]{0., 0.176908, 0.544331}2}{\color[rgb]{0.188562, 0.942809, 0.}1}{\color[rgb]{0.816497, 0.551135, 0.}0}{\color[rgb]{0., 0., 0.}2}_{3} & \cdots {\color[rgb]{0., 0.176908, 0.544331}1}{\color[rgb]{0.188562, 0.942809, 0.}1}{\color[rgb]{0.816497, 0.551135, 0.}1}{\color[rgb]{0., 0., 0.}0}_{3} & \cdots {\color[rgb]{0., 0.176908, 0.544331}1}{\color[rgb]{0.188562, 0.942809, 0.}1}{\color[rgb]{0.816497, 0.551135, 0.}2}{\color[rgb]{0., 0., 0.}0}_{3} & \cdots {\color[rgb]{0., 0.176908, 0.544331}1}{\color[rgb]{0.188562, 0.942809, 0.}2}{\color[rgb]{0.816497, 0.551135, 0.}0}{\color[rgb]{0., 0., 0.}0}_{3} & \cdots {\color[rgb]{0., 0.176908, 0.544331}2}{\color[rgb]{0.188562, 0.942809, 0.}0}{\color[rgb]{0.816497, 0.551135, 0.}0}{\color[rgb]{0., 0., 0.}0}_{3} & \cdots {\color[rgb]{0., 0.176908, 0.544331}1}{\color[rgb]{0.188562, 0.942809, 0.}0}{\color[rgb]{0.816497, 0.551135, 0.}0}{\color[rgb]{0., 0., 0.}0}_{3} & \cdots {\color[rgb]{0., 0.176908, 0.544331}1}{\color[rgb]{0.188562, 0.942809, 0.}1}{\color[rgb]{0.816497, 0.551135, 0.}0}{\color[rgb]{0., 0., 0.}0}_{3}
\end{array}
$

%% file: witten.3.data.tex
\[
\begin{array}{r|cccccc}
M_1^{\Wit^\#} & t^6 & t^5 & t^4 & t^3 & t^2 & t^1 \\
\hline
q^1 & \hdots 2\textcolor{magenta}{2}\textcolor{orange}{20}_3 & \hdots 2\textcolor{magenta}{2}\textcolor{orange}{20}_3 & \hdots 2\textcolor{magenta}{2}\textcolor{orange}{20}_3 & \hdots 2\textcolor{magenta}{2}\textcolor{orange}{20}_3 & \hdots 2\textcolor{magenta}{2}\textcolor{orange}{20}_3 & \hdots 2\textcolor{magenta}{2}\textcolor{orange}{10}_3 \\
q^2 & \hdots 2\textcolor{magenta}{0}\textcolor{orange}{10}_3 & \hdots 2\textcolor{magenta}{0}\textcolor{orange}{20}_3 & \hdots 2\textcolor{magenta}{1}\textcolor{orange}{00}_3 & \hdots 2\textcolor{magenta}{1}\textcolor{orange}{10}_3 & \hdots 2\textcolor{magenta}{2}\textcolor{orange}{00}_3 & \hdots 2\textcolor{magenta}{1}\textcolor{orange}{00}_3 \\
q^3 & \hdots 2\textcolor{magenta}{2}\textcolor{orange}{20}_3 & \hdots 2\textcolor{magenta}{2}\textcolor{orange}{20}_3 & \hdots 2\textcolor{magenta}{2}\textcolor{orange}{20}_3 & \hdots 2\textcolor{magenta}{2}\textcolor{orange}{20}_3 & \hdots 2\textcolor{magenta}{2}\textcolor{orange}{20}_3 & \hdots 2\textcolor{magenta}{2}\textcolor{orange}{10}_3 \\
q^4 & \hdots 1\textcolor{magenta}{2}\textcolor{orange}{10}_3 & \hdots 2\textcolor{magenta}{0}\textcolor{orange}{00}_3 & \hdots 1\textcolor{magenta}{1}\textcolor{orange}{20}_3 & \hdots 2\textcolor{magenta}{2}\textcolor{orange}{10}_3 & \hdots 2\textcolor{magenta}{0}\textcolor{orange}{20}_3 & \hdots 1\textcolor{magenta}{1}\textcolor{orange}{10}_3 \\
q^5 & \hdots 0\textcolor{magenta}{1}\textcolor{orange}{20}_3 & \hdots 0\textcolor{magenta}{2}\textcolor{orange}{10}_3 & \hdots 1\textcolor{magenta}{2}\textcolor{orange}{00}_3 & \hdots 0\textcolor{magenta}{2}\textcolor{orange}{20}_3 & \hdots 2\textcolor{magenta}{1}\textcolor{orange}{00}_3 & \hdots 1\textcolor{magenta}{2}\textcolor{orange}{00}_3 \\
q^6 & \hdots 2\textcolor{magenta}{0}\textcolor{orange}{10}_3 & \hdots 2\textcolor{magenta}{0}\textcolor{orange}{20}_3 & \hdots 2\textcolor{magenta}{1}\textcolor{orange}{00}_3 & \hdots 2\textcolor{magenta}{1}\textcolor{orange}{10}_3 & \hdots 2\textcolor{magenta}{2}\textcolor{orange}{00}_3 & \hdots 2\textcolor{magenta}{1}\textcolor{orange}{00}_3 \\
q^7 & \hdots 0\textcolor{magenta}{0}\textcolor{orange}{00}_3 & \hdots 0\textcolor{magenta}{0}\textcolor{orange}{20}_3 & \hdots 1\textcolor{magenta}{1}\textcolor{orange}{10}_3 & \hdots 1\textcolor{magenta}{2}\textcolor{orange}{00}_3 & \hdots 2\textcolor{magenta}{0}\textcolor{orange}{10}_3 & \hdots 1\textcolor{magenta}{0}\textcolor{orange}{20}_3 \\
q^8 & \hdots 0\textcolor{magenta}{0}\textcolor{orange}{20}_3 & \hdots 0\textcolor{magenta}{2}\textcolor{orange}{10}_3 & \hdots 2\textcolor{magenta}{2}\textcolor{orange}{00}_3 & \hdots 1\textcolor{magenta}{0}\textcolor{orange}{20}_3 & \hdots 1\textcolor{magenta}{1}\textcolor{orange}{00}_3 & \hdots 2\textcolor{magenta}{2}\textcolor{orange}{00}_3 \\
q^9 & \hdots 2\textcolor{magenta}{2}\textcolor{orange}{20}_3 & \hdots 2\textcolor{magenta}{2}\textcolor{orange}{20}_3 & \hdots 2\textcolor{magenta}{2}\textcolor{orange}{20}_3 & \hdots 2\textcolor{magenta}{2}\textcolor{orange}{20}_3 & \hdots 2\textcolor{magenta}{2}\textcolor{orange}{20}_3 & \hdots 2\textcolor{magenta}{2}\textcolor{orange}{10}_3
\end{array}
\]
\[
\begin{array}{r|cccccc}
M_7^{\Wit^\#} & t^6 & t^5 & t^4 & t^3 & t^2 & t^1 \\
\hline
q^1 & \hdots 20\textcolor{orange}{20}_3 & \hdots 20\textcolor{orange}{20}_3 & \hdots 20\textcolor{orange}{20}_3 & \hdots 20\textcolor{orange}{20}_3 & \hdots 20\textcolor{orange}{20}_3 & \hdots 11\textcolor{orange}{10}_3 \\
q^2 & \hdots 12\textcolor{orange}{10}_3 & \hdots 01\textcolor{orange}{20}_3 & \hdots 21\textcolor{orange}{00}_3 & \hdots 10\textcolor{orange}{10}_3 & \hdots 12\textcolor{orange}{00}_3 & \hdots 01\textcolor{orange}{00}_3 \\
q^3 & \hdots 20\textcolor{orange}{20}_3 & \hdots 20\textcolor{orange}{20}_3 & \hdots 20\textcolor{orange}{20}_3 & \hdots 20\textcolor{orange}{20}_3 & \hdots 20\textcolor{orange}{20}_3 & \hdots 11\textcolor{orange}{10}_3 \\
q^4 & \hdots 21\textcolor{orange}{10}_3 & \hdots 10\textcolor{orange}{00}_3 & \hdots 02\textcolor{orange}{20}_3 & \hdots 01\textcolor{orange}{10}_3 & \hdots 21\textcolor{orange}{20}_3 & \hdots 20\textcolor{orange}{10}_3 \\
q^5 & \hdots 02\textcolor{orange}{20}_3 & \hdots 01\textcolor{orange}{10}_3 & \hdots 12\textcolor{orange}{00}_3 & \hdots 00\textcolor{orange}{20}_3 & \hdots 01\textcolor{orange}{00}_3 & \hdots 02\textcolor{orange}{00}_3 \\
q^6 & \hdots 12\textcolor{orange}{10}_3 & \hdots 01\textcolor{orange}{20}_3 & \hdots 21\textcolor{orange}{00}_3 & \hdots 10\textcolor{orange}{10}_3 & \hdots 12\textcolor{orange}{00}_3 & \hdots 01\textcolor{orange}{00}_3 \\
q^7 & \hdots 10\textcolor{orange}{00}_3 & \hdots 11\textcolor{orange}{20}_3 & \hdots 00\textcolor{orange}{10}_3 & \hdots 02\textcolor{orange}{00}_3 & \hdots 12\textcolor{orange}{10}_3 & \hdots 01\textcolor{orange}{20}_3 \\
q^8 & \hdots 01\textcolor{orange}{20}_3 & \hdots 11\textcolor{orange}{10}_3 & \hdots 22\textcolor{orange}{00}_3 & \hdots 11\textcolor{orange}{20}_3 & \hdots 21\textcolor{orange}{00}_3 & \hdots 12\textcolor{orange}{00}_3 \\
q^9 & \hdots 20\textcolor{orange}{20}_3 & \hdots 20\textcolor{orange}{20}_3 & \hdots 20\textcolor{orange}{20}_3 & \hdots 20\textcolor{orange}{20}_3 & \hdots 20\textcolor{orange}{20}_3 & \hdots 11\textcolor{orange}{10}_3
\end{array}
\]
\[
\begin{array}{r|cccccc}
M_{19}^{\Wit^\#} & t^6 & t^5 & t^4 & t^3 & t^2 & t^1 \\
\hline
q^1 & \hdots 0\textcolor{magenta}{2}\textcolor{orange}{20}_3 & \hdots 0\textcolor{magenta}{2}\textcolor{orange}{20}_3 & \hdots 0\textcolor{magenta}{2}\textcolor{orange}{20}_3 & \hdots 0\textcolor{magenta}{2}\textcolor{orange}{20}_3 & \hdots 0\textcolor{magenta}{2}\textcolor{orange}{20}_3 & \hdots 1\textcolor{magenta}{2}\textcolor{orange}{10}_3 \\
q^2 & \hdots 1\textcolor{magenta}{0}\textcolor{orange}{10}_3 & \hdots 0\textcolor{magenta}{0}\textcolor{orange}{20}_3 & \hdots 2\textcolor{magenta}{1}\textcolor{orange}{00}_3 & \hdots 1\textcolor{magenta}{1}\textcolor{orange}{10}_3 & \hdots 2\textcolor{magenta}{2}\textcolor{orange}{00}_3 & \hdots 2\textcolor{magenta}{1}\textcolor{orange}{00}_3 \\
q^3 & \hdots 0\textcolor{magenta}{2}\textcolor{orange}{20}_3 & \hdots 0\textcolor{magenta}{2}\textcolor{orange}{20}_3 & \hdots 0\textcolor{magenta}{2}\textcolor{orange}{20}_3 & \hdots 0\textcolor{magenta}{2}\textcolor{orange}{20}_3 & \hdots 0\textcolor{magenta}{2}\textcolor{orange}{20}_3 & \hdots 1\textcolor{magenta}{2}\textcolor{orange}{10}_3 \\
q^4 & \hdots 0\textcolor{magenta}{2}\textcolor{orange}{10}_3 & \hdots 2\textcolor{magenta}{0}\textcolor{orange}{00}_3 & \hdots 2\textcolor{magenta}{1}\textcolor{orange}{20}_3 & \hdots 1\textcolor{magenta}{2}\textcolor{orange}{10}_3 & \hdots 0\textcolor{magenta}{0}\textcolor{orange}{20}_3 & \hdots 0\textcolor{magenta}{1}\textcolor{orange}{10}_3 \\
q^5 & \hdots 1\textcolor{magenta}{1}\textcolor{orange}{20}_3 & \hdots 2\textcolor{magenta}{2}\textcolor{orange}{10}_3 & \hdots 1\textcolor{magenta}{2}\textcolor{orange}{00}_3 & \hdots 1\textcolor{magenta}{2}\textcolor{orange}{20}_3 & \hdots 2\textcolor{magenta}{1}\textcolor{orange}{00}_3 & \hdots 1\textcolor{magenta}{2}\textcolor{orange}{00}_3 \\
q^6 & \hdots 1\textcolor{magenta}{0}\textcolor{orange}{10}_3 & \hdots 0\textcolor{magenta}{0}\textcolor{orange}{20}_3 & \hdots 2\textcolor{magenta}{1}\textcolor{orange}{00}_3 & \hdots 1\textcolor{magenta}{1}\textcolor{orange}{10}_3 & \hdots 2\textcolor{magenta}{2}\textcolor{orange}{00}_3 & \hdots 2\textcolor{magenta}{1}\textcolor{orange}{00}_3 \\
q^7 & \hdots 0\textcolor{magenta}{0}\textcolor{orange}{00}_3 & \hdots 1\textcolor{magenta}{0}\textcolor{orange}{20}_3 & \hdots 0\textcolor{magenta}{1}\textcolor{orange}{10}_3 & \hdots 1\textcolor{magenta}{2}\textcolor{orange}{00}_3 & \hdots 1\textcolor{magenta}{0}\textcolor{orange}{10}_3 & \hdots 2\textcolor{magenta}{0}\textcolor{orange}{20}_3 \\
q^8 & \hdots 1\textcolor{magenta}{0}\textcolor{orange}{20}_3 & \hdots 2\textcolor{magenta}{2}\textcolor{orange}{10}_3 & \hdots 2\textcolor{magenta}{2}\textcolor{orange}{00}_3 & \hdots 2\textcolor{magenta}{0}\textcolor{orange}{20}_3 & \hdots 1\textcolor{magenta}{1}\textcolor{orange}{00}_3 & \hdots 2\textcolor{magenta}{2}\textcolor{orange}{00}_3 \\
q^9 & \hdots 0\textcolor{magenta}{2}\textcolor{orange}{20}_3 & \hdots 0\textcolor{magenta}{2}\textcolor{orange}{20}_3 & \hdots 0\textcolor{magenta}{2}\textcolor{orange}{20}_3 & \hdots 0\textcolor{magenta}{2}\textcolor{orange}{20}_3 & \hdots 0\textcolor{magenta}{2}\textcolor{orange}{20}_3 & \hdots 1\textcolor{magenta}{2}\textcolor{orange}{10}_3
\end{array}
\]

%% file: witten.2.data.tex
\[
\begin{array}{r|cccccc}
M_1^{\Wit^\#} & t^6 & t^5 & t^4 & t^3 & t^2 & t^1 \\
\hline
q^{1} & \hdots 1\textcolor{cyan}{1}\textcolor{purple}{1}\textcolor{orange}{110}_2 & \hdots 1\textcolor{cyan}{1}\textcolor{purple}{1}\textcolor{orange}{110}_2 & \hdots 1\textcolor{cyan}{1}\textcolor{purple}{1}\textcolor{orange}{110}_2 & \hdots 1\textcolor{cyan}{1}\textcolor{purple}{1}\textcolor{orange}{110}_2 & \hdots 1\textcolor{cyan}{1}\textcolor{purple}{1}\textcolor{orange}{110}_2 & \hdots 1\textcolor{cyan}{1}\textcolor{purple}{1}\textcolor{orange}{100}_2 \\
q^{2} & \hdots 1\textcolor{cyan}{1}\textcolor{purple}{1}\textcolor{orange}{110}_2 & \hdots 1\textcolor{cyan}{1}\textcolor{purple}{1}\textcolor{orange}{110}_2 & \hdots 1\textcolor{cyan}{1}\textcolor{purple}{1}\textcolor{orange}{110}_2 & \hdots 1\textcolor{cyan}{1}\textcolor{purple}{1}\textcolor{orange}{110}_2 & \hdots 1\textcolor{cyan}{1}\textcolor{purple}{1}\textcolor{orange}{110}_2 & \hdots 1\textcolor{cyan}{1}\textcolor{purple}{1}\textcolor{orange}{100}_2 \\
q^{3} & \hdots 0\textcolor{cyan}{0}\textcolor{purple}{0}\textcolor{orange}{110}_2 & \hdots 0\textcolor{cyan}{1}\textcolor{purple}{0}\textcolor{orange}{100}_2 & \hdots 1\textcolor{cyan}{0}\textcolor{purple}{0}\textcolor{orange}{000}_2 & \hdots 1\textcolor{cyan}{0}\textcolor{purple}{1}\textcolor{orange}{000}_2 & \hdots 1\textcolor{cyan}{1}\textcolor{purple}{1}\textcolor{orange}{000}_2 & \hdots 1\textcolor{cyan}{1}\textcolor{purple}{0}\textcolor{orange}{000}_2 \\
q^{4} & \hdots 1\textcolor{cyan}{1}\textcolor{purple}{1}\textcolor{orange}{110}_2 & \hdots 1\textcolor{cyan}{1}\textcolor{purple}{1}\textcolor{orange}{110}_2 & \hdots 1\textcolor{cyan}{1}\textcolor{purple}{1}\textcolor{orange}{110}_2 & \hdots 1\textcolor{cyan}{1}\textcolor{purple}{1}\textcolor{orange}{110}_2 & \hdots 1\textcolor{cyan}{1}\textcolor{purple}{1}\textcolor{orange}{110}_2 & \hdots 1\textcolor{cyan}{1}\textcolor{purple}{1}\textcolor{orange}{100}_2 \\
q^{5} & \hdots 0\textcolor{cyan}{1}\textcolor{purple}{1}\textcolor{orange}{010}_2 & \hdots 0\textcolor{cyan}{0}\textcolor{purple}{0}\textcolor{orange}{000}_2 & \hdots 1\textcolor{cyan}{1}\textcolor{purple}{1}\textcolor{orange}{100}_2 & \hdots 1\textcolor{cyan}{0}\textcolor{purple}{0}\textcolor{orange}{100}_2 & \hdots 1\textcolor{cyan}{1}\textcolor{purple}{0}\textcolor{orange}{100}_2 & \hdots 1\textcolor{cyan}{0}\textcolor{purple}{1}\textcolor{orange}{000}_2 \\
q^{6} & \hdots 0\textcolor{cyan}{0}\textcolor{purple}{0}\textcolor{orange}{110}_2 & \hdots 0\textcolor{cyan}{1}\textcolor{purple}{0}\textcolor{orange}{100}_2 & \hdots 1\textcolor{cyan}{0}\textcolor{purple}{0}\textcolor{orange}{000}_2 & \hdots 1\textcolor{cyan}{0}\textcolor{purple}{1}\textcolor{orange}{000}_2 & \hdots 1\textcolor{cyan}{1}\textcolor{purple}{1}\textcolor{orange}{000}_2 & \hdots 1\textcolor{cyan}{1}\textcolor{purple}{0}\textcolor{orange}{000}_2
\end{array}
\]
\[
\begin{array}{r|cccccc}
M_3^{\Wit^\#} & t^6 & t^5 & t^4 & t^3 & t^2 & t^1 \\
\hline
q^{1} & \hdots 100\textcolor{orange}{110}_2 & \hdots 100\textcolor{orange}{110}_2 & \hdots 100\textcolor{orange}{110}_2 & \hdots 100\textcolor{orange}{110}_2 & \hdots 100\textcolor{orange}{110}_2 & \hdots 001\textcolor{orange}{100}_2 \\
q^{2} & \hdots 100\textcolor{orange}{110}_2 & \hdots 100\textcolor{orange}{110}_2 & \hdots 100\textcolor{orange}{110}_2 & \hdots 100\textcolor{orange}{110}_2 & \hdots 100\textcolor{orange}{110}_2 & \hdots 001\textcolor{orange}{100}_2 \\
q^{3} & \hdots 001\textcolor{orange}{110}_2 & \hdots 110\textcolor{orange}{100}_2 & \hdots 110\textcolor{orange}{000}_2 & \hdots 011\textcolor{orange}{000}_2 & \hdots 101\textcolor{orange}{000}_2 & \hdots 010\textcolor{orange}{000}_2 \\
q^{4} & \hdots 100\textcolor{orange}{110}_2 & \hdots 100\textcolor{orange}{110}_2 & \hdots 100\textcolor{orange}{110}_2 & \hdots 100\textcolor{orange}{110}_2 & \hdots 100\textcolor{orange}{110}_2 & \hdots 001\textcolor{orange}{100}_2 \\
q^{5} & \hdots 110\textcolor{orange}{010}_2 & \hdots 110\textcolor{orange}{000}_2 & \hdots 011\textcolor{orange}{100}_2 & \hdots 100\textcolor{orange}{100}_2 & \hdots 110\textcolor{orange}{100}_2 & \hdots 101\textcolor{orange}{000}_2 \\
q^{6} & \hdots 001\textcolor{orange}{110}_2 & \hdots 110\textcolor{orange}{100}_2 & \hdots 110\textcolor{orange}{000}_2 & \hdots 011\textcolor{orange}{000}_2 & \hdots 101\textcolor{orange}{000}_2 & \hdots 010\textcolor{orange}{000}_2
\end{array}
\]
\[
\begin{array}{r|cccccc}
M_5^{\Wit^\#} & t^6 & t^5 & t^4 & t^3 & t^2 & t^1 \\
\hline
q^{1} & \hdots 00\textcolor{purple}{1}\textcolor{orange}{110}_2 & \hdots 00\textcolor{purple}{1}\textcolor{orange}{110}_2 & \hdots 00\textcolor{purple}{1}\textcolor{orange}{110}_2 & \hdots 00\textcolor{purple}{1}\textcolor{orange}{110}_2 & \hdots 00\textcolor{purple}{1}\textcolor{orange}{110}_2 & \hdots 01\textcolor{purple}{1}\textcolor{orange}{100}_2 \\
q^{2} & \hdots 00\textcolor{purple}{1}\textcolor{orange}{110}_2 & \hdots 00\textcolor{purple}{1}\textcolor{orange}{110}_2 & \hdots 00\textcolor{purple}{1}\textcolor{orange}{110}_2 & \hdots 00\textcolor{purple}{1}\textcolor{orange}{110}_2 & \hdots 00\textcolor{purple}{1}\textcolor{orange}{110}_2 & \hdots 01\textcolor{purple}{1}\textcolor{orange}{100}_2 \\
q^{3} & \hdots 01\textcolor{purple}{0}\textcolor{orange}{110}_2 & \hdots 01\textcolor{purple}{0}\textcolor{orange}{100}_2 & \hdots 00\textcolor{purple}{0}\textcolor{orange}{000}_2 & \hdots 00\textcolor{purple}{1}\textcolor{orange}{000}_2 & \hdots 01\textcolor{purple}{1}\textcolor{orange}{000}_2 & \hdots 11\textcolor{purple}{0}\textcolor{orange}{000}_2 \\
q^{4} & \hdots 00\textcolor{purple}{1}\textcolor{orange}{110}_2 & \hdots 00\textcolor{purple}{1}\textcolor{orange}{110}_2 & \hdots 00\textcolor{purple}{1}\textcolor{orange}{110}_2 & \hdots 00\textcolor{purple}{1}\textcolor{orange}{110}_2 & \hdots 00\textcolor{purple}{1}\textcolor{orange}{110}_2 & \hdots 01\textcolor{purple}{1}\textcolor{orange}{100}_2 \\
q^{5} & \hdots 00\textcolor{purple}{1}\textcolor{orange}{010}_2 & \hdots 10\textcolor{purple}{0}\textcolor{orange}{000}_2 & \hdots 11\textcolor{purple}{1}\textcolor{orange}{100}_2 & \hdots 10\textcolor{purple}{0}\textcolor{orange}{100}_2 & \hdots 11\textcolor{purple}{0}\textcolor{orange}{100}_2 & \hdots 10\textcolor{purple}{1}\textcolor{orange}{000}_2 \\
q^{6} & \hdots 01\textcolor{purple}{0}\textcolor{orange}{110}_2 & \hdots 01\textcolor{purple}{0}\textcolor{orange}{100}_2 & \hdots 00\textcolor{purple}{0}\textcolor{orange}{000}_2 & \hdots 00\textcolor{purple}{1}\textcolor{orange}{000}_2 & \hdots 01\textcolor{purple}{1}\textcolor{orange}{000}_2 & \hdots 11\textcolor{purple}{0}\textcolor{orange}{000}_2
\end{array}
\]
\[
\begin{array}{r|cccccc}
M_9^{\Wit^\#} & t^6 & t^5 & t^4 & t^3 & t^2 & t^1 \\
\hline
q^{1} & \hdots 0\textcolor{cyan}{1}\textcolor{purple}{1}\textcolor{orange}{110}_2 & \hdots 0\textcolor{cyan}{1}\textcolor{purple}{1}\textcolor{orange}{110}_2 & \hdots 0\textcolor{cyan}{1}\textcolor{purple}{1}\textcolor{orange}{110}_2 & \hdots 0\textcolor{cyan}{1}\textcolor{purple}{1}\textcolor{orange}{110}_2 & \hdots 0\textcolor{cyan}{1}\textcolor{purple}{1}\textcolor{orange}{110}_2 & \hdots 1\textcolor{cyan}{1}\textcolor{purple}{1}\textcolor{orange}{100}_2 \\
q^{2} & \hdots 0\textcolor{cyan}{1}\textcolor{purple}{1}\textcolor{orange}{110}_2 & \hdots 0\textcolor{cyan}{1}\textcolor{purple}{1}\textcolor{orange}{110}_2 & \hdots 0\textcolor{cyan}{1}\textcolor{purple}{1}\textcolor{orange}{110}_2 & \hdots 0\textcolor{cyan}{1}\textcolor{purple}{1}\textcolor{orange}{110}_2 & \hdots 0\textcolor{cyan}{1}\textcolor{purple}{1}\textcolor{orange}{110}_2 & \hdots 1\textcolor{cyan}{1}\textcolor{purple}{1}\textcolor{orange}{100}_2 \\
q^{3} & \hdots 1\textcolor{cyan}{0}\textcolor{purple}{0}\textcolor{orange}{110}_2 & \hdots 0\textcolor{cyan}{1}\textcolor{purple}{0}\textcolor{orange}{100}_2 & \hdots 1\textcolor{cyan}{0}\textcolor{purple}{0}\textcolor{orange}{000}_2 & \hdots 1\textcolor{cyan}{0}\textcolor{purple}{1}\textcolor{orange}{000}_2 & \hdots 1\textcolor{cyan}{1}\textcolor{purple}{1}\textcolor{orange}{000}_2 & \hdots 1\textcolor{cyan}{1}\textcolor{purple}{0}\textcolor{orange}{000}_2 \\
q^{4} & \hdots 0\textcolor{cyan}{1}\textcolor{purple}{1}\textcolor{orange}{110}_2 & \hdots 0\textcolor{cyan}{1}\textcolor{purple}{1}\textcolor{orange}{110}_2 & \hdots 0\textcolor{cyan}{1}\textcolor{purple}{1}\textcolor{orange}{110}_2 & \hdots 0\textcolor{cyan}{1}\textcolor{purple}{1}\textcolor{orange}{110}_2 & \hdots 0\textcolor{cyan}{1}\textcolor{purple}{1}\textcolor{orange}{110}_2 & \hdots 1\textcolor{cyan}{1}\textcolor{purple}{1}\textcolor{orange}{100}_2 \\
q^{5} & \hdots 1\textcolor{cyan}{1}\textcolor{purple}{1}\textcolor{orange}{010}_2 & \hdots 0\textcolor{cyan}{0}\textcolor{purple}{0}\textcolor{orange}{000}_2 & \hdots 1\textcolor{cyan}{1}\textcolor{purple}{1}\textcolor{orange}{100}_2 & \hdots 1\textcolor{cyan}{0}\textcolor{purple}{0}\textcolor{orange}{100}_2 & \hdots 1\textcolor{cyan}{1}\textcolor{purple}{0}\textcolor{orange}{100}_2 & \hdots 1\textcolor{cyan}{0}\textcolor{purple}{1}\textcolor{orange}{000}_2 \\
q^{6} & \hdots 1\textcolor{cyan}{0}\textcolor{purple}{0}\textcolor{orange}{110}_2 & \hdots 0\textcolor{cyan}{1}\textcolor{purple}{0}\textcolor{orange}{100}_2 & \hdots 1\textcolor{cyan}{0}\textcolor{purple}{0}\textcolor{orange}{000}_2 & \hdots 1\textcolor{cyan}{0}\textcolor{purple}{1}\textcolor{orange}{000}_2 & \hdots 1\textcolor{cyan}{1}\textcolor{purple}{1}\textcolor{orange}{000}_2 & \hdots 1\textcolor{cyan}{1}\textcolor{purple}{0}\textcolor{orange}{000}_2
\end{array}
\]